\documentclass[a4paper]{article}
\usepackage{amsmath,amssymb,amsfonts,amsthm,mathtools,xfrac,enumitem}
\usepackage{graphicx}
\usepackage{color}  
\usepackage[latin1]{inputenc}
\usepackage[T1]{fontenc}
\usepackage{amscd} 
\usepackage{hyperref} 

\numberwithin{equation}{section}  
\renewcommand{\a }{\alpha }
\renewcommand{\b }{\beta }
\renewcommand{\d}{\delta }
\renewcommand{\l}{\lambda }
\newcommand{\N}{\mathbb{N}}
\newcommand{\R}{\mathbb{R}}

\makeatletter
\renewcommand*\env@matrix[1][*\c@MaxMatrixCols c]{%
  \hskip -\arraycolsep
  \let\@ifnextchar\new@ifnextchar
  \array{#1}} 
\makeatother

\newtheorem{remark}{Remark}[section]
\newtheorem{lemma}{Lemma}[section]
\newtheorem{proposition}{Proposition}[section] 
\newtheorem{thm}{Theorem}
\newtheorem{definition}{Definition}[section]

\title{Prescribing Morse scalar curvatures: blow-up analysis} 

\author{Andrea Malchiodi and Martin Mayer \\ \\
To appear on International Mathematical Research Notes\\ \ \\
\small Scuola Normale Superiore,
Piazza dei Cavalieri 7,
50126 Pisa, ITALY \\ 
\small andrea.malchiodi@sns.it, martin.mayer@sns.it}
  
\normalsize

\begin{document}  
 
\maketitle

 
 
{\footnotesize 
\begin{abstract} 
  
\noindent We study finite-energy blow-ups 
for  prescribed Morse scalar curvatures in both the subcritical and the critical regime. After general considerations on Palais-Smale sequences we determine precise blow up rates for subcritical solutions: in particular the possibility of tower bubbles is excluded in all dimensions. 
In subsequent papers  we aim to establish the sharpness of this result, proving a converse existence statement, together with a one to one correspondence of blowing-up subcritical solutions and {\em critical 
points at infinity}.
This analysis will be then applied 
to deduce new existence results for the geometric problem.

\

\noindent{\it Key Words:}
Conformal geometry, sub-critical approximation, blow-up analysis. 
\end{abstract}
}

\tableofcontents

  \section{Introduction}

  \noindent
  The problem of  prescribing the scalar curvature of a manifold conformally
 has a long history, starting from \cite{kw}, see also \cite{kw1}, \cite{kw2}. In case of the round sphere, this is known as {\em Nirenberg's problem}.

  Given a closed manifold $(M,g_{0})$ of dimension $n \geq 3$  and a conformal 
  metric $g = u^{\frac{4}{n-2}} g_{0}$ for a positive function $u > 0$ on $M$,  the conformal change of the scalar curvature 
 is given by 
 $$
 R_{g_u} u^{\frac{n+2}{n-2}} = L_{g_{0}}u,
 $$
  where by definition 
  $$
  L_{g_0} u = - c_{n} \Delta_{g_{0}} u + R_{g_{0}} u,
 \quad 
  c_n = \frac{4(n-1)}{n-2}
  $$
is the \emph{conformal Laplacian}, while  $\Delta_{g_{0}}$ is the Laplace-Beltrami operator with respect to $g_0$. Thus,  in order to prescribe a function $K$ on $M$ as the scalar curvature with respect to $g$,
one needs to solve 
\begin{equation}\label{eq:scin}
L_{g_0} u = K u^{\frac{n+2}{n-2}}, \quad u > 0
\end{equation}
pointwise on $M$, see \cite{aul}. The exponent on the right-hand side is critical with respect to Sobolev's embedding, 
  which makes the problem particularly challenging. In contrast to the {\em Yamabe problem},
  which amounts to finding a constant scalar curvature metric, 
 for $K$ varying on $M$ there are 
  obstructions to the existence for  
  \eqref{eq:scin}. 
  For example Kazdan and Warner proved in \cite{kw}  that on the round sphere 
 $(S^n,g_{S^n})$ 
  every solution $u$ of \eqref{eq:scin} must satisfy  
  \begin{equation*}
  \int_{S^n}  \langle \nabla K, \nabla f \rangle_{g_{S^n}} u^{\frac{2n}{n-2}} \, d\mu_{g_{S^{n}}} = 0
  \end{equation*}
  for any restriction $f$ to $S^n$ of an affine function on $\R^{n+1}$. 
  In particular, since $u$ is positive, a necessary condition for the existence of 
  solutions is that the function $\langle \nabla K, \nabla f \rangle_{g_{S^n}}$ changes sign.

  One of the first answers to Nirenberg's problem was given by J. Moser in \cite{mo}
  for two dimensions, where the counterpart of \eqref{eq:scin} has an exponential form. He proved 
  that for $K$ being an even function on $S^2$
 a  solution always exists. A related result was given by J. Escobar and R.
  Schoen in \cite{es}, showing existence of solutions when $K$
  is invariant under some group $G$ acting without fixed points, under  
  suitable \emph{flatness assumptions} of order $n-2$. In the same paper some results 
  were also found for non-spherical manifolds using positivity of the mass. 
  Other sufficient conditions for the existence in case of $G$-invariant
  functions were  given by E. Hebey and M. Vaugon in \cite{hv}, \cite{hv93}, allowing the 
  possibility of fixed points.

Other existence results were obtained by A. Chang and P. Yang, see \cite{[ACPY1]}, \cite{[ACPY2]}, for the case  $n = 2$ without requiring any symmetry of $K$. One condition, for which they obtained existence, is the following. First they assumed, that  $ K $ is a positive Morse function satisfying
\begin{equation}\label{eq:nd}
\{ \nabla K  = 0 \} \cap \{ \Delta K = 0 \} = \emptyset,
\end{equation}
where here and in the following 
$\nabla = \nabla_{g_0}$ and $\Delta = \Delta_{g_0}$, cf. \eqref{conformal_normal_metric} and below. Secondly, they supposed that $K$ possesses $p$ local maxima and $q$ saddle points with negative Laplacian and $p \neq q + 1$. 
The latter condition was used to prove the result via a Leray-Schauder degree-theoretical argument. 
In the same papers other results were given, requiring conditions only at some prescribed levels of $K$. Typically $K$ must possess two maxima $x_0$ and $x_1$, $K(x_1) \leq K(x_0)$, which are connected by some path $x(t)$ for which 
\begin{equation*}
x \mbox{ saddle point for } K\; \wedge \; \inf_t K(x(t))\! \leq K(x)\! <\!
K(x_0) \; \Rightarrow \; \Delta K(x) > 0.
\end{equation*}
Statements of this last kind have been obtained in \cite{[CD]} for $n = 2$ and in \cite{[BI]} for $n \geq 3$. 
Another existence result was given by A. Bahri and J.M. Coron in \cite{bc} for $n = 3$ and a Morse function $K$  satisfying \eqref{eq:nd} and 
\begin{equation}\label{eq:bcin}
\sum_{ x \in \{\nabla K = 0\} \cap \{ \Delta K < 0\}}(-1)^{m(x,K)} \neq -1. 
\end{equation}
  Here $m(x,K)$ denotes the Morse index of $K$ at $x$, cf. also \cite{[ACGPY]}. The result of Bahri and
  Coron, which relies on a topological argument, has been extended in several
  directions. 
  
An extension of condition \eqref{eq:bcin}, based
  on Morse's inequalities, was given by Schoen and Zhang in \cite{[SZ]}
  for the case $n = 3$. For a Morse function $K$ 
  satisfying \eqref{eq:nd} and setting 
  $$c_q = \sharp\{ x \in M\; : \; \nabla K (x) = 0,\;  \Delta K(x) < 0\; \text{ and } \;
  m(K,x) = 3 - q \}$$
  they required that  either 
  $c_0 - c_1 + c_2 \neq 1$ or $c_0 - c_1 > 1$.
  Note that the first condition is equivalent to
  (\ref{eq:bcin}) and the second one for $n = 2$  corresponds to the  condition $p + 1 > q$ in \cite{[ACPY1]}.

  Other results of perturbative type and relying on finite-dimensional reductions were given by A. Chang and P. Yang in
  \cite{[ACPY3]} and by A. Ambrosetti, J. Garcia-Azorero and I. Peral in
  \cite{agp}, see also \cite{mal}. The authors considered the case in which $K$ is close to a constant
  and satisfies an analogue of  \eqref{eq:bcin}, i.e.
  \begin{equation*}
  \sum_{ x \in \{\nabla K = 0\} \cap \{ \Delta K < 0\}}(-1)^{m(x,K)} \neq (-1)^n.  
  \end{equation*}

  In \cite{yy} Y.Y. Li proved existence of solutions for every dimension, 
  if the function $K$ near each  critical point has a Morse-type 
  structure, but with a flatness of order $\beta \in (n-2,n)$. His proof relied on a homotopy argument: 
  considering  $K_t  = t \, K + (1 - t)$, $t \in [0,1]$ the author used the degree-counting formula 
  of \cite{[ACPY3]} for $t$ small, and then a refined 
  blow-up analysis of equation \eqref{eq:scin}, when $t$ tends to 
  $1$.  A different degree formula under more general flatness conditions 
  was introduced in \cite{cl4}. 
  Other results obtained by different approaches can also be found in \cite{[BCCHhigh]}, \cite{[BIEG]}, \cite{[DN]}.
  
  \
  
  A useful tool for  the above results is a subcritical approximation of \eqref{eq:scin}, namely 
  \begin{equation}\label{eq:scin-tau}
  - c_{n} \Delta_{g_{0}} u + R_{g_{0}} u = K\, u^{\frac{n+2}{n-2}-\tau}, \quad
  0<\tau \ll 1.
  \end{equation}
  The advantage of \eqref{eq:scin-tau}, compared to \eqref{eq:scin},  is that the lower exponent makes the 
  problem compact, so it is easier to construct solutions. 
 However, the interesting point is passing to the limit of solutions for $\tau \longrightarrow 0$ and in 
  general one expects some of them to diverge with zero weak limit.  The approach in 
  \cite{[ACGPY]}, \cite{[SZ]}, \cite{yy} was to understand in detail the 
  behaviour of blowing-up solutions and then to use degree- or Morse-theoretical arguments 
  to show that some  solutions stay bounded.

  Consider now a Morse function $K$ on the sphere satisfying \eqref{eq:nd}. In dimension $n=3$ or
  under a flatness condition in higher dimensions,
  it turns out that blowing-up solutions to \eqref{eq:scin-tau} develop a single bubble 
  at critical points of $K$ with negative Laplacian. 
  {\em Bubbles} correspond to  solutions of \eqref{eq:scin} on $S^n$ with $K\equiv 1$ and
  were classified in \cite{cgs}, see also \cite{auw}, \cite{tal}, and after proper dilation represent 
  the profiles of diverging solutions, cf.  Section \ref{s:set-up}  for  precise formulas.

  The single-bubble phenomenon  can be qualitatively explained exploiting the variational 
  features of the problem, which admits the Euler-Lagrange energy $J = J_K$ given by 
  $$
  J(u)
  =
  \frac
  {\int_{M} \left( c_{n}\vert \nabla u\vert_{g_{0}}^{2}+R_{g_{0}}u^{2} \right) d\mu_{g_{0}}}
  {(\int Ku^{\frac{2n}{n-2}}d\mu_{g_{0}})^{\frac{n-2}{n}}}, 
  $$
  see also 
  \eqref{eq:JKt} regarding \eqref{eq:scin-tau}. Denote by 
  $\delta_{a,\lambda}$ a bubble centered at $a \in S^n$ with dilation 
  parameter $\lambda$. Then for distinct and fixed points 
  $a_1, a_2$ and $\lambda$ large one has the expansions  
  \begin{equation}\label{eq:interactions-K-b}
  \int_{S^n} K (\delta_{a_{1},\lambda} + \delta_{a_{2},\lambda})^{\frac{2n}{n-2}} d \mu_{g_{S^n}} \simeq 
 K(a_1) + K(a_2) + \frac{c_1}{\l^{n-2}},\;  \int_{S^n} K \d_{a_{i},\l}^{\frac{2n}{n-2}} d \mu_{g_{S^n}}
 \simeq c_2 K(a_{i}) - \frac{c_3}{\l^2} \Delta K (a_{i})
  \end{equation}
  with constants $c_{i}>0$, 
  where $c_1$ depends  on $a_1$ and $a_2$. 
We refer to Section \ref{s:funct-infty}
  for more accurate results. 
  Terms similar to the above ones appear in the expression of $J_{\tau}$. 
 By the latter formulas and for $\l \longrightarrow \infty$ and $n = 3$ the {\em interaction} of the bubbles 
  with $K$ is \emph{dominated} by the mutual interactions among bubbles. This causes multiple bubbles to suppress 
  each other allowing only one blow-up point at a time, which has to be close to at critical 
  points of $K$ with negative Laplacian due to a Pohozaev  identity. 
  
  This analysis was carried over in \cite{yy2} also on $S^4$. In this case the above interactions 
  are of the same order and multiple blow-ups occur. It was also shown there 
  that multiple bubbles cannot accumulate at a single point. Using a terminology from \cite{Sch1}, \cite{s}
  such blow-ups are called {\em isolated simple}. In four dimensions  a different 
  constraint on multiple blow-up points replaces $\Delta K < 0$, depending on the least eigenvalue of a matrix 
  constructed out of  $K$ and the location of the blow-up points, cf.  
  (0.8) in \cite{yy2}. On general four-dimensional manifolds there is an extra term 
  due to the {\em mass} of the manifold leading to similar phenomena, but with modified formulas, see 
  \cite{[BCCH4]}. 
  
  \
  
  The goal of this paper is to investigate the blow-up behaviour in an opposite regime, when 
  the dimension $n\geq 5$ and the function $K$ is  Morse. In this case 
  the second term in \eqref{eq:interactions-K-b} dominates the first one, so it is drastically 
  different from situation of low-dimensions or with flat curvatures.  
  However we can still show that blow-ups are isolated simple, 
  which is important in understanding the Morse-theoretical structure of the energy functional. 
   Here is our main result. 
  
  \medskip
  
  \begin{thm} \label{t:main}
  Let $(M^n,g_0)$, $n \geq 5$ be a closed manifold of positive Yamabe invariant and  $K : M \longrightarrow \R$ a smooth positive Morse function satisfying \eqref{eq:nd}. 
  Then positive sequences of solutions to \eqref{eq:scin-tau} for $\tau_m \searrow 0$ with uniformly bounded 
  $W^{1,2}$-energy and zero weak limit have only isolated simple blow-ups at critical points of $K$ with 
  negative Laplacian. 
  \end{thm} 
  
  \noindent The above theorem follows from Proposition \ref{blow_up_analysis}, where a general characterization of 
  blowing-up Palais-Smale sequences for \eqref{eq:scin-tau} as $\tau \longrightarrow 0$ is given, and from Theorem 
  \ref{lem_top_down_cascade}, where a lower bound on the norm of the gradient of the Euler-Lagrange functional 
  $J_\tau$ for \eqref{eq:scin-tau} is proved, see \eqref{eq:JKt}. 

  \medskip
  
\begin{remark}\label{r:precise} 
Solutions of \eqref{eq:scin-tau} can be found as suitably normalized critical points  of 
the scaling-invariant energy $J_\tau$ in \eqref{eq:JKt}.  For a  sequence of critical points $(u_m)$ 
of $J_{\tau_m}$, with $\tau_{m}$ as in Theorem \ref{t:main}, there exist 
up to  subsequences  $q \in \N$ and \underline{distinct points} $x_1, \dots, x_q \in M$ with $\nabla K(x_j) = 0$ and $\Delta K (x_j) < 0$ 
  such that  
  \begin{equation*}
  \bigg\| u_{m} - \sum_{j=1}^q \alpha_{j,m} \delta_{\lambda_{j,m}, a_{j,m}} \bigg\|_{W^{1,2}(M,g_0)}
  \longrightarrow 0 \quad \text{ as } \quad m \longrightarrow  \infty
  \end{equation*} 
  for some  
  $$
  \alpha_{j,m}=\frac{\Theta}{K(x_{j})^{\frac{n-2}{4}}}+o(1)
  ,\quad  a_{j,m} \longrightarrow x_j
  \quad \text{ and }\quad 
  \lambda_{j,m} \simeq \lambda_{\tau_{m}} = \tau^{- \frac 12}_{m},$$
  where the multiplicative constant $\Theta$ 
reflects the scaling invariance of $J_{\tau_{m}}$, see \eqref{eq:JKt}, and can be fixed for instance by prescribing  the conformal volume, cf. Remark \ref{r:precision}. In Theorem \ref{lem_top_down_cascade} we will show much more precise estimates, that will be crucial for 
  \cite{MM2}.  For example, if $n\geq 6$,  we find
  \begin{equation*}
\lambda_{j,m}=c_{1}\sqrt{\frac{\Delta K(x_{j})}{K(x_{j})\tau}}
,
\quad 
a_{j,m}=c_{2}(\nabla^{2} K(x_{j}))^{-1}\frac{\nabla \Delta K(x_{j})}{\lambda_{j,m}^{3}},
\quad
\alpha_{j}=\Theta\cdot
 \sqrt[p-1]{\frac{\lambda_{j}^{\theta}}{K(a_{j,m})}}
  \end{equation*}
up to errors of order $o(\lambda_{\tau_{m}}^{-3})$, where $c_{1},c_{2}$ are dimensional constants and we  identify by a slight abuse of notation $a_{j,m}$ with its image in conformal normal coordinates at $x_{j}$, cf. \cite{lp}. 
Hence all the finite dimensional variables, i.e. $\alpha_{j,m}, a_{j,m}$ and $\lambda_{j,m}$ are determined to a precision of order $o(\lambda_{\tau_{m}}^{-3})$.
\end{remark}

  \begin{remark}  We next compare Theorem \ref{t:main} to some existing literature and add further comments. 
  \begin{enumerate}
 \item[(a)]  
On $S^3$ and $S^4$ the isolated-simpleness of solutions was proved in \cite{[ACGPY]}, \cite{yy}, \cite{yy2}, \cite{[SZ]}  for arbitrary sequences of solutions by a refined blow-up analysis. The  uniform $W^{1,2}$-bound  is then derived a-posteriori.  In  dimension $n \geq 5$  the latter bound may not hold true in general - we refer the reader to \cite{cl1}, \cite{cl2}, \cite{cl3}, where in some cases it is shown that  blowing-up solutions for the purely critical equation \eqref{eq:scin} must have diverging energy 
  and blow-ups of diverging energies and towering bubbles are also constructed, cf.  also \cite{leung}, \cite{rv}, \cite{wy}. 
  However, in the forthcoming paper \cite{MM4} we will construct solutions to \eqref{eq:scin-tau} via 
  min-max or Morse theory with the purpose of finding a non-zero weak limit. These 
  will indeed satisfy the required energy bound. 
  This will allow us to obtain existence results 
  under  less stringent conditions  compared to some others in the literature, 
  as in \cite{[BI]} and \cite{cx12}. 
\item[(b)]
On  manifolds not conformally equivalent to $S^n$
 a-priori estimates were proved in \cite{lizhu} for $n=3$ in both critical and subcritical cases.
 Our analysis carries over for $n=4$ as well, where the matrix in Definition \ref{def_none-degeneracy},
 introduced in \cite{[BCCH4]}, \cite{yy2} and also involving the {\em mass}, gives constraints on the 
 location of multiple blow-up points. 
  The main new aspect of our result is the isolated simple blow-up behaviour 
  in dimension $n \geq 5$, so we chose to state Theorem \ref{t:main} in a simple form only for 
  this case. We refer to Theorem \ref{lem_top_down_cascade} for a more precise version of the 
  result: here we derive indeed estimates on solutions with high precision as $\tau \longrightarrow 0$, 
  as well as estimates that are uniform in this parameter.

\item[(c)]  
In \cite{MM2} we will show a converse statement.  Given any distinct points 
  $p_1, \dots, p_k$ in $\{\nabla K = 0\} \cap \{ \Delta K < 0 \}$
  and  $\tau_i \searrow 0$ there exist solutions 
  $(u_i)_i$ to \eqref{eq:scin-tau} blowing-up at $p_1, \dots, p_k$ exactly as described above. Thence the characterization of Theorem \ref{t:main} is optimal. We refer to \cite{yy}, \cite{yy2} for the counterparts 
  on three- and four-spheres. 
Proposition A2 in \cite{bab1} regards the construction of a {\em pseudo gradient flow} for problem \eqref{eq:scin} ruling out 
multiple bubble formation at the same point for any $n$, although we believe the proof there is not complete. 
  We refer to \cite{MM3} for details and for the proof of 
  a one-to-one correspondence of 
 blowing-up sequences and {\em critical points at infinity},  cf.  \cite{bab}. See also 
 \cite{MM5}  for some delicate relations between $L^2$- and pseudo gradient flows. 
\item[(d)] 
We expect the same conclusion of Theorem \ref{t:main} should hold true replacing the 
  energy bound with a Morse index bound. It would also be  interesting to understand the 
  case of non-zero weak limits. 
  \end{enumerate}
\end{remark}

  We discuss next some heuristics about the proof of Theorem \ref{t:main}. 
  First we show a quantization result for Palais-Smale sequences of solutions to 
  \eqref{eq:scin-tau} as $\tau \longrightarrow 0$. We are inspired in this step from 
  a result by M.  Struwe in  \cite{str84}, where the same was proved for $\tau = 0$: in 
  our case we need extra work in the limiting process, due to a different 
  dilation covariance of subcritical equations.

  We then prove that we are in a perturbative regime and every solution to \eqref{eq:scin-tau} for $\tau$ sufficiently small
  can be written as a finite sum of highly peaked bubbles and an error term small in $W^{1,2}$-norm, which 
  we prove to have a minor effect in the  expansions. Performing a careful analysis 
  of the interactions of the bubbles among themselves and with  $K$, it is not difficult to see 
  that for $n \geq 5$ blow-ups should occur at critical points of $K$ with negative Laplacian only, cf.  also Theorem 1.1 
  in \cite{cl2}, and we are left with 
  excluding multiple bubbles towering at the same limit point, which is the crucial result in our paper. 
  
  We give  an idea of this fact in some \underline{particular cases},  that are easy to describe. 
  Let $J_\tau$ be the Euler-Lagrange energy of \eqref{eq:scin-tau}, see \eqref{eq:JKt}. 
  For  a critical point $a$ of $K$,  the following expansion holds for $J_{\tau}$ on a \emph{bubble}  
  concentrated at $a$ 
  \begin{equation}\label{eq:heur-exp}
  J_{\tau}(\d_{a,\lambda}) \simeq \frac{1}{K^{\frac{n-2}{n}}(a)}( \lambda^\tau- \frac{ \Delta K(a)}{K(a) \lambda^{2}}),
  \end{equation}
  cf. Proposition \ref{prop_functional_at_infinity}. By elementary considerations one checks that for 
  $\Delta K(a)<0$ the  function in the right-hand side has a 
  non-degenerate minimum point at $\lambda = \lambda_\tau \simeq \tau^{- \frac{1}{2}}$, see 
  also Proposition 2.1 in \cite{[SZ]}. Since 
  bubbles have an \emph{attractive interaction }, cf. the first equation in \eqref{eq:interactions-K-b}, 
  even in terms of dilations
 centering more bubbles at the point $a$ would make all dilation parameters {\em collapse} at $\lambda = \lambda_\tau$, 
 see Figure \ref{fig:awesome_image1}. 
  For the same reason, still by \eqref{eq:heur-exp}, one would get collapse with respect to the 
  center points of multiple bubbles distributed along the unstable directions from a 
  critical point of  $K$, since points with lager values of $K$ have smaller energy, due to \eqref{eq:heur-exp}, 
  see Figure \ref{fig:awesome_image2}. 
  \begin{figure}[!htb]
  \minipage{0.28\textwidth}
  \includegraphics[width=\linewidth]{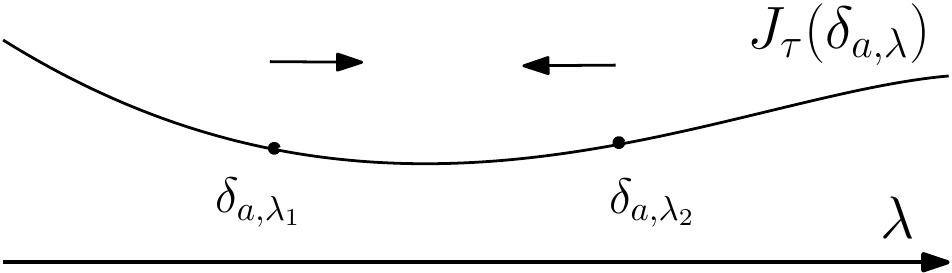}
  \caption{\small two bubbles with same center, different $\lambda$'s}\label{fig:awesome_image1}
  \endminipage\hfill
  \minipage{0.28\textwidth}
  \includegraphics[width=\linewidth]{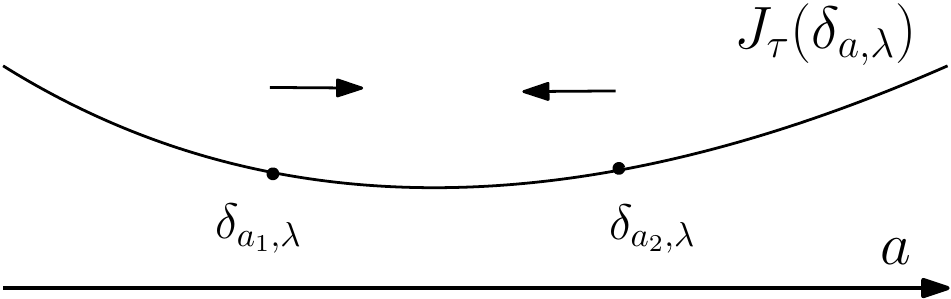}
  \caption{two bubbles  along unstable direction of $K$,  same $\lambda$}\label{fig:awesome_image2}
  \endminipage\hfill
  \minipage{0.28\textwidth}%
  \includegraphics[width=\linewidth]{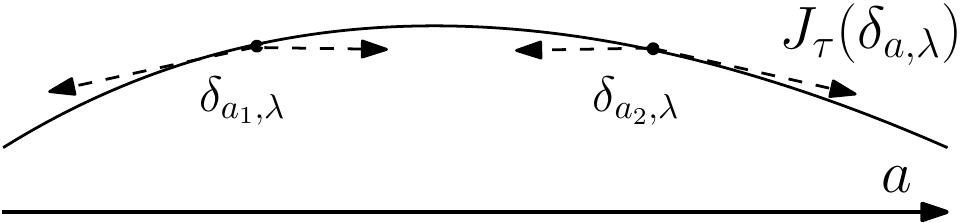}
  \caption{two bubbles  along stable direction of $K$, same $\lambda$}\label{fig:awesome_image3}
  \endminipage
  \end{figure}
  We consider then the case of  bubbles centered at two points $a_1, a_2$ symmetrically located at distance $d$ 
  from a  critical point $\bar{p}$ such that $\Delta K(\bar{p}) < 0$, and along a  stable direction of $K$, with the same $\l$'s. Here in principle the attractive force among bubbles could 
  compensate the {\em repulsive} interaction from the critical point $\bar{p}$ of  $K$, see Figure \ref{fig:awesome_image3}. 
  For this configuration one gets an energy expansion of the form 
  $$
  J_{\tau}(\d_{a_1,\lambda} + \d_{a_2,\lambda}) \simeq  
 \frac{c_0}{K^{\frac{n-2}{n}}(a_1)}( \lambda^\tau- \frac{ \Delta K(a_1)}{K(a_1) \lambda^{2}})
  - c_1 \frac{1}{d^{n-2} \lambda^{n-2}} 
  \simeq (c_2 - c_3 d^2) \left(  \lambda^\tau + c_4 \lambda^{-2}  \right) - c_1 \frac{1}{d^{n-2} \lambda^{n-2}}
%
  $$
  with $c_i > 0$. From the analysis in Proposition \
  \ref{blow_up_analysis} it turns out that 
 $\lambda^\tau \simeq 1$, so imposing criticality in both  $\l$ and $d$  one finds the relations
  $$
  \frac{1}{\lambda^2} \simeq \tau + \frac{1}{(\l d)^{n-2}} \quad \text{ and } \quad d \simeq \frac{1}{\l^{n-2} d^{n-1}}.  
  $$
  These asymptotics imply that $\lambda^{-2} \simeq \tau + \lambda^{-\frac{2(n-2)}{n}}$, 
  which is impossible for $\lambda$ large. 
  The general case is rather 
  involved to study and will be treated by a top-down cascade of estimates in Section \ref{s:lower-bds}.

  \

  The plan of the paper is the following. In Section \ref{s:set-up} we introduce the variational setting of the 
  problem and list some preliminary results. We then study some approximate 
 solutions of \eqref{eq:scin}, highly concentrated at arbitrary points of $M$. 
 From these one can carry out a reduction procedure of the problem, which is 
 done later in the paper. 
  In Section \ref{s:bu} we prove a general quantization result for Palais-Smale 
  sequences of \eqref{eq:scin-tau} with uniformly bounded $W^{1,2}$-energy. 
  In Section \ref{s:red} we reduce the problem to a finite-dimensional one, while in Section \ref{s:funct-infty} we derive some precise asymptotic expansions of the Euler-Lagrange energy. Section \ref{s:lower-bds} is then devoted to 
  proving suitable bounds on the gradient to exclude tower bubbles and prove our main result. 
  We finally collect in the appendix the proofs of some useful technical estimates as well as a list of relevant constants appearing.

  \
  
  \noindent {\bf Acknowledgments.}
  A.M. has been supported by the project {\em Geometric Variational Problems} and {\em Finanziamento a supporto della ricerca di base} from Scuola Normale Superiore and by MIUR Bando PRIN 2015 2015KB9WPT$_{001}$.  He is also member of GNAMPA as part of INdAM.

\section{Variational setting and preliminaries}\label{s:set-up}

In this section we collect some background and preliminary material, concerning the 
variational properties of the problem and some estimates on highly-concentrated 
approximate solutions of bubble type. 


We consider a smooth, closed Riemannian manifold 
$
M=(M^{n},g_{0}) 
$
with volume measure $\mu_{g_{0}}$ and scalar curvature $R_{g_{0}}$.
Letting 
$
\mathcal{A}=
\{
u\in W^{1,2}(M,g_{0})\mid u\geq 0,u\not \equiv 0
\}
$
the {\em Yamabe invariant}  is defined as 
\begin{equation*}\begin{split}
Y(M,g_{0})
= &
\inf_{\mathcal{A}}
\frac
{\int \left( c_{n}\vert \nabla u \vert_{g_{0}}^{2}+R_{g_{0}}u^{2} \right) d\mu_{g_{0}}}
{(\int u^{\frac{2n}{n-2}}d\mu_{g_{0}})^{\frac{n-2}{n}}}, \quad c_{n}=4\frac{n-1}{n-2}. 
\end{split}\end{equation*}
We will assume from now on that the invariant is positive.
As a consequence the {\em conformal Laplacian} 
\begin{equation*}\begin{split} 
L_{g_{0}}=-c_{n}\Delta_{g_{0}}+R_{g_{0}}
\end{split}\end{equation*}
 is a positive and self-adjoint operator. Without loss of generality 
we assume $R_{g_{0}}>0$ and denote by
$$
G_{g_{0}}:M\times M \setminus \Delta \longrightarrow\R_{+}
$$
the Green's function of $L_{g_0}$. 
Considering a conformal metric $g=g_{u}=u^{\frac{4}{n-2}}g_{0}$ there holds 
\begin{equation*}\begin{split} d\mu_{g_{u}}=u^{\frac{2n}{n-2}}d\mu_{g_{0}}
\; \quad \text{ and }\; \quad 
R=R_{g_{u}}=u^{-\frac{n+2}{n-2}}(-c_{n} \Delta_{g_{0}} u+R_{g_{0}}u) =
u^{-\frac{n+2}{n-2}}L_{g_{0}}u.
\end{split}\end{equation*}
Note that 
\begin{align*}
c\Vert u \Vert^{2}_{W^{1,2}(M,g_0)}\leq\int u \, L_{g_{0}}u \, d\mu_{g_{0}}=\int \left( c_{n}\vert \nabla u \vert^{2}_{g_{0}}+R_{g_{0}}u^{2}
\right) d\mu_{g_{0}}\leq C\Vert u \Vert^{2}_{W^{1,2}(M,g_0)}.
\end{align*}
In particular we may define and use
$\Vert u \Vert^2 = \Vert u \Vert_{L_{g_0}}^2 = \int u \, L_{g_{0}}u \, d\mu_{g_{0}}$ 
as an equivalent norm on $W^{1,2}$. For
\begin{equation*}
p=\frac{n+2}{n-2}-\tau\; \text{ and }\; 0\leq \tau \longrightarrow 0
\end{equation*}
we want to study the scaling-invariant functionals
\begin{equation}\label{eq:JKt}
J_{\tau}(u)
=
\frac
{\int_{M} \left( c_{n}\vert \nabla u\vert_{g_{0}}^{2}+R_{g_{0}}u^{2} \right) d\mu_{g_{0}}}
{(\int Ku^{p+1}d\mu_{g_{0}})^{\frac{2}{p+1}}},
 \quad  u \in \mathcal{A}.
\end{equation}
Since the conformal scalar curvature $R=R_{u}$ for $g=g_{u}=u^{\frac{4}{n-2}}g_{0}$ satisfies
\begin{equation}\label{eq:r}
r=r_{u}=\int R d\mu_{g_{u}}=\int u L_{g_{0}} ud\mu_{g_{0}},
\end{equation}
we have
\begin{equation}\label{eq:kp}
J_{\tau}(u)=\frac{r}{k_{\tau}^{\frac{2}{p+1}}}
\; \quad \text{ with } \; \quad k_{\tau} = \int K \, u^{p+1} d \mu_{g_{0}}. 
\end{equation}
The first- and second-order derivatives of the functional are given by 
\begin{equation*}
\partial J_{\tau}(u)v
= 
\frac{2}{k_{\tau}^{\frac{2}{p+1}}}
\big[\int L_{g_{0}}uvd\mu_{g_{0}}-\frac{r}{k_{\tau}}\int Ku^{p}vd\mu_{g_{0}}\big];
\end{equation*} 
\begin{equation*}
\begin{split}
\partial^{2} J_{\tau}(u) vw
= &
\frac{2}{k_{\tau}^{\frac{2}{p+1}}}
\big[\int L_{g_{0}}vwd\mu_{g_{0}}-p\frac{r}{k_{\tau}}\int Ku^{p-1}vwd\mu_{g_{0}}\big] \\
& -
\frac{4}{k_{\tau}^{\frac{2}{p+1}+1}}
\big[
\int L_{g_{0}}uvd\mu_{g_{0}}\int Ku^{p}wd\mu_{g_{0}}
+
\int L_{g_{0}}uwd\mu_{g_{0}}\int Ku^{p}vd\mu_{g_{0}}
\big] \\
& +
\frac{2(p+3)r}{k_{\tau}^{\frac{2}{p+1}+2}}
\int Ku^{p}vd\mu_{g_{0}}\int Ku^{p}wd\mu_{g_{0}}.
\end{split}
\end{equation*}
In particular $J_{\tau}$ is of class $C^{2, \alpha}_{loc}(\mathcal{A})$ and uniformly H\"older continuous on each
set of the form 
\begin{align*}
U_{\epsilon}=\{u\in  \mathcal{A} \mid  \epsilon<\Vert u \Vert,\,J_{\tau}(u)\leq \epsilon^{-1}\}.
\end{align*}
Indeed
$u\in U_{\epsilon}$ implies 
\begin{align*}
\epsilon^{2}\leq r\leq \epsilon^{-2}
\;\text{ and }\;
c \, \epsilon^{3}\leq k_{\tau}^{\frac{1}{p+1}}=J_{\tau}(u)^{-1}r_{u}\leq C\epsilon^{-3}.
\end{align*}
Thus uniform H\"older continuity on $U_{\epsilon}$ follows from the standard pointwise estimates
\begin{equation} \label{eq:ineq-p-2} 
\left\{
\begin{matrix}[cl]
\vert \vert a\vert^{p}-\vert b \vert^{p}\vert
\leq 
C_{p}\vert a-b\vert^{p}  
& 
\text{in case } \;0<p<1
\\
\vert \vert a \vert^{p}-\vert b \vert^{p}\vert
\leq 
C_{p}\max\{\vert a\vert^{p-1},\vert b \vert^{p-1}\}\vert a-b\vert 
&
\text{in case }\; p\geq 1
\end{matrix}
\right.
\end{equation}

We consider next some approximate solutions to \eqref{eq:scin}, highly concentrated at arbitrary points 
of $M$. As we will see, for suitable values of $\lambda$ these are also approximate solutions 
of \eqref{eq:scin-tau}. 
Let us recall the construction of {\em conformal normal coordinates} from \cite{lp}.
Given $a \in M$, one chooses a special conformal metric 
\begin{equation}\label{conformal_normal_metric}
\begin{split}
g_{a}=u_{a}^{\frac{4}{n-2}}g_{0}
\quad \text{ with } \quad 
u_{a}=1+O(d^{2}_{g_{0}}(a,\cdot)),
\end{split}
\end{equation}
whose volume element in $ g_{a}$-geodesic normal coordinates coincides with the Euclidean one, see also \cite{gunther}. In particular
\begin{equation*}
\begin{split} 
(exp_{a}^{g_{0}})^{-}\circ \exp_{a}^{g_{a}}(x)
=
x+O(\vert x \vert^{3})
\end{split}
\end{equation*}
for the exponential maps centered at $ a $, which e.g. implies
\begin{equation*}
\begin{split}
\nabla_{g_{0}}K(a)=\nabla_{g_{a}}K(a),\;
\nabla^{2}_{g_{0}}K(a)=\nabla^{2}_{g_{a}}K(a), 
\end{split}
\end{equation*}
and in case $ \nabla K(a)=0 $  also 
\begin{equation*}
\begin{split}
\nabla^{3}_{g_{0}}K(a)=\nabla^{3}_{g_{a}}K(a).  
\end{split}
\end{equation*}
Moreover by smoothness of the  exponential map $\exp_{g_{a}}=\exp^{g_{a}}_{a}$ with respect to 
$a$ there holds 
\begin{equation}\label{conformal_normal_coordinates_expansion}
\nabla_{a}\exp_{g_{a}}(x)=id+O(\vert x \vert^{2})
\end{equation}
in a $ g_{a}$-normal chart, as seen from the corresponding geodesic equation. 
We then denote by $r_a$ the geodesic distance from $a$ with respect to the metric $g_a$ 
just introduced. With this choice  the expression of the 
Green's function $G_{g_{ a }}$  with pole at $a \in M$, denoted by 
$G_{a}=G_{g_{a}}(a,\cdot)$,  for the conformal 
Laplacian $L_{g_{a}}$ simplifies considerably. From Section 6 in \cite{lp}
one may expand
\begin{equation}\begin{split} \label{eq:exp-G}
G_{ a }=\frac{1}{4n(n-1)\omega _{n}}(r^{2-n}_{a}+H_{ a }), \; r_{a}=d_{g_{a}}(a, \cdot)
, \; 
H_{ a }=H_{r,a }+H_{s, a }\; \text{ for } \; g_{a}=u_{a}^{\frac{4}{n-2}}g_{0}, 
\end{split}\end{equation}
where $\omega_{n} = |S^{n-1}|$. Here $H_{r,a }\in C^{2, \alpha}_{loc}$, while the {\em singular} error term satisfies
\begin{equation*}
\begin{split}
H_{s,a}
=
O
\begin{pmatrix}
0 & \text{ for }\, n=3
\\
r_{a}^{2}\ln r_{a} & \text{ for }\, n=4 
\\
r_{a}& \text{ for }\, n=5
\\
\ln r_{a} & \text{ for }\, n=6
\\
r_{a}^{6-n} & \text{ for }\, n\geq 7
\end{pmatrix}.
\end{split}
\end{equation*} 
Precisely  the leading term in $H_{s,a}$ for $n=6$ is 
$-\frac{\vert \mathbb{W}(a)\vert^{2}}{288 c_{n}}\ln r$,
where $\mathbb{W}$ denotes the Weyl tensor.
Let 
\begin{equation}\begin{split}\label{eq:bubbles}
\varphi_{a, \lambda }
= &
u_{ a }\left(\frac{\lambda}{1+\lambda^{2} \gamma_{n}G^{\frac{2}{2-n}}_{ a }}\right)^{\frac{n-2}{2}}, \quad 
G_{ a }=G_{g_{ a }}( a, \cdot), \quad
\gamma_{n}=(4n(n-1)\omega _{n})^{\frac{2}{2-n}}
\quad \text{ for }\quad  \lambda>0.
\end{split}\end{equation}
We notice that the constant $\gamma_{n}$ is chosen so that 
$$\gamma_{n}G^{\frac{2}{2-n}}_{ a }(x) = 
d_{g_a}^{2}(a,x) + o(d_{g_a}^{2}(a,x)) \quad \hbox { as } \quad x \longrightarrow a.
$$ 
Evaluating the conformal Laplacian on such functions shows that they are approximate solutions. 
\begin{lemma}\label{lem_emergence_of_the_regular_part}
There holds 
$
L_{g_{0}}\varphi_{a, \lambda}=
O
(
\varphi_{a, \lambda}^{\frac{n+2}{n-2}}
).
$
More precisely on a geodesic ball $B_{\alpha}( a )$ for $\alpha>0$ small 
\begin{equation*}
L_{g_{0}}\varphi_{a, \lambda}
= 
4n(n-1)\varphi_{a, \lambda}^{\frac{n+2}{n-2}}
-
2nc_{n}
r_{a}^{n-2}((n-1)H_{a}+r_{a}\partial_{r_{a}}H_{a}) \varphi_{a, \lambda}^{\frac{n+2}{n-2}} 
+
\frac{u_{a}^{\frac{2}{n-2}}R_{g_{a}}}{\lambda}\varphi_{a, \lambda}^{\frac{n}{n-2}}
+
o(r_{a}^{n-2})\varphi_{a, \lambda}^{\frac{n+2}{n-2}},
\end{equation*}
where  $r_{a}=d_{g_{a}}(a, \cdot)$. Since $R_{g_{a}}=O(r_{a}^{2})$ in conformal normal coordinates, cf. \cite{lp}, we obtain
\begin{enumerate}[label=(\roman*)]
 \item \quad 
$
L_{g_{0}}\varphi_{a, \lambda}
= 
4n(n-1)
[1
-
\frac{c_{n}}{2}r_{a}^{n-2}(
H_{a}(a) + n \nabla H_{a}(a)x
)
]
\varphi_{a, \lambda}^{\frac{n+2}{n-2}} 
 +
O
\begin{pmatrix}
\lambda^{-\frac{3}{2}}\varphi_{a,\lambda}^{\frac{n-1}{n-2}}& \text{ for }\, n=3
\\
\frac{\ln r}{\lambda^{2}}\varphi_{a,\lambda}^{\frac{n-1}{n-2}}& \text{ for }\, n=4
\\
\lambda^{-2}\varphi_{a,\lambda}& \text{ for }\, n=5
\end{pmatrix};
$
\item \quad 
$L_{g_{0}}\varphi_{a,\lambda}=
4n(n-1)\varphi_{a,\lambda}^{\frac{n+2}{n-2}}
=
4n(n-1)[1+\frac{c_{n}}{2}W(a)\ln r]\varphi_{a,\lambda}^{\frac{n+2}{n-2}}
+
O
(\lambda^{-2}\varphi_{a,\lambda}) \quad \hbox{ for } n = 6$;
\item \quad 
$L_{g_{0}}\varphi_{a,\lambda}=
4n(n-1)\varphi_{a,\lambda}^{\frac{n+2}{n-2}}
+
O
(\lambda^{-2}\varphi_{a,\lambda}) \quad \hbox{ for } n \geq 7. 
$
\end{enumerate}
The expansions stated above persist upon taking $\lambda \partial \lambda$ and $\frac{\nabla_{a}}{\lambda}$ derivatives.
\end{lemma}
\begin{proof}
A straightforward calculation shows that 
\begin{equation*}
\Delta _{g_{ a }}\big(\frac{\lambda}{1+\lambda^{2}\gamma_{n}G_{ a }^{\frac{2}{2-n}}}\big)^{\frac{n-2}{2}}
= 
\frac{n}{2-n}\gamma_{n}\big(\frac{\varphi_{a, \lambda}}{u_{a}}\big)^{\frac{n+2}{n-2}}
\vert \nabla G_{ a } \vert^{2}_{g_{a}} 
G_{a}^{2\frac{n-1}{2-n}}
 +
\gamma_{n}\lambda\big(\frac{\varphi_{a, \lambda}}{u_{a}}\big)^{\frac{n}{n-2}}G_{ a }^{\frac{n}{2-n}}\Delta _{g_{a}} G_{ a },
\end{equation*}  
which is due to
$
\vert \nabla G_{ a } \vert^{2}_{g_{a}} 
G_{a}^{2\frac{n-1}{2-n}}
=
(n-2)^{2}\vert \nabla G_{ a }^{\frac{1}{2-n}} \vert^{2}_{g_{a}}
\;
\text{ and }\;
c_{n}\Delta _{g_{a}}G_{a}=-\delta_{a}+R_{g_{a}}G_{a}
$
with $\delta_{a}$ denoting the Dirac measure at $a$. This is equivalent to
\begin{equation*}
\Delta _{g_{ a }}\big(\frac{\lambda}{1+\lambda^{2}\gamma_{n}G_{ a }^{\frac{2}{2-n}}}\big)^{\frac{n-2}{2}}
= 
n(2-n)\gamma_{n}(\frac{\varphi_{a, \lambda}}{u_{a}})^{\frac{n+2}{n-2}}
\vert \nabla G_{ a }^{\frac{1}{2-n}} \vert^{2}_{g_{a}}
 +
\frac{R_{g_{ a }}\gamma_{n}}{c_{n}}\lambda\big(\frac{\varphi_{a, \lambda}}{u_{a}}\big)^{\frac{n}{n-2}}G_{ a }^{\frac{2}{2-n}}
.
\end{equation*}
Since $L_{g_{a}}=-c_{n}\Delta _{g_{a}}+R_{g_{a}}$ with $c_{n}=4\frac{n-1}{n-2}$, we obtain
\begin{equation*}\begin{split}
L_{g_{ a }}\frac{\varphi_{a, \lambda}}{u_{a}}
= &
4n(n-1)\big(\frac{\varphi_{a, \lambda}}{u_{a}}\big)^{\frac{n+2}{n-2}}\gamma_{n}
\vert \nabla G_{ a }^{\frac{1}{2-n}} \vert^{2}_{g_{a}} 
+
\frac{R_{g_{ a }}}{\lambda}\big(\frac{\varphi_{a, \lambda}}{u_{a}}\big)^{\frac{n}{n-2}}.
\end{split}\end{equation*}
By conformal covariance we also get  
\begin{equation*}
L_{g_{0}} \varphi_{a, \lambda}
= 
4n(n-1)\varphi_{a, \lambda}^{\frac{n+2}{n-2}}\gamma_{n}\vert \nabla G_{ a }^{\frac{1}{2-n}} \vert^{2}_{g_{a}}
+
\frac{u_{a}^{\frac{2}{n-2}}R_{g_{a}}}{\lambda}\varphi_{a, \lambda}^{\frac{n}{n-2}},
\end{equation*}
in particular $L_{g_{0}}\varphi_{a, \lambda}=O(\varphi_{a, \lambda}^{\frac{n+2}{n-2}})$.
Expanding $G_a$ as 
$
G_{ a }=\frac{1}{4n(n-1)\omega _{n}}(r_{a}^{2-n}+H_{ a }), \, r_{ a }=d_{g_{a}}(a, \cdot)
$
we find
\begin{equation*}\begin{split}
\gamma_{n} \vert \nabla G_{ a }^{\frac{1}{2-n}} \vert^{2}_{g_{a}}
= &
\vert \nabla(r_{a}(1+r_{a}^{n-2} H_{ a })^{\frac{1}{2-n}})\vert^{2}_{g_{a}} 
= 
1
-
\frac{2}{n-2}
( (n-1)H_{ a }+r_{a}\partial_{r_{a}}H_{ a }) r_{a}^{n-2}
+
o(r_{a}^{n-2}), 
\end{split}\end{equation*}
and conclude that 
\begin{equation*}
L_{g_{0}} \varphi_{a, \lambda}
= 
4n(n-1)\varphi_{a, \lambda}^{\frac{n+2}{n-2}}
-
2nc_{n}
((n-1)H_{ a }+r_{a}\partial_{r_{a}}H_{ a }) r_{a}^{n-2}\varphi_{a, \lambda}^{\frac{n+2}{n-2}} 
 +
o(r_{a}^{n-2}\varphi_{a, \lambda}^{\frac{n+2}{n-2}})
+
\frac{u_{a}^{\frac{2}{n-2}}R_{g_{a}}}{\lambda}\varphi_{a, \lambda}^{\frac{n}{n-2}}.
\end{equation*}
Clearly these calculations transcend to the $\lambda$ and $a$ derivatives. Then the claim follows from  the 
above expansion of the Green's function.
\end{proof}

After introducing some notation we  state a useful lemma, which will be proved in the first appendix. 

\bigskip

\noindent {\bf Notation.} Given an exponent $p \geq 1$ we will denote by $L^p_{g_0}$ 
the set of functions of class $L^p$ with 
respect to the measure $d \mu_{g_{0}}$. 
Recall also that for $u \in W^{1,2}(M,g_0)$ we set $r_u = \int u L_{g_0} u d \mu_{g_{0}}$, while for a
point $a \in M$ we  denote by $r_a$ the geodesic distance from $a$ with respect to the metric $g_a$ 
introduced above.
For a  set of points $\{a_i\}_i$ of $M$  we will denote  by $K_{i},\nabla K_i$ and $\Delta K_i$ for instance  
$$
K(a_{i}),\;
\nabla K(a_{i}) = \nabla_{g_0} K(a_{i}) 
\quad \text{ and } \quad 
\Delta K(a_{i}) = \Delta_{g_0} K(a_{i}).$$
\noindent
For $k,l=1,2,3$ and $ \lambda_{i} >0, \, a _{i}\in M, \,i= 1, \ldots,q$ let
\begin{enumerate}[label=(\roman*)]
 \item \quad 
$\varphi_{i}=\varphi_{a_{i}, \lambda_{i}}$ and $(d_{1,i},d_{2,i},d_{3,i})=(1,-\lambda_{i}\partial_{\lambda_{i}}, \frac{1}{\lambda_{i}}\nabla_{a_{i}});$
 \item \quad
$\phi_{1,i}=\varphi_{i}, \;\phi_{2,i}=-\lambda_{i} \partial_{\lambda_{i}}\varphi_{i}, \;\phi_{3,i}= \frac{1}{\lambda_{i}} \nabla_{ a _{i}}\varphi_{i}$, so
$
\phi_{k,i}=d_{k,i}\varphi_{i}. 
$
\end{enumerate}
Note  that with the above definitions  the $\phi_{k,i}$'s are uniformly bounded in $W^{1,2}(M,g_0)$. 
\
\begin{lemma}\label{lem_interactions}$_{}$
Let $\theta=\frac{n-2}{2}\tau$ and $k,l=1,2,3$ and $i,j = 1, \ldots,q$. Then for
\begin{equation}\label{eq:eij}
\varepsilon_{i,j}
=
(
\frac{\lambda_{j}}{\lambda_{i}}
+
\frac{\lambda_{i}}{\lambda_{j}}
+
\lambda_{i}\lambda_{j}\gamma_{n}G_{g_{0}}^{\frac{2}{2-n}}(a _{i},a _{j})
)^{\frac{2-n}{2}}
\end{equation}
there holds uniformly as $0\leq \tau\longrightarrow 0$
\begin{enumerate}[label=(\roman*)]
 \item \quad
$ \vert \phi_{k,i}\vert, 
  \vert \lambda_{i}\partial_{\lambda_{i}}\phi_{k,i}\vert,
  \vert \frac{1}{\lambda_{i}}\nabla_{a_{i}} \phi_{k,i}\vert
  \leq 
  C \varphi_{i};$
 \item \quad
$ 
\lambda_{i}^{\theta}\int \varphi_{i}^{\frac{4}{n-2}-\tau} \phi_{k,i}\phi_{k,i}d\mu_{g_{0}}
=
c_{k}\cdot id
+
O(\tau +\frac{1}{\lambda_{i}^{n-2+\theta}}+\frac{1}{\lambda_{i}^{2+\theta}}), \;c_{k}>0;$
\item \quad 
for  $i\neq j$ up to some error of order 
$O(\tau^{2}+\sum_{i\neq j}(\frac{1}{\lambda_{i}^{4}}+\frac{1}{\lambda_{i}^{2(n-2)}}+\varepsilon_{i,j}^{\frac{n+2}{n}}))$
\begin{equation*}
\lambda_{i}^{\theta}\int \varphi_{i}^{\frac{n+2}{n-2}-\tau}\phi_{k,j}d\mu_{g_{0}}
= 
b_{k}d_{k,i}\varepsilon_{i,j}
=
\int \varphi_{i}^{1-\tau}d_{k,j}\varphi_{j}^{\frac{n+2}{n-2}}  d\mu_{g_{0}} ;  \end{equation*}
 \item \quad 
$
\lambda_{i}^{\theta}\int \varphi_{i}^{\frac{4}{n-2}-\tau} \phi_{k,i}\phi_{l,i}d\mu_{g_{0}}
= 
O(\frac{1}{\lambda_{i}^{n-2}}+\frac{1}{\lambda_{i}^{2}})$
for $k\neq l$ and for $k=2,3$
$$\textstyle  \quad 
\lambda_{i}^{\theta}\int \varphi_{i}^{\frac{n+2}{n-2}-\tau} \phi_{k,i}d\mu_{g_{0}}
=
O\left(
\tau
+
\begin{pmatrix}
\lambda_{i}^{2-n} & \text{for } n\leq 5 \\
\frac{\ln \lambda_{i}}{\lambda_{i}^{4}} & \text{for } n=6 \\
\lambda_{i}^{4} &  \text{for } n\geq 7
\end{pmatrix}
\right);
$$
 \item \quad
$
\lambda_{i}^{\theta}\int \varphi_{i}^{\alpha-\tau }\varphi_{j}^{\beta} d\mu_{g_{0}}
=
O(\varepsilon_{i,j}^{\beta})
$
for $i\neq j,\;\alpha +\beta=\frac{2n}{n-2}, \; \alpha-\tau>\frac{n}{n-2}>\beta\geq 1;$
\item \quad
$
\int \varphi_{i}^{\frac{n}{n-2}}\varphi_{j}^{\frac{n}{n-2}} d\mu_{g_{0}}
=
O(\varepsilon^{\frac{n}{n-2}}_{i,j}\ln \varepsilon_{i,j}), \,i\neq j;
$
 \item \quad 
$
(1, \lambda_{i}\partial_{\lambda_{i}}, \frac{1}{\lambda_{i}}\nabla_{a_{i}})\varepsilon_{i,j}=O(\varepsilon_{i,j})
, \,i\neq j$. 
\end{enumerate}  
with constants $ b_{k}=\underset{\R^{n}}{\int}\frac{dx}{(1+r^{2})^{\frac{n+2}{2}}}$ for $k = 1, 2, 3$ and 
$$
c_{1}=\underset{\R^{n}}{\int}\frac{dx}{(1+r^{2})^{n}},\;\;
c_{2}=\frac{(n-2)^{2}}{4}\underset{\R^{n}}{\int}\frac{( r^{2}-1)^{2}dx}{(1+r^{2})^{n+2}} ,\;\;
c_{3}=\frac{(n-2)^{2}}{n}\underset{\R^{n}}{\int}\frac{r^{2}dx}{(1+r^{2})^{n+2}}.
$$
\end{lemma}

 \section{Blow-up analysis}\label{s:bu}

 In this section we prove a result related to a well-known one in \cite{str84}. We obtain 
 indeed similar conclusions, but allowing the exponent in the equation to vary 
 along a sequence of approximate solutions.

 \begin{proposition}\label{blow_up_analysis}$_{}$
 Let $(u_{m})_{m} \subset W^{1,2}(M,g_0)$ be a sequence with $u_m \geq 0$ and $k_{\tau_{m}}=1$ 
 satisfying 
 \begin{equation*}
 J_{\tau_{m}}(u_{m})=r_{u_{m}}\longrightarrow r_{\infty} \; \text{ and }\;  \partial J_{\tau_{m}}(u_{m})\longrightarrow 0
 \hbox{ in } W^{-1,2}(M,g_0).  
 \end{equation*}
 Then  up to a subsequence there exist $u_\infty : M \longrightarrow [0,\infty)$ smooth, 
 $q \in \mathbb{N}_{0}$ and for $i=1, \ldots,q$ sequences
 \begin{equation*}\begin{split}
 M \supset (a_{i,m})\longrightarrow a_{i_{\infty}}
 \; \text{ and }\;
 \R_{+}\supset \lambda_{i,m} \longrightarrow \infty
 \; \text{ as }\; m\longrightarrow  \infty
 \end{split}\end{equation*}
 such that 
 $u_{m}
 =
 u_{\infty }
 +
 \sum_{i=1}^{q}\alpha_{i}\varphi_{a_{i,m},\lambda_{i,m}}+v_{m}$
 with 
 \begin{equation*}
 \partial J_{0}(u_{\infty})=0,\; \quad 
 \Vert v_{m} \Vert \longrightarrow 0 
 , \quad \; \lambda_{i,m}^{\tau_{m}}\longrightarrow  1 \quad 
 \; \text{ and }\; \quad 
 \frac{r_{\infty}K(a_{i_{\infty}})\alpha_{i}^{\frac{4}{n-2}}}{4n(n-1)}= 1 
 \end{equation*}
 and $(\varepsilon_{i,j})_{m} \longrightarrow 0$ as $m \longrightarrow \infty$ for each pair $1\leq i<j\leq q$  . 
 \end{proposition}
 
 \begin{proof} Setting  $J=J_{\tau_{m}}$, by our assumptions we have 
 \begin{equation*}
 J(u_{m})=\int u_{m} L_{g_{0}}u_{m}d\mu_{g_{0}}\longrightarrow r_{\infty}
 \;  \text{ and } \; 
 \partial J(u_{m})
 =
 L_{g_{0}}u_{m}-r_{\infty}Ku_{m}^{p_{m}}=o(1) \hbox{ in } W^{-1,2}(M,g_0).
 \end{equation*}
 In particular $(u_{m})\subset W^{1,2}(M,g_0)$ is bounded, hence 
 $ 
 u_{m}\rightharpoonup u_{\infty}$  weakly  in $W^{1,2}(M,g_0) $ and strongly in  $L^{q}(M,g_0)$, $q<\frac{2n}{n-2}$.  
 Notice that $u_\infty \geq 0$ is a critical point of $J_{0}$ and therefore it is a smooth function. 
 We may then write $u_{m}=u_{\infty}+u_{1,m}$ with $u_{1,m}\rightharpoonup 0\text{ weakly, and strongly in  } L^{q}(M,g_0).$ Thus
 \begin{equation*}
 r_{\infty}\longleftarrow J(u_{m})=\int u_{\infty} L_{g_{0}}u_{\infty}d\mu_{g_{0}}+\int u_{1,m} L_{g_{0}}u_{1,m}d\mu_{g_{0}}+o(1),
 \end{equation*}
 whence 
 $\int u_{1,m} L_{g_{0}}u_{1,m} d\mu_{g_{0}} \longrightarrow  r_{1,\infty}\geq 0$ and secondly, due to \eqref{eq:ineq-p-2}, that 
 \begin{equation}\label{eq:Eu1m}
 E(u_{1,m}):
 =
 L_{g_{0}}u_{1,m}-r_{\infty}Ku_{1,m}^{p_{m}}=o(1)  \quad \hbox{ in } W^{-1,2}(M,g_0). 
 \end{equation} 
 We may assume $r_{1,\infty}>0$, since otherwise we are done. We now claim the concentration behavior  
 \begin{equation}\label{eq:conc-grad}
 \forall \; 0<\varepsilon \ll 1 \;\exists\; \lambda_{m}\longrightarrow \infty
 \;:\; \sup_{x\in M}\int_{B_{\frac{1}{\lambda_{m}}}(x)} \vert \nabla u_{1,m}\vert_{g_{0}}^{2}d\mu_{g_{0}} \geq \varepsilon.
 \end{equation}
 Indeed we have for a fixed cut-off function 
 \begin{equation*}
 \begin{split}
 o(1)
 = &
 \langle E(u_{1,m}),u_{1,m}\eta^{2}\rangle 
 = 
 \int \left[ (\eta u_{1,m}) L_{g_{0}}(\eta u_{1,m})-r_{\infty}K\vert \eta u_{1,m}\vert^{2}u_{1,m}^{p_{m}-1} \right] d\mu_{g_{0}} +o(1) \\
 \geq &
 \Vert \nabla (\eta u_{1,m})\Vert_{}^{2} 
 -
 r_{\infty}K_{\min}\Vert \eta u_{1,m}\Vert_{L^{p_{m}+1}_{\mu_{g_{0}}}}^{2}\Vert u_{1,m}\Vert_{L^{p_{m}+1}_{\mu_{g_{0}}}(supp(\eta))}^{p_{m}-1} +o(1).
 \end{split}
 \end{equation*}
 Using H\"older's inequality and  Sobolev's embedding we obtain
 \begin{equation*}
 \begin{split}
 o(1)
 \geq &
 \Vert \nabla (\eta u_{1,m})\Vert_{}^{2} 
 (
 1-C\Vert u_{1,m}\Vert_{L^{p_{m}+1}_{\mu_{g_{0}}}(supp(\eta))}^{p_{m}-1}
 )
 +
 o(1).
 \end{split}
 \end{equation*}
 Thus, if $u_{1,m}$ does not concentrate in $L^{p_{m}+1}(M,g_0)$ similarly to \eqref{eq:conc-grad}, then by a covering argument
 \begin{equation*}
 \int \vert \nabla u_{1,m}\vert_{g_{0}}^{2}d\mu_{g_{0}}\longrightarrow  0
 \end{equation*}
 contradicting $r_{1,\infty}>0$. By \eqref{eq:Eu1m}
 concentration in $L^{p_{m}+1}(M,g_0)$ is equivalent to concentration in 
 $L^{2}$-norm for the gradient, which had to be shown. 
 Fixing $\varepsilon>0$ small, we measure the rate of concentration via
 \begin{equation*}
 \Lambda_{1,m}
 =
 \sup \bigg\{\lambda>0 \;\mid \; \max _{x\in M}\int_{B_{\frac{1}{\lambda}}(x)}\vert \nabla u_{1,m}\vert_{g_{0}}^{2}d\mu_{g_{0}}=\varepsilon\bigg\} \longrightarrow \infty, 
 \end{equation*}
 and choose for any 
 $\lambda_{1,m}\nearrow \infty$ 
 with 
 $1\leq \lim_{m\to \infty }\frac{\Lambda_{1,m}}{\lambda_{1,m}}=\delta<\infty$
 up to a subsequence 
 \begin{equation*}
 (a_{1,m})\subset M \;:\; \int_{B_{\frac{1}{\lambda_{1,m}}}(a_{1,m})}\vert \nabla u_{1,m}\vert_{g_{0}}^{2}d\mu_{g_{0}}
 =
 \sup _{x\in M}\int_{B_{\frac{1}{\lambda_{1,m}}}(x)}\vert \nabla u_{1,m}\vert_{g_{0}}^{2}d\mu_{g_{0}}\geq c
 \end{equation*}
 for some positive  $ c=c(\varepsilon, \delta)$ to be specified later. On a suitably small ball $B_{\rho}(a_{1,m})$ we then rescale
 \begin{equation*}
 w_{1,m}
 = 
 \lambda_{1,m}^{\frac{2-n}{2}}u_{1,m}\bigg(\exp_{g_{a_{1,m}}}\frac{\cdot}{\lambda_{1,m}}\bigg).
 \end{equation*}
 The function $w_{1,m }$ is well defined on $B_{\rho\lambda_{1,m}}(0)$ and satisfies, with $\theta_{m}=\frac{n-2}{2}\tau_{m}$,
 \begin{equation*}
 -c_{n}\Delta w_{1,m}-\frac{r_{\infty}K(a_{i,m})}{\lambda_{1,m}^{\theta_{m}}}w_{1,m}^{p_{m}}=o(1)
 \quad \hbox{ in }  W^{-1,2}_{loc}(\R^n), \quad  \Delta=\Delta_{\R^{n}}. 
 \end{equation*}
 Since $\int \vert \nabla u_{1,m}\vert^{2}d\mu_{g_{0}}$ is bounded, so it is $\int_{B_{\rho\lambda_{1,m}}(0)}\vert \nabla w_{1,m}\vert^{2}dx$ for any $\rho > 0$. Hence 
 \begin{equation*}
 w_{1,m}\rightharpoonup w_{1,\infty}\text{ weakly in } W^{-1,2}_{loc}(\R^n)
 \; \text{ with } \;
 -\Delta w_{1,\infty}=\sigma_{1} r_{\infty}\kappa_{1}w_{1,\infty}^{\frac{n+2}{n-2}},
 \end{equation*}
 where 
 \begin{equation*}
 \kappa_1=\lim_{m\to \infty }K(a_{1,m}) \; \text{ and } \;\sigma_{1} =\lim_{m\to \infty }\lambda_{1,m}^{-\theta_{m}}\in [0,1]. 
 \end{equation*}
 Given a compactly supported cut-off  $\eta$, we calculate
 \begin{equation}\label{blow_up_control_critical_estimate}
 \begin{split}
 0\longleftarrow
 &
 \underset{\R^{n}}{\int} 
 (w_{1,m}-w_{1,\infty})\eta^{2}\big(\Delta w_{1,m}
 +
 \frac{r_{\infty}K}{\lambda_{1,m}^{\theta_{m}}}w_{1,m}^{p_{m}}\big) dx \\
 = &
 \underset{\R^{n}}{\int} 
 (w_{1,m}-w_{1,\infty})\eta^{2}\big(\Delta (w_{1,m}-w_{1,\infty}) 
 +
 \sigma_{1} r_{\infty}K (w_{1,m}^{p_{m}}-w_{1,\infty}^{\frac{n+2}{n-2}}) \big) dx
 +
 o(1) 
 \\
 \leq  & 
 -\underset{\R^{n}}{\int} \vert \nabla ((w_{1,m}-w_{1,\infty})\eta)\vert^{2} dx
 +
 \sigma_{1} r_{\infty}\underset{\R^{n}}{\int} K \eta^{2} \vert w_{1,m}-w_{1,\infty}\vert^{p_{m}+1}dx
 +
 o(1) \\
 = &
 -\underset{\R^{n}}{\int} \vert \nabla ((w_{1,m}-w_{1,\infty})\eta)\vert^{2}dx
 +
 \sigma_{1} r_{\infty}\underset{\R^{n}}{\int} K \eta^{2} \vert w_{1,m}-w_{1,\infty}\vert^{p_{m}+1}dx
 +
 o(1). 
 \end{split}
 \end{equation}
 The main step here is the inequality in the above formula.  Passing from 
 $\frac{n+2}{n-2}$ to $p_{m}=\frac{n+2}{n-2} - \tau_{m}$
 in the exponent  is easy, as $w_{1,\infty}$ is fixed. Since $w_{1,m}\rightarrow w_{1,\infty}$ in $L^{p}(supp(\eta)), \; p<\frac{2n}{n-2}$,
 we have 
 \begin{equation*}
 \begin{split}
 \underset{\R^{n}}{\int} K \eta^{2} & ( w_{1,m}  -w_{1,\infty}) ( w_{1,m}^{p_{m}}-w_{1,\infty}^{p_{m}})dx
 = 
 \underset{\R^{n}}{\int} K \eta^{2}(w_{1,m}^{p_{m}+1}-w_{1,\infty}^{p_{m+1}}) dx \\
 = &
 \underset{\R^{n}}{\int} K\eta^{2}
 \bigg[
 -
 \int^{1}_{0}\partial_{s}\vert w_{1,m} 
 -
 s w_{1,\infty}\vert^{p_{m}+1}d s-w_{1,\infty}^{p_{m}+1} 
 +
 \vert w_{1,m}-w_{1,\infty}\vert^{p_{m+1}}
 \bigg] dx.
 \end{split}
 \end{equation*}
 Therefore the main inequality follows from observing that
 \begin{equation*}
 \begin{split}
 \bigg\vert \underset{\R^{n}}{\int} K & \eta^{2}
 \bigg[
 -
 \int^{1}_{0}\partial_{s}\vert w_{1,m}
 -
 s w_{1,\infty}\vert^{p_{m}+1}ds-w_{1,\infty}^{p_{m}+1}
 \bigg]dx
 \bigg\vert \\
 \leq &
 \int^{1}_{0} ds \underset{\R^{n}}{\int} K\eta^{2}
 [
 (p_{m+1})(w_{1,m}-s w_{1,\infty})\vert w_{1,m}-s w_{1,\infty} \vert^{p_{m}-1}w_{1,\infty}
 -
 w_{1,\infty}^{p_{m+1}}] dx \\
 &  \;\longrightarrow
 \int^{1}_{0} ds \underset{\R^{n}}{\int} K\eta^{2}
 [
 (p_{m+1})(1-s)^{p_{m}}w_{1,\infty}^{p_{m}+1}
 -
 w_{1,\infty}^{p_{m+1}}] dx =0.
 \end{split}
 \end{equation*}
 Hence \eqref{blow_up_control_critical_estimate} is justified and we obtain as before
 \begin{equation*}
 \underset{\R^{n}}{\int} \vert \nabla ((w_{1,m}-w_{1,\infty})\eta)\vert^{2}
 (1-C\Vert w_{1,m}-w_{1,\infty}\Vert_{L^{p_{m+1}}(supp(\eta))}^{p_{m}-1})dx
 \leq 
 o(1).
 \end{equation*}
 Thus $w_{1,m}\longrightarrow  w_{1,\infty}$ locally strongly, unless $w_{1,m}$ concentrates in $L^{p_{m}+1}$, but
 by our choice of $\Lambda_{1,m}$
 \begin{equation*}
 \varepsilon
 = 
 \sup_{x\in M}\int_{B_{\frac{1}{\Lambda_{1,m}}}(x)}\vert \nabla u_{1,m}\vert_{g_{0}}^{2}d\mu_{g_{0}}
 \geq
 \sup_{x\in B_{c\lambda_{1,m}}(0)\subset \R^{n}}\int_{B_{\frac{\lambda_{1,m}}{\Lambda_{1,m}}}(x)}\vert \nabla w_{1,m}\vert^{2}dx
 \end{equation*}
 and 
 $1\geq \frac{\lambda_{1,m}}{\Lambda_{1,m}}\not \hspace{-3pt}\longrightarrow 0$,
 so the $L^{2}$-gradient norm does not concentrate beyond $\varepsilon$ and, since
 \begin{equation*}
 -c_{n}\Delta_{\R^{n}}w_{1,m}-\frac{r_{\infty}K(a_{1,m})}{\lambda_{1,m}^{\theta_{m}}}w_{1,m}^{p_{m}}=o(1) 
 \; \text{ locally strongly in }  W^{-1,2}_{loc}(\R^n),
 \end{equation*}
 neither the $L^{p_{m}+1}$-norm does.
 Thus $w_{1,m}\longrightarrow w_{1,\infty}$ locally strongly. In particular
 \begin{equation*}
 \int_{B_{1}(0)}\vert \nabla w_{1,\infty}\vert^{2}dx\longleftarrow\int_{B_{\frac{1}{\lambda_{1,m}}}(a_{1,m})}\vert \nabla u_{1,m}\vert_{g_{0}}^{2}d\mu_{g_{0}}
 \geq 
 c=c(\varepsilon,\delta).
 \end{equation*}
 But $\sigma_{1}=0$ implies $w_{1,\infty}=0$ by harmonicity, so $\sigma_{1}\in (0,1]$, cf. \eqref{blow_up_control_critical_estimate}, and we easily show $w_{1,\infty} > 0$ and
 \begin{equation*}
 w_{1,\infty}=\alpha_{1} \bigg(\frac{\tilde \lambda_{1} }{1+\tilde \lambda_{1}^{2}r_{a}^{2}}\bigg)^{\frac{n-2}{2}} \; \text{ with }\;
 \alpha_{1}>0,\; r_{a}=\vert x-a\vert,\; a\in \R^{n}\; \text{ and }\; \tilde \lambda_{1}>0.
 \end{equation*}
Note that $-\Delta_{\R^{n}}w_{1,\infty}=\sigma_{1} r_{\infty}\kappa_{1} w_{1,\infty}^{\frac{n+2}{n-2}}$ implies
 $
 \sigma_{1} r_{\infty}\kappa_{1}\alpha_{1}^{\frac{4}{n-2}}=4n(n-1).
 $
 Moreover by construction 
 \begin{equation*}
 \int_{B_{1}(0)}\vert \nabla w_{1,m}\vert^{2}dx
 \geq \sup_{x\in B_{c\lambda_{1,m}}(0)}\int_{B_{1}(x)}\vert \nabla w_{1,m}\vert^{2}dx,
 \end{equation*}
 which transfers to $w_{1,\infty}$ by locally strong convergence. This implies  $a=0$ and
 \begin{equation*}
 \frac{\tilde \lambda_{1}^{n}}{1+\tilde \lambda_{1}^{n}}
 \sim
 \int_{B_{1}(0)} \bigg\vert \nabla \bigg(\frac{\tilde \lambda_{1}}{1+\tilde \lambda_{1}^{2}r^{2}} \bigg)^{\frac{n-2}{2}}\bigg\vert^{2}
 dx =
 \varepsilon\alpha_{1}^{-2}
 =
 \varepsilon(\sigma_{1} r_{\infty}\kappa_{1})^{\frac{n-2}{2}}.
 \end{equation*}
 By  $\underset{m\to \infty }{\lim}\lambda_{1,m}^{-\theta_{m}}=\sigma_{1} \in (0,1]$ and $0<\varepsilon \ll 1$ we get
 $\tilde \lambda_{1} \sim\underset{m\to \infty }{\lim}\lambda_{1,m}^{\frac{2-n}{2n}\theta_{m}}$. 
 Dilating back we may then write
 \begin{equation*}
 u_{m}
 = 
 u_{\infty}+\alpha_{1}\varphi_{1,m}+u_{2,m}
 , \quad 
 \; \varphi_{1,m}=\varphi_{a_{1,m},\bar\lambda_{1,m}}, \quad \bar \lambda_{1,m}=\tilde \lambda_{1}\lambda_{1,m}. 
 \end{equation*}
 Moreover we know that $u_{2,m}\rightharpoonup 0$ weakly in $W^{1,2}(M,g_0)$ and
 \begin{equation*}
 w_{2,m}=(\bar \lambda_{1,m})^{\frac{2-n}{2}}u_{2,m}\bigg(\exp_{g_{a_{1,m}}}\frac{\cdot}{\bar \lambda_{1,m}}\bigg) \longrightarrow 0
 \; \text{ locally strongly in } \; W^{1,2}(\R^{n}).
 \end{equation*}
 Since the initial sequence $(u_m)$ was non-negative, it follows that $u_\infty \geq 0$ and  the negative part 
 of $u_{2,m}$ tends to zero as $m \longrightarrow \infty$ in $W^{1,2}$-norm.  
 Using a dilation argument, the latter property and the above formula,
 it is easy to show that, if $\a, \b \geq 1$ with $\a + \b = \frac{2n}{n-2}$, then 
 \begin{equation}\label{eq:ineq-a-b}
 \int \varphi_{1,m}^\a |u_{2,m}|^\beta d \mu_{g_{0}} \longrightarrow 0 \quad \hbox{ as  } 
 m \longrightarrow \infty, 
 \end{equation}
 and that also 
 $\int u_{2,m} L_{g_{0}}\varphi_{1,m}d\mu_{g_{0}}=o(1)$. Thence as before for $u_{1,m}$
 \begin{equation*}
 r_{\infty}
 \longleftarrow 
 J_{\tau_{m}}(u_{m})
 =
 \int u_{\infty} L_{g_{0}}u_{\infty}d\mu_{g_{0}}
  +
 \alpha_{1}^{2}\int \varphi_{1,m} L_{g_{0}}\varphi_{1,m}d\mu_{g_{0}} 
  +
 \int u_{2,m} L_{g_{0}}u_{2,m}d\mu_{g_{0}}
 \end{equation*}
 and therefore $ \int u_{2,m} L_{g_{0}}u_{2,m} d \mu_{g_{0}} \longrightarrow r_{2,\infty}\geq 0$.  Likewise
 \begin{equation*}
 E(u_{2,m})=L_{g_{0}}u_{2,m}-r_{\infty}Ku_{2,m}^{p_{m}}=o(1) \quad \hbox{ in }  W^{-1,2}_{loc}(\R^n)
 \end{equation*}
 since by expansion of the non-linear term of $\partial J_{\tau_{m}}(u_{m})$ we find
 \begin{equation*}
 \begin{split}
 o(1)
 = &
 L_{g_{0}}(u_{\infty}+\alpha_{1}\varphi_{1,m}+u_{2,m})
 -
 r_{\infty}K(u_{\infty}+\alpha_{1}\varphi_{1,m}+u_{2,m})^{p_{m}} \\
 = &
 L_{g_{0}}u_{\infty} - r_{\infty}Ku_{\infty}^{p_{m}}
 +
 \alpha_{1} L_{g_{0}}\varphi_{1,m} - r_{\infty}K\alpha_{1}^{p_{m}}\varphi_{1,m}^{p_{m}} \\
 & +
 L_{g_{0}}u_{2,m} -r_{\infty}Ku_{2,m}^{p_{m}}+o(1)
 = 
 L_{g_{0}}u_{2,m} -r_{\infty}Ku_{2,m}^{p_{m}}+o(1) \quad \hbox{ in } W^{-1,2}(M,g_0). 
 \end{split}
 \end{equation*}
 The second equality follows from applying the latter formulas to any test function 
 in $W^{1,2}(M,g_0)$ and then applying Sobolev's and H\"older's inequalities together 
 with \eqref{eq:ineq-a-b}. 
 We may therefore iterate the afore going and find for a finite sum
 $ 
 u_{m}=\sum_{i}\alpha_{i}\varphi_{i,m}+v_{m},
 $ 
 with energy
 \begin{equation*}
 r_{\infty}\longleftarrow J(u_{m})
 \geq 
 \int u_{\infty} L_{g_{0}}u_{\infty}d\mu_{g_{0}}
 +
 \sum_{i}\alpha_{i}^{2}\int \varphi_{i,m} L_{g_{0}}\varphi_{i,m}d\mu_{g_{0}}. 
 \end{equation*}
But all $\alpha_{i}$ are uniformly lower bounded due to

 \begin{equation*}
 \sigma_{i} r_{\infty}\kappa_{i}\alpha_{i}^{\frac{4}{n-2}}=1,\sigma_{i}=\lim_{m\to \infty }\lambda_{i,m}^{-\theta_{m}}\in (0,1]
 \;\text{ and }\;
 \kappa_{i}=\lim_{m\to \infty }K(a_{i,m}),
 \end{equation*}
 thence the iteration has to stop after finitely-many steps. In particular 
 $v_{m}$ does not concentrate locally and consequently vanishes strongly as $m\longrightarrow \infty$. Now take any fixed index $j$ and recall that 
 \begin{equation*}
 w_{j,m}=\bar \lambda_{j,m}^{\frac{2-n}{2}}u_{j,m} \bigg(\exp_{g_{a_{j,m}}}\frac{\cdot}{\bar \lambda_{j,m}}\bigg)
 \end{equation*}
 and that by construction $\frac{\bar \lambda_{k,m}}{\bar \lambda_{l,m}}\not \hspace{-3pt} \longrightarrow 0 $ for $k<l$. We had seen
 \begin{equation*}
 w_{j,m}\longrightarrow w_{j,\infty} 
 \; \text{ weakly and locally strongly, where }\;
 -c_{n}\Delta w_{j,\infty}- \sigma_{j} r_{\infty}\kappa_{j}w_{j,\infty}^{\frac{n+2}{n-2}}=0.
 \end{equation*}
 On the other hand 
 \begin{equation*}
 w_{j,m}=\alpha_{j}\bigg(\frac{1}{1+r^{2}}\bigg)^{\frac{n-2}{2}}
 +
 \sum_{i>j}u_{a_{i,m}}(a_{j})\alpha_{i}
 \bigg(
 \frac
 {\frac{\bar \lambda_{i,m}}{\bar \lambda_{j,m}}}
 {1+\bar \lambda_{i,m}^{2}\gamma_{n}G_{a_{i,m}}^{\frac{2}{2-n}}\big(\exp_{g_{a_{j,m}}}\frac{\cdot}{\bar \lambda_{j,m}}\big)}
 \bigg)^{\frac{n-2}{2}}
 \end{equation*}
 up to some error of order $o(1)$ locally in $W^{1,2}$,
and the latter sum has to vanish, which is equivalent to
 \begin{equation*}
 \frac{\bar \lambda_{j,m}}{\bar \lambda_{i,m}}\longrightarrow \infty
 \; \text{ or }
 \
 \bar \lambda_{i,m}\bar \lambda_{j,m}G_{a_{i,m}}(a_{j,m})\longrightarrow \infty.
 \end{equation*}
 Recalling \eqref{eq:eij}, this shows  that $(\varepsilon_{i,j})_m \longrightarrow 0$ for all $i \neq j$. We are left with proving
 $\bar \lambda_{i,m}^{\tau_{m}}\longrightarrow 1$. Ordering 
 \begin{equation*}
  \bar \lambda_{1,m}\geq \ldots \geq \bar \lambda_{q,m} 
 \end{equation*}
 up to a subsequence, let
 \begin{equation*}
 1\leq \bar q=\sharp \bigg\{l=1,\ldots, q\mid \lim_{m\to \infty }\frac{\bar \lambda_{1,m}}{\bar\lambda_{l,m}} <\infty\bigg\}.
 \end{equation*}
 Then $\frac{\bar \lambda_{k,m}}{\bar \lambda_{l,m}}\longrightarrow \infty$ for $k\leq \bar q <l$ and 
 $c\leq \lim_{m\to \infty }\frac{\bar \lambda_{k,m}}{\bar \lambda_{l,m}}\leq C$ for $k,l\leq \bar q$. Select 
 a half-ball $B^{+}_{\delta}(a_{k,m})$ with
 \begin{equation*}
 1\leq k\leq \bar q \text{ and } 0<\delta\ll 1 
 \; \hbox{ such that } \;
 B^{+}_{\delta}(a_{k,m})\cap \{a_{l,m}\mid 1\leq l \leq \bar q,  l\neq k\} = \emptyset
 \end{equation*}
 up to a subsequence, where for some affine function  $\nu_{k,m}$ with unit gradient we have set 
 \begin{equation*}
 B^{+}_{\delta}(a_{k,m})=B_{\delta}(a_{k,m}) \cap \{ \nu_{k,m}  >0\}  
 \end{equation*}
 in a local coordinate system. Then rescaling $u_{m}$ on $B_{\delta}^{a_{k,m}}\cap \{\nu_{k,m} >\frac{1}{\bar \lambda_{k,m}} \}$ we find 
 \begin{equation*}
 w_{k,m}=\bar \lambda_{k,m}^{\frac{2-n}{2}}u_{m}\bigg(\exp_{g_{a_{i}}}\frac{\cdot}{\bar \lambda_{k,m}}\bigg)
 =
 \alpha_{l}\bigg(\frac{1}{1+r^{2}}\bigg)^{\frac{n-2}{2}}
 +
 o(1)
 \; \text{ on }\; B_{c\bar \lambda_{k,m}}(0)\cap \{x_{1}>1\}.
 \end{equation*}
 On the other hand side, $w_{k,m}$ solves
 \begin{equation*}
 -c_{n}\Delta w_{k,m}-\frac{r_{\infty}\kappa_{k}}{\bar \lambda_{k,m}^{\theta_{m}}}w_{k,m}^{p_{m}}=o(1),\;\kappa_{k}=\lim_{m\to \infty } K(a_{k,m})
 \; \text{ on }\; B_{c\bar \lambda_{k,m}}(0).
 \end{equation*}
 Recalling that $p_{m}=\frac{n+2}{n-2}-\tau_{m}$ and $\theta_{m} =\frac{n-2}{2}\tau_{m}$, this implies, that  up to rotating coordinates
 \begin{equation*}
 (1+r^{2})^{\theta_{m}}  \;\text{ is nearly constant on }\;B^{+}_{c\bar \lambda_{k,m}}(0)\cap \{x_{1}>1\}. 
 \end{equation*}
 Thus $\bar \lambda_{k,m}^{\theta_{m}}\longrightarrow 1$. The claim follows, since
 $\lim_{m\to \infty }\frac{\bar \lambda_{k,m}}{\bar \lambda_{l,m}}\geq c$ for all $l=1,\ldots,q.$
 \end{proof}

 \section{Reduction and v-part estimates}\label{s:red}

 In this section we will consider a sequence $u_m$ as in Proposition \ref{blow_up_analysis},  
 with zero weak limit. We will recall some well-known facts about finite-dimensional 
 reductions and derive preliminary error estimates and on suitable components 
 of the gradient of $J_{\tau}$.

 For $\varepsilon>0,\; q\in \N,\; u\in W^{1,2}(M,g_{0})$ and $( \alpha^{i}, \lambda_{i}, a_{i})\in (\R^{q}_{+}, \R^{q}_{+},M^{q})$ we  define 
 \begin{enumerate}[label=(\roman*)]
  \item \quad 
  $
 A_{u}(q, \varepsilon)
 =
 \{ 
 ( \alpha^{i}, \lambda_{i}, a_{i}) \mid
  \;
 \underset{i\neq j}{\forall\;}\;
  \lambda_{i}^{-1}, \lambda_{j}^{-1}, \varepsilon_{i,j}, \bigg\vert 1-\frac{r\alpha_{i}^{\frac{4}{n-2}}K(a_{i})}{4n(n-1)k_{\tau}}\bigg\vert,
 \Vert u-\alpha^{i}\varphi_{a_{i}, \lambda_{i}}\Vert
 <\varepsilon, \, \lambda_{i}^\tau < 1 + \varepsilon
 \};
  $
 \item \quad
 $ 
 V(q, \varepsilon)
 = 
 \{
 u\in W^{1,2}(M,g_{0})  
 \mid
 A_{u}(q, \varepsilon)\neq \emptyset
 \},
 $
 \end{enumerate}
 cf.  \eqref{eq:r}, \eqref{eq:kp} and \eqref{eq:bubbles}. For both  conditions $\lambda_{i} > \varepsilon^{-1}, \lambda_{i}^\tau < 1 + \varepsilon$ to hold, we 
 will always assume that $\tau \ll \varepsilon$ and this is  consistent with the 
 statement of Proposition \ref{blow_up_analysis}. 
 Under the above conditions on the parameters $\alpha_{i}, a_i$ and $\lambda_{i}$ the functions 
 $\sum_{i=1}^q \alpha^{i}\varphi_{a_{i}, \lambda_{i}}$ form a smooth manifold in $W^{1,2}(M,g_0)$, 
 which implies the following well known result, cf. \cite{bab}.

 \begin{proposition} 
 \label{prop_optimal_choice} 
 For every $\varepsilon_{0}>0$  there exists $\varepsilon_{1}>0$ such that  for $u\in V( q, \varepsilon)$ 
 with $\varepsilon<\varepsilon_{1}$
 \begin{equation*}\begin{split}
 \inf_
 {
 (\tilde\alpha_{i}, \tilde a_{i}, \tilde\lambda_{i})\in A_{u}(q,2\varepsilon_{0}) 
 }
 \int 
 (
 u
 -
 \tilde\alpha^{i}\varphi_{\tilde a_{i}, \tilde \lambda_{i}}
 ) L_{g_{0}}
 (
 u
 -
 \tilde\alpha^{i}\varphi_{\tilde a_{i}, \tilde \lambda_{i}}
 )
 d\mu_{g_{0}}
 \end{split}\end{equation*}
 admits an unique minimizer $(\alpha_{i},a_{i}, \lambda_{i})\in A_{u}(q, \varepsilon_{0})$ 
 depending smoothly on $u$
 and we set 
 \begin{equation}\begin{split} \label{eq:v}
 \varphi_{i}=\varphi_{a_{i}, \lambda_{i}}, \quad v=u-\alpha^{i}\varphi_{i}, \quad K_{i}=K(a_{i}).
 \end{split}\end{equation}
 \end{proposition} 
\noindent
 The  term  
 $v=u-\alpha^{i}\varphi_{i}$ is orthogonal to all  
 $ 
 \varphi_{i} ,-\lambda_{i}\partial_{\lambda_{i}}\varphi_{i}, \frac{1}{\lambda_{i}}\nabla_{a_{i}}\varphi_{i}, 
 $ with respect to the product 
 \begin{equation*}\begin{split}
 \langle \cdot, \cdot\rangle_{L_{g_{0}}}
 =
 \langle L_{g_{0}}\cdot,\cdot\rangle_{L^{2}_{g_{0}}}. 
 \end{split}\end{equation*}
For $u\in V( q, \varepsilon)$ let
 \begin{equation} 
 \label{eq:Hu}
 H_{u}( q, \varepsilon)
 =
 \langle 
  \varphi_{i},\lambda_{i}\partial_{\lambda_{i}}\varphi_{i}, \frac{1}{\lambda_{i}}\nabla_{a_{i}}\varphi_{i}
 \rangle
 ^{\perp_{L_{g_{0}}}}.
 \end{equation}
 We next have an estimate on the projection of the gradient of $J_{\tau}$ onto $H_u$.

 \begin{lemma} 
 \label{lem_testing_with_v} 
 For $u\in V(q,\varepsilon)$ with $k_{\tau}=1$, cf. \eqref{eq:kp},and $\nu\in H_{u}(q,\varepsilon)$ there holds
 \begin{equation*}
 \partial J_{\tau}(\alpha^{i}\varphi_{i})\nu
 = 
 O\bigg(
 \bigg[
 \sum_{r}\frac{\tau}{\lambda_{r}^{\theta}}
 +
 \sum_{r}\frac{\vert \nabla K_{r}\vert}{\lambda_{r}^{1+\theta}}
 +
 \sum_{r}\frac{1}{\lambda_{r}^{2+\theta}}
 +
 \sum_{r}\frac{1}{\lambda_{r}^{n-2+\theta}}
 +
 \sum_{r\neq s}\frac{\varepsilon_{r,s}^{\frac{n+2}{2n}}}{\lambda_{r}^{\theta}}
 \bigg]
 \Vert \nu \Vert \bigg).  
 \end{equation*}
 \end{lemma}
 \begin{proof}
 Due to the fact that $k_{\tau}=1$ and $\nu\in H_{u}(q,\varepsilon)$ we have
 \begin{equation*}
 \begin{split}
 -\frac{1}{2}\partial J_{\tau}(\alpha^{i}\varphi_{i})\nu
 = &
 r_{\alpha^{i}\varphi_{i}} \int K(\alpha^{i}\varphi_{i})^{p}\nu d\mu_{g_{0}},
 \end{split}
 \end{equation*}
 and therefore
 \begin{equation*}
 \partial J_{\tau}(\alpha^{i}\varphi_{i})\nu
 \simeq 
 \int K(\alpha^{i}\varphi_{i})^{p}\nu d\mu_{g_{0}}.
 \end{equation*}
 Decomposing iteratively $M$ as  
 $ 
 \big\{ \alpha_{j}\varphi_{j}>\sum_{i>j}\alpha_{i}\varphi_{i} \big\} 
 \cup 
 \big\{ \alpha_{j}\varphi_{j}\leq \sum_{i>j}\alpha_{i}\varphi_{i}\big\},
 $ 
 we find
 \begin{equation*}
 \begin{split}
 \int K(\alpha^{i}\varphi_{i})^{p}\nu d\mu_{g_{0}}
 = &
 \sum_{i}\int K(\alpha_{i}\varphi_{i})^{p}\nu d\mu_{g_{0}} 
 +
 O(\sum_{r\neq s}
 \int_{\{\alpha_{s}\varphi_{s}\leq \alpha_{r}\varphi_{r}\}}
 (\alpha^{r}\varphi_{r})^{p-1}\alpha^{s}\varphi_{s}\vert \nu \vert d\mu_{g_{0}}).
 \end{split}
 \end{equation*}
 Using H\"older's inequality  with exponents 
 $ 
 1=\frac{1}{p}+\frac{1}{q}=\frac{n+2}{2n}+\frac{n-2}{2n} 
 $ 
 and Lemma \ref{lem_interactions} $(v)$ applied to the latter error term, where 
 the inequality $\varphi_{s} \lesssim \varphi_{r}$ can be used to apply it with $\beta \geq 1$,
  we get
 \begin{equation*}
 \int K(\alpha^{i}\varphi_{i})^{p}\nu d\mu_{g_{0}}
 =
 \sum_{i}\int K(\alpha^{i}\varphi_{i})^{p}\nu d\mu_{g_{0}}
 +
 O\bigg(\sum_{r\neq s}\frac{\varepsilon_{r,s}^{\frac{n+2}{2n}}}{\lambda_{r}^{\theta}}\Vert \nu \Vert\bigg),
 \end{equation*}
 and by a simple expansion we also obtain 
 \begin{equation}\label{non_linearity_expansion}
 \int K(\alpha^{i}\varphi_{i})^{p}\nu d\mu_{g_{0}}
 = 
 \sum_{i}K_{i}\alpha_{i}^{p}\int \varphi_{i}^{p}\nu d\mu_{g_{0}} 
 +
 O\bigg(
 \bigg[
 \sum_{r}\frac{\vert \nabla K_{r}\vert}{\lambda_{r}^{1+\theta}}
 +
 \sum_{r}\frac{1}{\lambda_{r}^{2+\theta}}
 +
 \sum_{r\neq s}\frac{\varepsilon_{r,s}^{\frac{n+2}{2n}}}{\lambda_{r}^{\theta}}
 \bigg]
 \Vert \nu \Vert\bigg).
 \end{equation}
Note that 
 \begin{equation*}
 \begin{split}
 \Vert \lambda_{i}^{-\theta}\varphi_{i}^{\frac{n+2}{n-2}} & -\varphi_{i}^{p}\Vert_{L_{g_{0}}^{\frac{2n}{n+2}}}^{\frac{2n}{n+2}}
 = 
 \int 
 \varphi_{i}^{\frac{2n}{n-2}-\frac{2n}{n+2}\tau}\vert 1-\lambda_{i}^{-\theta}\varphi_{i}^{\tau}\vert^{\frac{2n}{n+2}}d\mu_{g_{0}}\\
 \lesssim &
 \int_{B_{c}(0)}
 \bigg(\frac{\lambda_{i}}{1+\lambda_{i}^{2}r^{2}}\bigg)^{n-\frac{2n}{n+2}\theta}
 \bigg\vert 1-\bigg(\frac{1}{1+\lambda_{i}^{2}O(r^{2})}\bigg)^{\theta}\bigg\vert^{\frac{2n}{n+2}} dx
 +
 O\bigg(\frac{1}{\lambda_{i}^{n-\frac{2n}{n+2}\theta}}\bigg)
 \\
 = &
 \lambda_{i}^{-\frac{2n}{n+2}\theta}
 \int_{B_{c\lambda_{i}}(0) }
 \bigg(\frac{1}{1+r^{2}}\bigg)^{n-\frac{2n}{n+2}\theta}\bigg\vert 1-\bigg(\frac{1}{1+O(r^{2})}\bigg)^{\theta}\bigg\vert^{\frac{2n}{n+2}} dx
 +
 O\bigg(\frac{1}{\lambda_{i}^{n-\frac{2n}{n+2}\theta}}\bigg),
 \end{split}
 \end{equation*}
 whence 
 \begin{equation}\label{bubble_critical_minus_subcritical}
 \begin{split}
 \Vert \lambda_{i}^{-\theta}\varphi_{i}^{\frac{n+2}{n-2}} & -\varphi_{i}^{p}\Vert_{L_{g_{0}}^{\frac{2n}{n+2}}}
 = 
 O\bigg(\frac{\theta}{\lambda_{i}^{\theta}} + \frac{1}{\lambda_{i}^{n-\frac{2n}{n+2}\theta}}\bigg).
 \end{split}
 \end{equation}
Thus up to some 
$
 O(
[
 \sum_{r}\frac{\tau}{\lambda_{r}^{\theta}}
 +
 \sum_{r}\frac{\vert \nabla K_{r}\vert}{\lambda_{r}^{1+\theta}}
 +
 \sum_{r}\frac{1}{\lambda_{r}^{2+\theta}}
 +
 \sum_{r\neq s}\frac{\varepsilon_{r,s}^{\frac{n+2}{2n}}}{\lambda_{r}^{\theta}}
]
 \Vert \nu \Vert)
$
we arrive at 
 \begin{equation*}
 \begin{split}
 \int K(\alpha^{i}\varphi_{i})^{p}\nu d\mu_{g_{0}}
 = &
 \sum_{i}K_{i}\lambda_{i}^{-\theta}\alpha_{i}^{\frac{n+2}{n-2}}\int \varphi_{i}^{\frac{n+2}{n-2}}\nu d\mu_{g_{0}}.
 \end{split}
 \end{equation*}
 Finally from Lemma \ref{lem_emergence_of_the_regular_part} and the fact that $\nu \in H_{u}(q,\varepsilon)$ (hence $\int \nu L_{g_0} \varphi_i 
 d \mu_{g_{0}} = 0$) we obtain 
 \begin{equation}\label{non_linear_v_part_interaction}
 \bigg\vert \int \varphi_{i}^{\frac{n+2}{n-2}}\nu d\mu_{g_{0}}\bigg\vert 
 \leq 
 \Vert v \Vert \bigg\Vert \frac{L_{g_{0}}\varphi_{i}}{4n(n-1)}-\varphi_{i}^{\frac{n+2}{n-2}}\bigg\Vert_{L_{g_{0}}^{\frac{2n}{n+2}}} 
 = 
 O
 \begin{pmatrix}
 \lambda_{i}^{-1}& \;\text{ for }\; n=3  \\
 \lambda_{i}^{-2}&\;\text{ for }\; n=4  \\
 \lambda_{i}^{-3} & \;\text{ for }\; n=5  \\
 \ln^{\frac{2}{3}}\lambda_{i}\lambda_{i}^{-\frac{10}{3}} & \;\text{ for }\; n=6  \\
 \lambda_{i}^{-4} &\;\text{ for }\; n\geq 7  
 \end{pmatrix}
 \Vert v \Vert, 
 \end{equation}
 so the claim follows. 
 \end{proof}
 
 \begin{lemma} 
 \label{lem_v_part_gradient} 
 For $u\in V(q,\varepsilon)$ with $k_{\tau}=1$ and $v$ is as in \eqref{eq:v} there holds
 \begin{equation*}
 \Vert v \Vert
=
 O\bigg(
 \sum_{r}\frac{\tau}{\lambda_{r}^{\theta}}
 +
 \sum_{r}\frac{\vert \nabla K_{r}\vert}{\lambda_{r}^{1+\theta}}
 +
 \sum_{r}\frac{1}{\lambda_{r}^{2+\theta}}
 +
 \sum_{r}\frac{1}{\lambda_{r}^{n-2+\theta}}
 +
 \sum_{r\neq s}\frac{\varepsilon_{r,s}^{\frac{n+2}{2n}}}{\lambda_{r}^{\theta}}
 +
 \vert \partial J_{\tau}(u)\vert
 \bigg).
 \end{equation*}
  \end{lemma}
 \begin{proof}
 Since the Hessian of $J_\tau$ is uniformly H\"older continuous on bounded 
 sets of $W^{1,2}$, we have 
 \begin{equation*}
 \partial J_{\tau}(u)v=\partial J_{\tau}(\alpha^{i}\varphi_{i})v+\partial^2 J_{\tau}(\alpha^{i}\varphi_{i})v^{2}+o(\Vert v \Vert^{2})
 =
 \partial J_{\tau}(\alpha^{i}\varphi_{i})v+\partial^2 J_{\tau}(u)v^{2}+o(\Vert v \Vert^{2});
 \end{equation*}
 \begin{equation}  \label{eq:d2Jtau}
 \begin{split}
 \partial^{2} J_{\tau}(u) v^{2}
 = & 
 2
 \bigg[\int v L_{g_{0}}v d\mu_{g_{0}}-p r_u Ku^{p-1}v^2d\mu_{g_{0}}\bigg] 
 -
 8
 \int u L_{g_{0}}vd\mu_{g_{0}}\int Ku^{p}vd\mu_{g_{0}} \\
 & + 
 2(p+3)r
 \int Ku^{p}vd\mu_{g_{0}}\int Ku^{p}vd\mu_{g_{0}}. 
 \end{split}
 \end{equation}
 Since $v\in H_{u}(q,\varepsilon)$, by similar expansions we then find (also  replacing $p$ with $\frac{n+2}{n-2}$ 
 with an error $o(1)$)
 \begin{equation*}
 \begin{split}
 \partial^{2} & J_{\tau}(u) v^{2}
 = 
 2
 \bigg[\int v L_{g_{0}}vd\mu_{g_{0}}-pr_u\int Ku^{p-1}v^2d\mu_{g_{0}}\bigg] 
 \\
 = &
 2
 \bigg[\int vL_{g_{0}}vd\mu_{g_{0}}-\frac{n+2}{n-2}\big(\int (\alpha^{i}\varphi_{i}) L_{g_{0}}(\alpha^{j}\varphi_{j})d\mu_{g_{0}}\big)\int K(\alpha^{i}\varphi_{i})^{p-1}v^2d\mu_{g_{0}}\bigg] 
 \\
 = &
 2
 \bigg[\int v L_{g_{0}}vd\mu_{g_{0}}-\frac{n+2}{n-2}\sum_{i,j}
 \frac{K_{i}\alpha_{i}^{\frac{4}{n-2}}\alpha_{j}^{2}\int \varphi_{j} L_{g_{0}}\varphi_{j}d\mu_{g_{0}}}{\lambda_{i}^{\theta}}\int \varphi_{i}^{\frac{4}{n-2}}v^2d\mu_{g_{0}}\bigg]
 \end{split}
 \end{equation*}
 up to some $o(\Vert v \Vert^{2})$. 
 Furthermore by definition of $V(q,\varepsilon)$ there holds $\lambda_{i}^{\theta}=1+o(1)$ and 
 \begin{equation*}
 K_{i}\alpha_{i}^{\frac{4}{n-2}}=\frac{1}{r_{\alpha^{i}\varphi_{i}}}+o(1)
 =
 \frac{1}{\int \sum_{j}\alpha_{j}^{2}\varphi_{j}L_{g_{0}}\varphi_{j}d\mu_{g_{0}}}+o(1).
 \end{equation*}
 Thus 
 \begin{equation*}
 \begin{split}
 \partial^{2} J_{\tau}(u) v^{2}
 = &
 2
 \bigg[\int vL_{g_{0}}vd\mu_{g_{0}}-\frac{n+2}{n-2}
 \int \varphi_{i}^{\frac{4}{n-2}}v^2d\mu_{g_{0}}\bigg]
 +
 o(\Vert v \Vert^{2}).
 \end{split}
 \end{equation*}
 This quadratic form is positive definite for $\varepsilon$ sufficiently small on the subspace 
 $v$ belongs to, cf. \cite{bab},  so
 \begin{equation*}
 \Vert v \Vert^{2}(1+o(1))
 \leq 
 C \partial^{2} J_{\tau}(u)v^{2}
 \leq 
 C[\partial J_{\tau}(\alpha^{i}\varphi_{i})v+\vert \partial J_{\tau}(u)\vert^{2}].
 \end{equation*}
 Therefore the claim follows from Lemma \ref{lem_testing_with_v}.
 \end{proof}

 We now establish cancellations testing the gradient of $J_{\tau}$ orthogonally to 
 $H_u(q,\varepsilon)$. 
 

 \begin{lemma} 
 \label{lem_v_part_interactions} 
 For $u\in V(q,\varepsilon)$ with $k_{\tau}=1$ the quantity $\partial J_{\tau}(u)\phi_{k,i}$ expands as 
 \begin{equation*}
 %
 \partial J_{\tau}(\alpha^{j}\varphi_{j})\phi_{k,i}
 +
 O\bigg(
 \sum_{r}
 \frac{\tau^{2}}{\lambda_{r}^{2\theta}}
 +
 \sum_{r}\frac{\vert \nabla K_{r}\vert^{2}}{\lambda_{r}^{2+2\theta}}
 +
 \sum_{r}\frac{1}{\lambda_{r}^{4+2\theta}} 
 +
 \sum_{r}\frac{1}{\lambda_{r}^{2(n-2)+2\theta}} 
 +
 \sum_{r\neq s}\frac{\varepsilon_{r,s}^{\frac{n+2}{n}}}{\lambda_{r}^{2\theta}}
 +
 \vert \partial J_{\tau}(u)\vert^{2}
 \bigg).
 \end{equation*}
 \end{lemma}
 \begin{proof}
 By the mean value theorem and \eqref{eq:d2Jtau} we have, with some $\sigma \in [0,1]$
 \begin{equation*}
 \begin{split}
 \partial J_{\tau}(u) & \phi_{k,i}  -\partial J_{\tau}(\alpha^{j}\varphi_{j})\phi_{k,i}
 = 
 \partial^{2} J_{\tau}(\alpha^{j}\varphi_{j}+\sigma v)\phi_{k,i}v \\
 = &
 2(1+O(\Vert v \Vert))
 \bigg[\int vL_{g_{0}}\phi_{k,i}d\mu_{g_{0}}
  -
 pr_{\alpha^{i}\varphi_{i}}(1+O(\Vert v \Vert))\int K(\alpha^{j}\varphi_{j}+\sigma v)^{p-1}v\phi_{k,i}d\mu_{g_{0}}\bigg] \\
 & -
 4(1+O(\Vert v \Vert))
 \bigg[
 \int (\alpha^{j}\varphi_{j}+\sigma v) L_{g_{0}}vd\mu_{g_{0}}\int K(\alpha^{j}\varphi_{j}+\sigma v)^{p}\phi_{k,i} d\mu_{g_{0}} \\
 & \quad\quad\quad\quad\quad\quad\quad\quad +
 \int (\alpha^{j}\varphi_{j}+\sigma v) L_{g_{0}}\phi_{k,i}d\mu_{g_{0}}\int K(\alpha^{j}\varphi_{j}+\sigma v)^{p}vd\mu_{g_{0}}
 \bigg] \\
 & +
 2(p+3)r_{\alpha^{i}\varphi_{i}}(1+O(\Vert v \Vert))
 \int K(\alpha^{j}\varphi_{j}+\sigma v)^{p}vd\mu_{g_{0}}
 \int K(\alpha^{j}\varphi_{j}+\sigma v)^{p}\phi_{k,i}d\mu_{g_{0}}. 
 \end{split}
 \end{equation*}
 Therefore, since $v\in H_{u}(q,\varepsilon)$, up to some $O(\Vert v \Vert^{2})$ we also get 
 \begin{equation*}
 \begin{split}
 & \partial J_{\tau}(u)\phi_{k,i}  -\partial J_{\tau}(\alpha^{j}\varphi_{j})\phi_{k,i} 
 = 
 -2
 pr_{\alpha^{i}\varphi_{i}}\int K(\alpha^{j}\varphi_{j}+\sigma v)^{p-1}v\phi_{k,i}d\mu_{g_{0}} \\
 & -
 4
 \int (\alpha^{j}\varphi_{j})L_{g_{0}}\phi_{k,i}d\mu_{g_{0}}\int K(\alpha^{j}\varphi_{j}+\sigma v)^{p}vd\mu_{g_{0}}
 \\
 & +
 2(p+3)r_{\alpha^{i}\varphi_{i}}
 \int K(\alpha^{j}\varphi_{j}+\sigma v)^{p}vd\mu_{g_{0}}
 \int K(\alpha^{j}\varphi_{j}+\sigma v)^{p}\phi_{k,i}d\mu_{g_{0}}.
 \end{split}
 \end{equation*}
 Decomposing now $M$ as 
 $ 
\{\alpha^{j}\varphi_{j}\leq 2\| v \|\}\cup\{\alpha^{j}\varphi_{j}\geq 2\| v \|\}, 
 $ 
 and using  $\vert \phi_{k,i}\vert \leq C \alpha_{i} \varphi_i \leq C\alpha^{j}\varphi_{j}$, we find
 \begin{equation*}
 \begin{split}
 & \partial J_{\tau}(u)\phi_{k,i}  -\partial J_{\tau}(\alpha^{j}\varphi_{j})\phi_{k,i} 
 = 
 -2
 pr_{\alpha^{i}\varphi_{i}}\int K(\alpha^{j}\varphi_{j})^{p-1}v\phi_{k,i}d\mu_{g_{0}} \\
 & -
 4
 \int (\alpha^{j}\varphi_{j}) L_{g_{0}}\phi_{k,i}d\mu_{g_{0}}\int K(\alpha^{j}\varphi_{j})^{p}vd\mu_{g_{0}}
 \\
 & +
 2(p+3)r_{\alpha^{i}\varphi_{i}}
 \int K(\alpha^{j}\varphi_{j})^{p}vd\mu_{g_{0}}
 \int K(\alpha^{j}\varphi_{j})^{p}\phi_{k,i}d\mu_{g_{0}} + O(\|v\|^2).
 \end{split}
 \end{equation*}
 Now, arguing as for \eqref{non_linearity_expansion} and using Lemma \ref{lem_interactions} $(iv)$, we have 
 \begin{equation*}
 \int K(\alpha^{j}\varphi_{j})^{p}v d\mu_{g_{0}}
 = 
 \sum_{j}K_{j}\alpha_{j}^{p}\int \varphi_{j}^{p}v d\mu_{g_{0}} 
  +
 O\bigg(
 \bigg[
 \sum_{r}\frac{\vert \nabla K_{r}\vert}{\lambda_{r}^{1+\theta}}
 +
 \sum_{r}\frac{1}{\lambda_{r}^{2+\theta}}
 +
 \sum_{r\neq s}\frac{\varepsilon_{r,s}^{\frac{n+2}{2n}}}{\lambda_{r}^{\theta}}
 \bigg]
 \Vert v \Vert\bigg);
 \end{equation*}
 \begin{equation*}
 \int K(\alpha^{j}\varphi_{j})^{p-1}\phi_{k,i}v d\mu_{g_{0}}
 = 
 K_{i}\alpha_{i}^{p-1}\int \varphi_{i}^{p-1}\phi_{k,i}v d\mu_{g_{0}} 
 +
 O\bigg(
 \bigg[
 \sum_{r}\frac{\vert \nabla K_{r}\vert}{\lambda_{r}^{1+\theta}}
 +
 \sum_{r}\frac{1}{\lambda_{r}^{2+\theta}}
 +
 \sum_{r\neq s}\frac{\varepsilon_{r,s}^{\frac{n+2}{2n}}}{\lambda_{r}^{\theta}}
 \bigg]
 \Vert v \Vert\bigg),
 \end{equation*}
 whence
 \begin{equation*}
 \begin{split}
 \partial J_{\tau}(u)\phi_{k,i}  -\partial J_{\tau}(\alpha^{j}\varphi_{j})\phi_{k,i} 
 = & 
 -2
 pr_{\alpha^{i}\varphi_{i}}K_{i}\alpha_{i}^{p-1}\int \varphi_{i}^{p-1}\phi_{k,i}vd\mu_{g_{0}} \\
 & -
 4\alpha_{i}
 \int \varphi_{i} L_{g_{0}}\phi_{k,i}d\mu_{g_{0}}
 \sum_{j}K_{j}\alpha_{j}^{p}
 \int \varphi_{j}^{p}vd\mu_{g_{0}}
 \\
 & +
 2(p+3)r_{\alpha^{i}\varphi_{i}}K_{i}\alpha_{i}^{p}\int \varphi_{i}^{p}\phi_{k,i}d\mu_{g_{0}}
 \sum_{j}K_{j}\alpha_{j}^{p}\int \varphi_{j}^{p}vd\mu_{g_{0}}
 \end{split}
 \end{equation*}
up to some 
$
 O\bigg(
 \sum_{r}\frac{\vert \nabla K_{r}\vert^{2}}{\lambda_{r}^{2+2\theta}}
 +
 \sum_{r}\frac{1}{\lambda_{r}^{4+2\theta}}
 +
 \sum_{r\neq s}\frac{\varepsilon_{r,s}^{\frac{n+2}{n}}}{\lambda_{r}^{2\theta}}
 +
 \Vert v \Vert^{2}
 \bigg).
$
Using \eqref{bubble_critical_minus_subcritical} and \eqref{non_linear_v_part_interaction} we arrive at
 \begin{equation*}
 \begin{split}
\partial  J_{\tau}(u) \phi_{k,i} & - \partial J_{\tau}(\alpha^{j}\varphi_{j})\phi_{k,i} 
 = 
 -2
 pr_{\alpha^{i}\varphi_{i}}K_{i}\alpha_{i}^{p-1}\int \varphi_{i}^{p-1}\phi_{k,i}vd\mu_{g_{0}} \\
 & +
 O\bigg(
 \sum_{r}\frac{\tau^{2}}{\lambda_{r}^{2\theta}}
 +
 \sum_{r}\frac{\vert \nabla K_{r}\vert^{2}}{\lambda_{r}^{2+2\theta}}
 +
 \sum_{r}\frac{1}{\lambda_{r}^{4+2\theta}}
 +
 \sum_{r}\frac{1}{\lambda_{r}^{2(n-2)+2\theta}} 
 +
 \sum_{r\neq s}\frac{\varepsilon_{r,s}^{\frac{n+2}{n}}}{\lambda_{r}^{2\theta}}
 +
 \Vert v \Vert^{2}
 \bigg).
 \end{split}
 \end{equation*}
 Yet also the first summand on the right hand side is of the same order as the second one, arguing as for 
 \eqref{bubble_critical_minus_subcritical} and \eqref{non_linear_v_part_interaction}.
 Combining this with Lemma \ref{lem_v_part_gradient}, we obtain the conclusion. 
 \end{proof}

  \section{The functional and its derivatives}\label{s:funct-infty}
  For $u\in V(q,\varepsilon)$ and $\varepsilon>0$ sufficiently small let
  \begin{equation}\label{eq:akt}
  \alpha^{2}=\sum_{i}\alpha_{i}^{2}, \quad  \alpha_{K,\tau}^{s}=\sum_{i}\frac{K_{i}}{\lambda_{i}^{\theta}}\alpha_{i}^{s}, 
 \quad \theta=\frac{n-2}{2}\tau.
  \end{equation}
  Recalling the notation from the 
  previous section we may 
expand the Euler-Lagrange energy as follows.  
  
  \begin{proposition}\label{prop_functional_at_infinity} 
  For $u=\alpha^{i}\varphi_{i}+v\in V(q,\varepsilon)$ and $\varepsilon>0$, 
  both $J_{\tau}(u)$ and $J_{\tau} ( \alpha^{i}  \varphi_{i}) $ can be written as 
  \begin{equation*}
  \frac
  {\hat c_{0}\alpha^{2}}
  {(\alpha_{K,\tau}^{p+1})^{\frac{2}{p+1}}}
  \Bigg(
  1
 -
  \hat c_{1}\tau 
  -
  \hat c_{2}\sum_{i}\frac{\Delta K_{i}}{K_{i}\lambda_{i}^{2}} \frac{\alpha_{i}^{2} }{\alpha^{2}} 
  -
  \hat b_{1}\sum_{i\neq j}\frac{\alpha_{i}\alpha_{j}}{\alpha^{2}}
  \varepsilon_{i,j}
  -
  \hat d_{1} \sum_{i}
  \frac{\alpha_{i}^{2}}{\alpha^{2}}
  \begin{pmatrix}
  \frac{H_{i}}{\lambda_{i}}  & \text{for }\; n=3\\
  \frac{H_{i}+O(\frac{\ln \lambda_{i}}{\lambda_{i}^{2}})}{\lambda_{i}^{2 }} & \text{for }\; n=4\\
  \frac{H_{i}}{\lambda_{i}^{3}} & \text{for }\; n=5\\
  \frac{W_{i}\ln \lambda_{i}}{\lambda_{i}^{4}} & \text{for }\; n=6\\
  0 & \text{for }\; n\geq 7
  \end{pmatrix}
  \Bigg) 
  \end{equation*}
  with positive constants $\hat c_{0},\hat c_{1},\hat c_{2},\hat b_{1},\hat d_{1}$ up to  errors of the form 
  \begin{equation*}
  O
  (
  \tau^{2}
  +
  \sum_{r}\frac{\vert \nabla K_{r}\vert^{2}}{\lambda_{r}^{2}}
  +
  \frac{1}{\lambda_{r}^{4}}+\frac{1}{\lambda_{r}^{2(n-2)}}
  +
  \sum_{r\neq s}\varepsilon_{r,s}^{\frac{n+2}{n}}
  +
  \vert \partial J_{\tau}(u)\vert^{2}
  ).
  \end{equation*}
  \end{proposition}
  
  \begin{proof}
 The above expansion  for $J_\tau(\alpha^i \varphi_i)$ implies the one for $J_\tau(u)$ via  Lemmata \ref{lem_testing_with_v} and \ref{lem_v_part_gradient} 
 expanding
  \begin{equation*}
  J_{\tau}(u)=J_{\tau}(\alpha^{i}\varphi_{i})+\partial J_{\tau}(\alpha^{i}\varphi_{i})v+O(\Vert v \Vert^{2}).
  \end{equation*}
  We next start analyzing $J_\tau(\alpha^i \varphi_i)$ from the denominator. Decomposing iteratively $M$ as 
  \begin{equation*}
  M
  =
  \{\alpha_{j}\varphi_{j}>\sum_{i>j}\alpha_{i}\varphi_{i}\}
  +  
  \{\alpha_{j}\varphi_{j}\leq \sum_{i>j}\alpha_{i}\varphi_{i}\}
  \end{equation*}
  we may expand
  \begin{equation*}
  \begin{split}
  \int K  (\alpha^{i}\varphi_{i})^{p+1} d\mu_{g_{0}}
  = &
  \sum_{i}\alpha_{i}^{p+1}\int K\varphi_{i}^{p+1} d\mu_{g_{0}}
  +
  (p+1)\sum_{i\neq j}\alpha_{i}^{p}\alpha_{j}\int K\varphi_{i}^{p}\varphi_{j} d\mu_{g_{0}}\\
  & +
  O
  \bigg(
  \sum_{r\neq s}\int_{\{\alpha_{r}\varphi_{r}\geq \alpha_{s}\varphi_{s}\}}(\alpha_{r}\varphi_{r})^{p}\alpha_{s}\varphi_{s}d\mu_{g_{0}}
  \bigg).
  \end{split}
  \end{equation*}
Recalling $\lambda_{r}^{\theta}\sim 1$ and the boundedness of $\alpha_{r}$ by  the definition of $V(q,\varepsilon)$, using Lemma \ref{lem_interactions} and reasoning as for the proof of Lemma \ref{lem_testing_with_v}, the latter term is of order 
  $ 
  O(\sum_{r\neq s}\varepsilon_{i,j}^{\frac{n+2}{n}}), 
  $ 
 and also  
  \begin{equation*}
  \begin{split}
  \int K\varphi_{i}^{p}\varphi_{j}d\mu_{g_{0}}
  = &
  K_{i}\int \varphi_{i}^{p}\varphi_{j}d\mu_{g_{0}}
  +
  O\bigg(\int_{B_{c}(a_{i})} (\vert \nabla K_{i}\vert r_{a_{i}}+r_{a_{i}}^{2})\varphi_i^{p}\varphi_{j}d\mu_{g_{0}}\bigg) 
 +
  O\bigg(\frac{1}{\lambda_{i}^{\frac{n+2}{n-2}-\theta}\lambda_{j}^{\frac{n-2}{2}}}\bigg) \\
  = &
  K_{i}\int \varphi_{i}^{p}\varphi_{j}d\mu_{g_{0}}
  +
  O
  \bigg(
  \sum_{r\neq s} \frac{\vert \nabla K_{r}\vert^{2}}{\lambda_{r}^{2}}
  +
  \frac{1}{\lambda_{r}^{4}}
  +
  \varepsilon_{r,s}^{\frac{n+2}{n}}
  \bigg).
  \end{split}
  \end{equation*}
  Indeed we for example have 
  \begin{equation*}
  \int_{B_{c}(a_{i})}r_{a_{i}}\varphi_{i}^{p}\varphi_{j}d\mu_{g_{0}}
  = 
  \int_{B_{c}(a_{i})}r_{a_{i}}\varphi_{i}^{\frac{n+2}{n-2}-\frac{n+2}{2n}}\varphi_{i}^{\frac{n+2}{2n}}\varphi_{j}d\mu_{g_{0}} 
  \leq 
  C\varepsilon_{i,j}^{\frac{n+2}{2n}}
  \Vert r\varphi_{0,\lambda_{i}}^{\frac{n+2}{n-2}-\frac{n+2}{2n}}\Vert_{L_{\mu_{g_{0}}}^{(\frac{2n}{(n+2)})^{2}}}
  \end{equation*}
 with the latter norm that can be controlled by 
  \begin{equation*}
  \underset{\R^{n}}{\int} r^{(\frac{2n}{n+2})^{2}}(\frac{\lambda_{i}}{1+\lambda_{i}^{2}r^{2}})^{n}
  dx \leq 
  C\lambda_{i}^{-(\frac{2n}{n+2})^{2}}
  \left( 1 + \int^{\infty}_{1}r^{-1+n+(\frac{2n}{n+2})^{2}-2n}dr \right) 
  =
  O\bigg(\bigg(\frac{1}{\lambda_{i}}\bigg)^{(\frac{2n}{n+2})^{2}}\bigg).
  \end{equation*}
  Thus Lemma \ref{lem_interactions}, where $b_1$ is defined,  yields
  \begin{equation}\label{linearized_bubble_interaction}
  \int K\varphi_{i}^{p}\varphi_{j}d\mu_{g_{0}}
  = 
  b_{1}\frac{K_{i}}{\lambda_{i}^{\theta}}\varepsilon_{i,j}
  +
  O
  \bigg(
  \tau^{2}
  +
  \sum_{r\neq s} \frac{\vert \nabla K_{r}\vert^{2}}{\lambda_{r}^{2}}
  +
  \frac{1}{\lambda_{r}^{4}}
  +
  \frac{1}{\lambda_{r}^{2(n-2)}}
  +
  \varepsilon_{r,s}^{\frac{n+2}{n}}
  \bigg), 
  \end{equation}
  and  we arrive at
  \begin{equation}\label{def_bar_b_1}
  \begin{split}
  \int K  (\alpha^{i}\varphi_{i})^{p+1}d\mu_{g_{0}} 
  = &
  \sum_{i}\alpha_{i}^{p+1}\int K\varphi_{i}^{p+1} d\mu_{g_{0}}
  +
  (p+1)\sum_{i\neq j}\alpha_{i}^{p}\alpha_{j}
  b_{1}\frac{K_{i}}{\lambda_{i}^{\theta}}\varepsilon_{i,j} \\
  = &
  \sum_{i}\alpha_{i}^{p+1}\int K\varphi_{i}^{p+1} d\mu_{g_{0}}
  +
  \bar b_{1}\sum_{i\neq j}\alpha_{i}^{\frac{n+2}{n-2}}\alpha_{j}
  \frac{K_{i}}{\lambda_{i}^{\theta}}\varepsilon_{i,j}
  ,\quad 
  \bar b_{1}=\frac{2n}{n-2}b_{1}
  \end{split}
  \end{equation}
  up to an error  
  $
  O
  (
  \tau^{2}
  +
  \sum_{r\neq s} \frac{\vert \nabla K_{r}\vert^{2}}{\lambda_{r}^{2}}
  +
  \frac{1}{\lambda_{r}^{4}}
  +
  \frac{1}{\lambda_{r}^{2(n-2)}}
  +
  \varepsilon_{r,s}^{\frac{n+2}{n}}
  ).
  $
  Finally, recalling our notation in Section \ref{s:set-up} and denoting by $x^i$ a generic 
  polynomial of degree $i$ in the $x$-variables, we expand 
  \begin{equation}\label{single_bubble_K_integral_expansion}
  \begin{split}
  \int K  \varphi_{i}^{p+1} d\mu_{g_{0}}
  = &
  \int_{B_{c}(a_{i})} K\varphi_{i}^{p+1} d\mu_{g_{0}}
  +
  O\big(\frac{1}{\lambda_{i}^{n-\theta}}\big) \\
  = &
  K_{i}\int_{B_{c}(a_{i})} \varphi_{i}^{p+1} d\mu_{g_{0}}
  +
  \nabla K_{i}\int_{B_{c}(a_{i})} x\varphi_{i}^{p+1} d\mu_{g_{0}} \\
  & +
  \frac{\nabla^{2}}{2}K_{i}\int_{B_{c}(a_{i})} x^{2}\varphi_{i}^{p+1} d\mu_{g_{0}} 
  +
  \frac{\nabla^{3}}{6}K_{i}\int_{B_{c}(a_{i})} x^{3}\varphi_{i}^{p+1} d\mu_{g_{0}}  +
  O\bigg(\frac{1}{\lambda_{i}^{4}}+\frac{1}{\lambda_{i}^{2(n-2)}}\bigg)
  \end{split}
  \end{equation}
with an extra error of order  $ O \left( \frac{\ln \lambda}{\lambda^4} \right)$ if $n = 4$. For the first term 
on the right-hand side up to some
  $ 
  O(\tau^{2}+\frac{1}{\lambda_{i}^{4}})
  $ 
  we may pass integrating with respect to conformal normal coordinates. Indeed
  \begin{equation*}
\underset{B_{c}(a_{i})}{ \int} \varphi_{i}^{p+1}d\mu_{g_{0}}
  = 
\underset{B_{c}(a_{i})}{ \int}u_{a_{i}}^{-\tau}(\frac{\varphi_{i}}{u_{a_{i}}})^{\frac{2n}{n-2}-\theta}d\mu_{g_{a_{i}}}
  = 
\underset{B_{c}(a_{i})}{ \int}(\frac{\varphi_{i}}{u_{a_{i}}})^{\frac{2n}{n-2}-\theta}d\mu_{g_{a_{i}}}
  +
  O(\tau \int_{B_{c}(a_{i})} r_{a_{i}}^{2}(\frac{\varphi_{i}}{u_{a_{i}}})^{\frac{2n}{n-2}-\theta}d\mu_{g_{a_{i}}})
  \end{equation*}
  and the latter term is of order $O(\frac{\tau}{\lambda_{i}^{2+\theta}})$.
  From  \eqref{eq:bubbles} we find 
  \begin{equation*}
  \begin{split}
  \int  & \varphi_{i}^{p+1}  d\mu_{g_{0}} 
  = 
  \int_{B_{c}(a_{i})} (\frac{\lambda_{i}}{1+\lambda_{i}^{2}r_{a_{i}}^{2}(1+r_{a_{i}}^{n-2}H_{a_{i}})^{\frac{2}{2-n}}})^{n-\theta} d\mu_{g_{a_{i}}} \\
  = &
  \int_{B_{c}(0)} (\frac{\lambda_{i}}{1+\lambda_{i}^{2}r^{2}})^{n-\theta}
  \big(
  1
  +
  \frac{2(n-\theta)}{n-2}
  \frac{\lambda_{i}^{2}r^{n}H_{a_{i}}}{1+\lambda_{i}^{2}r^{2}}\big) dx, 
  \end{split}
  \end{equation*}
  up to some $O(\tau^{2}+\frac{1}{\lambda_{i}^{4}}+\frac{1}{\lambda_{i}^{n}})$. Clearly
  \begin{equation*}
  \begin{split}
  \int_{B_{c}(0)}  (\frac{\lambda_{i}}{1+\lambda_{i}^{2}r^{2}})^{n-\theta} dx
  = &
  \lambda_{i}^{-\theta}
  \int_{B_{c\lambda_{i}}}\frac{dx}{(1+r^{2})^{n-\theta}}
  =
  \lambda_{i}^{-\theta}
  \underset{\R^{n}}{\int}\frac{dx}{(1+r^{2})^{n-\theta}}
  +O(\lambda_{i}^{-n}) \\
  = &
  \frac{1}{\lambda^{\theta}}\underset{\R^{n}}{\int}\frac{dx}{(1+r^{2})^{n}}
  +
  \frac{\theta}{\lambda^{\theta}}\underset{\R^{n}}{\int}\frac{\ln (1+r^{2})dx}{(1+r^{2})^{n}}
  +
  O\bigg(\tau^{2}+\frac{1}{\lambda_{i}^{4}}+O\bigg(\frac{1}{\lambda_{i}^{2(n-2)}}\bigg)\bigg) \\
  = &
  \frac{\bar c_{0}}{\lambda_{i}^{\theta}}+\frac{\bar c_{1} \tau}{\lambda_{i}^{\theta}}
  +
  O\bigg(\tau^{2}+\frac{1}{\lambda_{i}^{4}}+O\bigg(\frac{1}{\lambda_{i}^{2(n-2)}}\bigg)\bigg)
  \end{split}
  \end{equation*}
letting 
\begin{equation}\label{barc0_barc1}
\bar c_{0}=\int_{\R^{n}}\frac{dx}{(1+r^{2})^{n}}
\;\text{ and } \;
\bar c_{1}=\frac{n-2}{2}\int_{\R^{n}}\frac{\ln (1+r^{2})}{(1+r^{2})^{n}}dx.
\end{equation}
  Moreover
  \begin{equation*}
  \begin{split}
  \underset{B_{c}(0)}{\int} \frac{\lambda_{i}^{n+4-\theta }r^{2n} H^{2}_{a_{i}}}{(1+\lambda_{i}^{2}r^{2})^{n+2-\theta}}dx
  \leq 
\underset{B_{c}(0)}{\int} \frac{\lambda_{i}^{n-\theta}r^{2(n-2)}H^{2}_{a_{i}}}{(1+\lambda_{i}^{2}r^{2})^{n-\theta}} dx 
  \leq 
  C\underset{B_{c}(0)}{\int}\frac{\lambda_{i}^{n-\theta} r^{2(n-2)}}{(1+\lambda_{i}^{2}r^{2})^{n-\theta}}
  \begin{pmatrix}
  1 &\text{for }\; n= 3\\
  1 &\text{for }\; n=4\\
  1 &\text{for }\; n=5\\
  \ln^{2}r  &\text{for }\; n=6 \\
  r^{2(6-n)}&\text{for }\; n\geq 7
  \end{pmatrix} dx
  \end{split}
  \end{equation*}
  up to some 
  $
O
(
  \frac{1}{\lambda_{i}^{2(n-2)}}+\frac{1}{\lambda_{i}^{4}}
)
  $
 and with an extra error of order  $ O \left( \frac{\ln \l}{\l^4} \right)$ if $n = 4$, and 
  \begin{equation*}
  \underset{B_{c}(0)}{\int}  \bigg(\frac{\lambda_{i}}{1+\lambda_{i}^{2}r^{2}}\bigg)^{n-\theta}\frac{\lambda_{i}^{2}r^{n}H_{a_{i}}}{1+\lambda_{i}^{2}r^{2}} dx
 = 
  \underset{B_{c}(0)}{\int} \bigg(\frac{\lambda_{i}}{1+\lambda_{i}^{2}r^{2}}\bigg)^{n-\theta}\frac{\lambda_{i}^{2}r^{2}}{1+\lambda_{i}^{2}r^{2}} 
  r^{n-2}
  \begin{pmatrix}
  H_{i}+\nabla H_{i}x+O(r^{2}) \\
  H_{i}+\nabla H_{i}x+O(r^{2}\ln r) \\
  H_{i}+O(r) \\
  -W_{i}\ln r +O(r\ln r) \\
  O(r^{6-n}) 
  \end{pmatrix} dx,
  \end{equation*}
  whence up to some 
  $O(\tau^{2}+\frac{1}{\lambda_{i}^{4}}+\frac{1}{\lambda_{i}^{2(n-2)}})$
  \begin{equation}\label{mass_integral_expansion}
  \begin{split}
  \int_{B_{c}(0)}  \bigg(\frac{\lambda_{i}}{1+\lambda_{i}^{2}r^{2}}\bigg)^{n-\theta}\frac{\lambda_{i}^{2}r^{n}H_{a_{i}}}{1+\lambda_{i}^{2}r^{2}} 
  dx = &
  \bar d_{1}
  \begin{pmatrix}
  \frac{H_{i}}{\lambda_{i}^{1+\theta }} \\
  \frac{H_{i}}{\lambda_{i}^{2+\theta }}+O(\frac{\ln \lambda_{i}}{\lambda_{i}^{4+\theta}})\\
  \frac{H_{i}}{\lambda_{i}^{3+\theta }} \\
  \frac{W_{i}\ln \lambda_{i}}{\lambda_{i}^{4+\theta}} \\
  0 
  \end{pmatrix}
  ,\; \bar d_{1}
  =
  \underset{\R^{n}}{\int} \frac{r^{n}dx}{(1+r^{2})^{n+1}}.
  \end{split}
  \end{equation}
  Likewise by radial symmetry and, since we may assume
  $d\mu_{g_{a_{i}}}  \equiv 1$, cf. \cite{gunther},
we find
\begin{enumerate}
 \item[(1)] \quad
 $
 \int_{B_{c}(a_{i})} x^{3}\varphi_{i}^{p+1} d\mu_{g_{a_{i}}}
  =
  O\bigg(\frac{1}{\lambda_{i}^{4}} + \frac{1}{\lambda_{i}^{2(n-2)}}\bigg);
 $
 \item[(2)] \quad
 $
 \frac{\nabla^{2}}{2}K_{i}\int_{B_{c}(a_{i})} x^{2}\varphi_{i}^{p+1} d\mu_{g_{0}}
  =
  \frac{\Delta K_{i}}{2n\lambda_{i}^{2+\theta}}\underset{\R^{n}}{\int}\frac{r^{2}dx}{(1+r^{2})^{n}}
  +
  O\bigg(\tau^{2}+\frac{1}{\lambda_{i}^{4}}+\frac{1}{\lambda_{i}^{2(n-2)}}\bigg);
 $
 \item[(3)] \quad
$
  \int_{B_{c}(a_{i})} x\varphi_{i}^{p+1} d\mu_{g_{0}}
  =
  O
  \bigg(\frac{1}{\lambda_{i}^{4}}+\frac{1}{\lambda_{i}^{2(n-2)}}\bigg)
$
\end{enumerate}
with an extra error of order  $ O \left( \frac{\ln \l}{\l^4} \right)$ if $n = 4$. Collecting all terms we arrive at
  \begin{equation}\label{single_bubble_K_integral_expansion_exact}
  \begin{split}
  \int K \varphi_{i}^{p+1}d\mu_{g_{0}}
  = &
  \frac{\bar c_{0}K_{i}}{\lambda_{i}^{\theta}}
  +
  \bar c_{1}\frac{K_{i}\tau}{\lambda_{i}^{\theta}}
  +
  \bar c_{2}\frac{\Delta K_{i}}{\lambda_{i}^{2+\theta}}  
  +
  \bar d_{1}K_{i}
  \begin{pmatrix}
  \frac{H_{i}}{\lambda_{i}^{1+\theta }} \\
  \frac{H_{i}}{\lambda_{i}^{2+\theta }}+O(\frac{\ln \lambda_{i}}{\lambda_{i}^{4+\theta}}) \\
  \frac{H_{i}}{\lambda_{i}^{3+\theta }} \\
  \frac{W_{i}\ln \lambda_{i}}{\lambda_{i}^{4+\theta}} \\
  0 
  \end{pmatrix},
  \; 
  \bar c_{2}=\frac{1}{2n}\underset{\R^{n}}{\int}\frac{r^{2}dx}{(1+r^{2})^{n}}
  \end{split}
  \end{equation} 
  up to an error  
  $O(\tau^{2}+\frac{1}{\lambda_{i}^{4}}+\frac{1}{\lambda_{i}^{2(n-2)}})$, 
  and thus obtain
  \begin{equation}\label{intergral_sum_of_bubble_nonlinear_evaluated}
  \begin{split}
  \int K  (\alpha^{i}  \varphi_{i})^{p+1}d\mu_{g_{0}} 
  = &
  \sum_{i}
  \left(
  \bar c_{0}\frac{K_{i}}{\lambda_{i}^{\theta}}\alpha_{i}^{p+1}
  +
  \bar c_{1}\frac{K_{i}}{\lambda_{i}^{\theta}}\alpha_{i}^{\frac{2n}{n-2}}\tau  
  +
  \bar c_{2}\frac{\Delta K_{i}}{\lambda_{i}^{2+\theta}} \alpha_{i}^{\frac{2n}{n-2}} 
  \right) \\
  & \quad\quad\quad +
  \bar d_{1}\sum_{i}\frac{K_{i}}{\lambda_{i}^{\theta}}\alpha_{i}^{\frac{2n}{n-2}}
  \begin{pmatrix}
  \frac{H_{i}}{\lambda_{i}} \\
  \frac{H_{i}+O(\frac{\ln \lambda_{i}}{\lambda_{i}^{2}})}{\lambda_{i}^{2 }}\\
  \frac{H_{i}}{\lambda_{i}^{3 }} \\
  \frac{W_{i}\ln \lambda_{i}}{\lambda_{i}^{4}} \\
  0 
  \end{pmatrix} 
  +
  \bar b_{1}\sum_{i\neq j}\alpha_{i}^{\frac{n+2}{n-2}}\alpha_{j}
  \frac{K_{i}}{\lambda_{i}^{\theta}}\varepsilon_{i,j}
  \end{split}
  \end{equation}
  up to some 
  $
  O
  (
  \tau^{2}
  +
  \sum_{r\neq s} \frac{\vert \nabla K_{r}\vert^{2}}{\lambda_{r}^{2}}
  +
  \frac{1}{\lambda_{r}^{4}}
  +
  \frac{1}{\lambda_{r}^{2(n-2)}}
  +
  \varepsilon_{r,s}^{\frac{n+2}{n}}
  ).
  $
  Consequently up to the same error 
  \begin{equation}\label{pure_functional_denominator_expanded}
  \begin{split}
  J_{\tau}(  \alpha^{i} & \varphi_{i})
  = 
  \frac
  {\alpha^{i}\alpha^{j}\int \varphi_{i} L_{g_{0}}\varphi_{j}d\mu_{g_{0}}}
  {(\int K(\sum_{i}\alpha_{i}\varphi_{i})^{p+1})^{\frac{2}{p+1}}} 
  = 
  \frac
  {\alpha^{i}\alpha^{j}\int \varphi_{i} L_{g_{0}} \varphi_{j}d\mu_{g_{0}}}
  {(\bar c_{0}\sum_{i}\frac{K_{i}}{\lambda_{i}^{\theta}}\alpha_{i}^{p+1})^{\frac{2}{p+1}}}
  \Bigg(
  1
  -
  \bar c_{1}\sum_{i}\frac{K_{i}}{\lambda_{i}^{\theta}}\frac{\alpha_{i}^{\frac{2n}{n-2}}}{\alpha_{K,\tau}^{\frac{2n}{n-2}}}\tau
  \\
  & -
  \bar c_{2}\sum_{i}\frac{\Delta K_{i}}{\lambda_{i}^{2+\theta}} \frac{\alpha_{i}^{\frac{2n}{n-2}} }{\alpha_{K,\tau}^{\frac{2n}{n-2}}}
  -
  \bar d_{1}\sum_{i}\frac{K_{i}}{\lambda_{i}^{\theta}}
  \begin{pmatrix}
  \frac{H_{i}}{\lambda_{i}} \\
  \frac{H_{i}+O(\frac{\ln \lambda_{i}}{\lambda_{i}^{2}})}{\lambda_{i}^{2 }} \\
  \frac{H_{i}}{\lambda_{i}^{3 }} \\
  \frac{W_{i}\ln \lambda_{i}}{\lambda_{i}^{4}} \\
  0 
  \end{pmatrix} 
  \frac{\alpha_{i}^{\frac{2n}{n-2}}}{\alpha_{K,\tau}^{\frac{2n}{n-2}}}  -
  \bar b_{1}\sum_{i\neq j}\frac{\alpha_{i}^{\frac{n+2}{n-2}}\alpha_{j}}{\alpha_{K,\tau}^{\frac{2n}{n-2}}}
  \frac{K_{i}}{\lambda_{i}^{\theta}}\varepsilon_{i,j}
  \Bigg).
  \end{split}
  \end{equation} 
  Next for $i\neq j$ using Lemma \ref{lem_emergence_of_the_regular_part} we get
  \begin{equation*}
  \int \frac{\varphi_{i} L_{g_{0}} \varphi_{j}}{4n(n-1)} d \mu_{g_{0}}
  =
  \int \varphi_{i}^{\frac{n+2}{n-2}}\varphi_{j} d \mu_{g_{0}}
  +
  O\bigg(\frac{1}{\lambda_{i}^{4}}+\frac{1}{\lambda_{i}^{2(n-2)}}+\varepsilon_{i,j}^{\frac{n+2}{n}}\bigg).
  \end{equation*}
  For example to check the error term, we may estimate 
  \begin{equation*}
  \begin{split}
  \underset{B_{c}(a_{i})}{\int} r_{a_{i}}^{n-2}\varphi_{i}^{\frac{n+2}{n-2}}\varphi_{j}d\mu_{g_{0}}
  \leq &
  \Vert r_{a_{i}}^{n-2}\varphi_{i}^{\frac{n+2}{n-2}-\frac{n+2}{2n}}\Vert_{L^{(\frac{2n}{n+2})^{2}}_{B_{c}(a_{i})}} 
  \Vert \varphi_{i}^{\frac{n+2}{2n}}\varphi_{j}\Vert_{L^{\frac{4n^{2}}{(3n+2)(n-2)}}},
  \end{split}
  \end{equation*}
  which is of order 
  $O(\frac{\varepsilon_{i,j}^{\frac{n+2}{2n}}}{\lambda_{i}^{n-2}})$
  thanks to Lemma \ref{lem_interactions}, and likewise for e.g. $n\geq 7$  
  \begin{equation*}
  \begin{split}
  \int \varphi_{i}\varphi_{j} d\mu_{g_{0}} 
  \lesssim  &
  \Vert \varphi_{i}\varphi_{j}\Vert_{L^{\frac{n}{n-2}}_{g_{0}}}
  = 
  O(\varepsilon_{i,j}\ln^{\frac{n-2}{n}}\varepsilon_{i,j}),
  \end{split}
  \end{equation*}
  whence 
  $
  \lambda_{i}^{-2}\int \varphi_{i}\varphi_{j} d\mu_{g_{0}} 
  =
  o(\frac{\varepsilon_{i,j}^{\frac{n+2}{2n}}}{\lambda_{i}^{2}})
  $.
  Thus Lemma \ref{lem_interactions} shows that 
  \begin{equation}\label{L_g_0_bubble_interaction}
  \int \varphi_{i} L_{g_{0}} \varphi_{j}d\mu_{g_{0}}
  =
  \tilde b_{1}\varepsilon_{i,j}
  +
  O(\sum_{r\neq s}\frac{1}{\lambda_{r}^{4}}+\frac{1}{\lambda_{r}^{2(n-2)}}+\varepsilon_{r,s}^{\frac{n+2}{n}}),
  \quad 
  \tilde b_{1}=4n(n-1)b_{1}. 
  \end{equation}
  Finally from  \eqref{eq:bubbles} and Lemma \ref{lem_emergence_of_the_regular_part} we find
  \begin{equation*}
  \begin{split}
  \int \frac{\varphi_{i} L_{g_{0}} \varphi_{i}}{4n(n-1)} d\mu_{g_{0}}
  = &
  \int \varphi_{i}^{\frac{2n}{n-2}}d\mu_{g_{0}} 
  -
  \frac{c_{n}}{2}
  \int_{B_{c}(0)}\frac{\lambda_{i}^{2}r^{n-2}}{(1+\lambda_{i}^{2}r^{2})^{n}}
  \begin{pmatrix}
  H_{i}+n\nabla H_{i}x  \\
  H_{i}+n\nabla H_{i}x \\
  H_{i}+n\nabla H_{i}x  \\
  -W_{i}\ln r  \\
  0 
  \end{pmatrix}
  \end{split}
  \end{equation*}
  up to some error terms of order
  $
  O
  (
  \lambda_{i}^{-3}  ,
  \lambda_{i}^{-4}\ln \lambda_{i}  ,
  \lambda_{i}^{-4}  ,
  \lambda_{i}^{-4}  ,
  \lambda_{i}^{-4} 
  ),
  $
  whence
  \begin{equation*}
  \begin{split}
  \int \frac{\varphi_{i} L_{g_{0}} \varphi_{i}}{4n(n-1)} d \mu_{g_{0}}
  = &
  \int \varphi_{i}^{\frac{2n}{n-2}}d\mu_{g_{0}}
  -
  \tilde d_{1}
  \begin{pmatrix}
  \frac{H_{i}}{\lambda_{i}} \\
  \frac{H_{i}}{\lambda_{i}^{2}} +O(\frac{\ln \lambda_{i}}{\lambda_{i}^{4}}) \\
  \frac{H_{i}}{\lambda_{i}^{3}} \\
  \frac{W_{i}\ln \lambda_{i}}{\lambda_{i}^{4}}  \\
  0 
  \end{pmatrix},\; \tilde d_{1}
  =
  \frac{c_{n}}{2}\underset{\R^{n}}{\int}\frac{r^{n-2}dx}{(1+r^{2})^{n}}, 
  \end{split}
  \end{equation*}
  up to  
  $
  O(\frac{1}{\lambda_{i}^{4}}+\frac{1}{\lambda_{i}^{2(n-2)}}).
  $
  Recalling \eqref{single_bubble_K_integral_expansion_exact}, we obtain 
  \begin{equation}\label{single_bubble_L_g_0_integral_expansion_exact}
  \begin{split}
  \int \frac{\varphi_{i} L_{g_{0}}\varphi_{i}}{4n(n-1)} d\mu_{g_{0}}
  = &
  \bar c_{0}
  +
  4n(n-1)
  (\bar d_{1}
  -
  \tilde d_{1}
  )
  \begin{pmatrix}
  \frac{H_{i}}{\lambda_{i}} \\
  \frac{H_{i}}{\lambda_{i}^{2}} +O(\frac{\ln \lambda_{i}}{\lambda_{i}^{4}}) \\
  \frac{H_{i}}{\lambda_{i}^{3}} \\
  \frac{W_{i}\ln \lambda_{i}}{\lambda_{i}^{4}} \\
  0 
  \end{pmatrix} 
  \end{split}
  \end{equation}
  up to some 
  $O(\tau^{2}+\frac{1}{\lambda_{i}^{4}}+\frac{1}{\lambda_{i}^{2(n-2)}})$.
As $\bar d_{1}=\tilde d_{1}$, cf. \eqref{mass_integral_expansion},  we simply get
  \begin{equation}\label{r_alpha_delta_expansion} 
  \begin{split}
  \alpha^{i}\alpha^{j}\int \varphi_{i} L_{g_{0}}\varphi_{j}d\mu_{g_{0}}
  = &
  4n(n-1)\bar c_{0}
  \sum_{i}
  \alpha_{i}^{2}
 +
  \tilde b_{1}\sum_{i\neq j}\alpha_{i}\alpha_{j}\varepsilon_{i,j}
  \end{split}
  \end{equation}
  up to an error of order 
$
  O\bigg(\tau^{2}+\sum_{r}\frac{1}{\lambda_{r}^{4}}+\frac{1}{\lambda_{r}^{2(n-2)}}+\sum_{r\neq s}\varepsilon_{r,s}^{\frac{n+2}{n}}\bigg).
$
  Plugging this into \eqref{pure_functional_denominator_expanded}, we obtain
  \begin{equation*}
  \begin{split}
  J_{\tau}  ( \alpha^{i}  \varphi_{i}) 
  = &
  \frac
  {4n(n-1)\bar c_{0}^{\frac{p-1}{p+1}}\sum_{i}\alpha_{i}^{2}}
  {(\sum_{i}\frac{K_{i}}{\lambda_{i}^{\theta}}\alpha_{i}^{p+1})^{\frac{2}{p+1}}}
  \Bigg(
  1
  -
  \bar c_{1}\sum_{i}\frac{K_{i}}{\lambda_{i}^{\theta}}\frac{\alpha_{i}^{\frac{2n}{n-2}}}{\alpha_{K,\tau}^{\frac{2n}{n-2}}}\tau
  -
  \bar c_{2}\sum_{i}\frac{\Delta K_{i}}{\lambda_{i}^{2+\theta}} \frac{\alpha_{i}^{\frac{2n}{n-2}} }{\alpha_{K,\tau}^{\frac{2n}{n-2}}}\\
  & \quad\quad\quad\quad\quad
  -
  \bar d_{1} \sum_{i}
  \frac{K_{i}}{\lambda_{i}^{\theta}}
  \frac{\alpha_{i}^{\frac{2n}{n-2}}}{\alpha_{K,\tau}^{\frac{2n}{n-2}}}
  \begin{pmatrix}
  \frac{H_{i}}{\lambda_{i}} \\
  \frac{H_{i}+O(\frac{\ln \lambda_{i}}{\lambda_{i}^{2}})}{\lambda_{i}^{2 }} \\
  \frac{H_{i}}{\lambda_{i}^{3}} \\
  \frac{W_{i}\ln \lambda_{i}}{\lambda_{i}^{4}} \\
  0 
  \end{pmatrix}
  -
  \sum_{i\neq j}(\bar b_{1}\frac{K_{i}}{\lambda_{i}^{\theta}}\frac{\alpha_{i}^{\frac{n+2}{n-2}}\alpha_{j}}{\alpha_{K,\tau}^{\frac{2n}{n-2}}}
  -
  \frac{\tilde b_{1}}{\bar c_{0}}\frac{\alpha_{i}\alpha_{j}}{\alpha^{2}}
  )
  \varepsilon_{i,j}
  \Bigg)
  \end{split}
  \end{equation*} 
  up to some 
  $
  O(\tau^{2}+\sum_{r}\frac{\vert \nabla K_{r}\vert^{2}}{\lambda_{r}^{2}}+\frac{1}{\lambda_{r}^{4}}+\frac{1}{\lambda_{r}^{2(n-2)}}+\sum_{r\neq s}\varepsilon_{r,s}^{\frac{n+2}{n}})
  $.
  Recalling 
  \begin{equation*}
  \bar b_{1}=\frac{2n}{n-2}b_{1}, \quad \tilde b_{1}=4n(n-1)b_{1}, \quad 
  \alpha^{2}=\sum_{i}\alpha_{i}^{2}, \quad 
  \alpha_{K,\tau}^{\frac{2n}{n-2}}
  =\frac{n}{n-2}\bar c_{0}\sum_{i}\frac{K_{i}}{\lambda_{i}^{\theta}}\alpha_{i}^{\frac{2n}{n-2}},
  \end{equation*}
  and setting 
  \begin{equation}\label{hat_constants}
  \hat c_{0}=4n(n-1)\bar c_{0}^{\frac{p-1}{p+1}},\; 
  \hat c_{1}=\frac{\bar c_{1}}{\bar c_{0}},\; \hat c_{2}=\frac{\bar c_{2}}{\bar c_{0}},\; 
  \hat d_{1}=\frac{\bar d_{1}}{\bar c_{0}},\; 
  \hat b_{1}=2\frac{ b_{1}}{\bar c_{0}}
 \end{equation}
  we may rewrite this as  
  \begin{equation*}
  \begin{split}
  J_{\tau}(u)
  =
  J_{\tau}  ( \alpha^{i}  \varphi_{i}) 
  = 
  \frac
  {\hat c_{0}\alpha^{2}}
  {(\alpha_{K,\tau}^{p+1})^{\frac{2}{p+1}}}
  \Bigg(
  1
  & -
  \hat c_{1}\sum_{i}\frac{K_{i}}{\lambda_{i}^{\theta}}\frac{\alpha_{i}^{\frac{2n}{n-2}}}{\alpha_{K,\tau}^{\frac{2n}{n-2}}}\tau
  -
  \hat c_{2}\sum_{i}\frac{\Delta K_{i}}{\lambda_{i}^{2+\theta}} \frac{\alpha_{i}^{\frac{2n}{n-2}} }{\alpha_{K,\tau}^{\frac{2n}{n-2}}}\\
  & 
  -
  \hat d_{1} \sum_{i}
  \frac{K_{i}}{\lambda_{i}^{\theta}}
  \frac{\alpha_{i}^{\frac{2n}{n-2}}}{\alpha_{K,\tau}^{\frac{2n}{n-2}}}
  \begin{pmatrix}
  \frac{H_{i}}{\lambda_{i}} \\
  \frac{H_{i}+O(\frac{\ln \lambda_{i}}{\lambda_{i}^{2}})}{\lambda_{i}^{2 }} \\
  \frac{H_{i}}{\lambda_{i}^{3}} \\
  \frac{W_{i}\ln \lambda_{i}}{\lambda_{i}^{4}} \\
  0 
  \end{pmatrix}
  -
  \hat b_{1}\sum_{i\neq j}\big(\frac{K_{i}}{\lambda_{i}^{\theta}}\frac{\alpha_{i}^{\frac{n+2}{n-2}}\alpha_{j}}{\alpha_{K,\tau}^{\frac{2n}{n-2}}}
  -
  \frac{\alpha_{i}\alpha_{j}}{2\alpha^{2}}
  \big)
  \varepsilon_{i,j}
  \Bigg). 
  \end{split}
  \end{equation*}
  Then the claim follows from Lemma \ref{lem_alpha_derivatives_at_infinity}.
  \end{proof}

  We next state three lemmas with some expansions for the derivatives of the functionals 
  with respect to the parameters involved (recall our notation from Section \ref{s:set-up}). The proofs are given in appendix  B.

  \begin{lemma} 
  \label{lem_alpha_derivatives_at_infinity} 
  For $u\in V(q,\varepsilon)$ and $\varepsilon>0$ sufficiently small the three quantities 
  $\partial J_{\tau} (u)  \phi_{1,j}$, 
  $\partial J_{\tau}(\alpha^{i}\varphi_{i})\phi_{1,j}$, 
  $\partial_{\alpha_{j}}J_{\tau}(\alpha^{i}\varphi_{i})$ can be written as 
  \begin{equation*}
  \begin{split}  
  \frac{\alpha_{j}}
  {(\alpha_{K,\tau}^{\frac{2n}{n-2}})^{\frac{n-2}{n}}}
  \bigg( 
  &
  \grave c_{0}\big(
  1
  -
  \frac{\alpha^{2}}{\alpha_{K,\tau}^{p+1}}
  \frac{K_{j}}{\lambda_{j}^{\theta}}\alpha_{j}^{p-1}
 \big)
  -
  \grave c_{2}
  \big(
  \frac{\Delta K_{j}}{K_{j}\lambda_{j}^{2}}
  -
  \sum_{k}\frac{\Delta K_{k}}{K_{k}\lambda_{k}^{2}}
  \frac{\alpha_{k}^{2}}{\alpha^{2}}
  \big)
  \\ & 
  +
  \grave b_{1} \bigg(
  \sum_{k\neq l}
  \frac{\alpha_{k}\alpha_{l}}{\alpha^{2}}
  \varepsilon_{k,l}
  -
  \sum_{j\neq i}\frac{\alpha_{i}}{\alpha_{j}}\varepsilon_{i,j}
  \bigg)
  -
  \grave d_{1}
  \begin{pmatrix}
  \frac{H_{j}}{\lambda_{j}}-\sum_{k}\frac{\alpha_{k}^{2}}{\alpha^{2}}\frac{H_{k}}{\lambda_{k}} &\text{for } n=3\\
  \frac{H_{j}}{\lambda_{j}^{2 }}-\sum_{k}\frac{\alpha_{k}^{2}}{\alpha^{2}}\frac{H_{k}}{\lambda_{k}^{2 }} 
  +
  O(\sum_{r}\frac{\ln \lambda_{r}}{\lambda_{r}^{4}})
  &\text{for } n=4\\
  \frac{H_{j}}{\lambda_{j}^{3}} -\sum_{k}\frac{\alpha_{k}^{2}}{\alpha^{2}}\frac{H_{k}}{\lambda_{k}^{3 }} &\text{for } n=5\\
  \frac{W_{j}\ln \lambda_{j}}{\lambda_{i}^{4}}-\sum_{k}\frac{\alpha_{k}^{2}}{\alpha^{2}}\frac{W_{k}\ln \lambda_{k}}{\lambda_{k}^{4}} &\text{for } n=6\\
  0 &\text{for }n\geq 7
  \end{pmatrix} 
  \bigg)
  \end{split}
  \end{equation*}
  with positive constants $\grave c_{0},\grave c_{2},\grave b_{1},\grave d_{1}$ up to an error of order 
  \begin{equation}\label{eq:error-19-8}
  O
  \big(
  \tau^{2}
  +
  \sum_{r\neq s} 
  \frac{\vert \nabla K_{r}\vert^{2}}{\lambda_{r}^{2}}
  +
  \frac{1}{\lambda_{r}^{4}}
  +
  \frac{1}{\lambda_{r}^{2(n-2)}}
  +
  \varepsilon_{r,s}^{\frac{n+2}{n}}
  +
  \vert \partial J_{\tau}(u)\vert^{2}
  \big).
  \end{equation}
  In particular for all $j$ 
  \begin{equation*}
  \frac{\alpha^{2}}{\alpha_{K,\tau}^{p+1}}\frac{K_{j}}{\lambda_{j}^{\theta}}\alpha_{j}^{p-1}
  =
  1
  +
  O 
  \big(
  \tau
  +
  \sum_{r\neq s} 
  \frac{1}{\lambda_{r}^{2}}
  +
  \frac{1}{\lambda_{r}^{n-2}}
  +
  \varepsilon_{r,s}
  +
  \vert \partial J_{\tau}(u)\vert
  \big).
  \end{equation*}  
  \end{lemma}
  
  \begin{lemma} 
  \label{lem_lambda_derivatives_at_infinity} 
  For $u\in V(q,\varepsilon)$ and $\varepsilon>0$ sufficiently small the three quantities 
  $\partial  J_{\tau}  (u) \phi_{2,j}$,  
  $\partial J_{\tau}(\alpha^{i}\varphi_{i})\phi_{2,j}$ and 
  $ \frac{\lambda_{j}}{\alpha_{j}}\partial_{\lambda_{j}}J_{\tau}(\alpha^{i}\varphi_{i})$
  can be written as 
  \begin{equation*}
  \begin{split}
  \frac{\alpha_{j}}{(\alpha_{K,\tau}^{\frac{2n}{n-2}})^{\frac{n-2}{n}}}
  \bigg(
  \tilde c_{1}\tau
  +
  \tilde c_{2}\frac{\Delta K_{j}}{K_{j}\lambda_{j}^{2}} 
  -
  \tilde b_{2}\sum_{j\neq i }\frac{\alpha_{i}}{\alpha_{j}}\lambda_{j}\partial_{\lambda_{j}}\varepsilon_{i,j} 
 +
  \tilde  d_{1}
  \begin{pmatrix}
  \frac{H_{j}}{\lambda_{j}} &\text{for }\; n=3\\
  \frac{H_{j}}{\lambda_{j}^{2}}+O(\frac{\ln \lambda_{j}}{\lambda_{j}^{4}})&\text{for }\; n=4\\
  \frac{H_{j}}{\lambda_{j}^{3 }} &\text{for }\; n=5\\
  \frac{W_{j}\ln \lambda_{j}}{\lambda_{j}^{4}} &\text{for }\; n=6\\
  0 &\text{for }\; n\geq 7
  \end{pmatrix} 
  \bigg),
  \end{split}
  \end{equation*}
  with positive constants $\tilde c_{1},\tilde c_{2},\tilde d_{1},\tilde b_{2}$ up to some error of the form 
  \begin{equation}\label{eq:error-19-8-2}
  O
  \big(
  \tau^{2}
  +
  \sum_{r\neq s}
  \frac{\vert \nabla K_{r}\vert^{2}}{\lambda_{r}^{2}}
  +
  \frac{1}{\lambda_{r}^{4}}
  +
  \frac{1}{\lambda_{r}^{2(n-2)}}
  +
  \varepsilon_{r,s}^{\frac{n+2}{n}}
  +
  \vert \partial J_{\tau}(u)\vert^{2}
  \big).
  \end{equation}
  \end{lemma}
  
  \begin{lemma} 
  \label{lem_a_derivatives_at_infinity} 
  For $u\in V(q,\varepsilon)$ and $\varepsilon>0$ sufficiently small the three quantities 
  $\partial J_{\tau}(u)\phi_{3,j} $, 
  $\partial J_{\tau}(\alpha^{i}\varphi_{i})\phi_{3,j}$ and $\frac{\nabla_{a_{j}}}{\alpha_{j}\lambda_{j}}J_{\tau}(\alpha^{i}\varphi_{i})$ can be written as 
  \begin{equation*}
  \begin{split}
  -\frac{\alpha_{j}}{(\alpha_{K,\tau}^{\frac{2n}{n-2}})^{\frac{n-2}{n}}}
  \left(
  \check{c}_{3}\frac{\nabla K_{j}}{K_{j}\lambda_{j}}
  +
  \check{c}_{4} \frac{\nabla \Delta K_j}{K_j \lambda_{j}^3}
  +
  \check{b}_{3}
  \sum_{j\neq i}
  \frac{\alpha_{i}}{\alpha_{j}}
  \frac{\nabla_{a_{j}}}{\lambda_{j}}\varepsilon_{i,j}
  \right), 
  \end{split}
  \end{equation*}
  with positive constants $\check{c}_{3}, \check{c}_{4} , \check{b}_{3}$  up to some error of the form 
  \begin{equation}\label{eq:error-19-8-3}
  O
  \big(
  \tau^{2}
  +
  \sum_{r\neq s}
  \frac{\vert \nabla K_{r}\vert^{2}}{\lambda_{r}^{2}}
  +
  \frac{1}{\lambda_{r}^{4}}
  +
  \frac{1}{\lambda_{r}^{2(n-2)}}
  +
  \varepsilon_{r,s}^{\frac{n+2}{n}}
  +
  \vert \partial J_{\tau}(u)\vert^{2}
  \big).
  \end{equation}
  \end{lemma}

 \section{Gradient bounds}\label{s:lower-bds}
 
 Theorem \ref{lem_top_down_cascade} will give suitable lower norm-bounds on the gradient of $J_{\tau}$, 
 yielding  Theorem \ref{t:main} as a corollary. We recall that on 
 $S^{3}$ and $S^{4}$
 the result  was proved in \cite{[ACGPY]}, \cite{yy}, \cite{yy2}, \cite{[SZ]} in more generality.

 \begin{definition} 
 \label{def_none-degeneracy} 
 Let $H$ be as in \eqref{eq:exp-G}. We call a positive Morse function $K$ on $M$ {\em non-degenerate} 
 \begin{enumerate}[label=(\roman*)]
  \item  {\em  of degree} $q\in \N$ in case $n=4$, if 
  $ 
  \{\nabla K=0\}\cap \{\tilde c_{2}\frac{\Delta K}{K}  
 +
 \tilde c_{3}H=0\}=\emptyset 
  $ 
 and if for every $1\leq k\leq q$ and every subset 
 $ 
 \{x_{1},\ldots,x_{k}\}\subseteq \{\nabla K=0\}\cap \{\tilde c_{2}\frac{\Delta K}{K}  
 +
 \tilde c_{3}H<0\}
  $ 
 the matrices
 \begin{equation*}
 \resizebox{.9\hsize}{!}
 {$
 \mathcal{M}_{x_{1},\ldots,x_{k}}=
 -
 \begin{pmatrix}
 \tilde c_{2}\frac{\Delta K(x_1)}{K(x_1)^{2}}  
 +
 \tilde c_{3}\frac{H(x_1)}{K(x_1)}
 &
 \tilde c_{4}\frac{G_{0}(x_{1},x_{2})}{\gamma_{n}(K(x_1)K(x_2))^{\frac{1}{2}}}
 &
 \ldots
 & 
 \tilde c_{4}\frac{G_{0}(x_{1},x_{k})}{\gamma_{n}(K(x_1)K(x_k))^{\frac{1}{2}}}
 \\
 \tilde c_{4}\frac{G_{0}(x_{2},x_{1})}{\gamma_{n}(K(x_2)K(x_1))^{\frac{1}{2}}}  
 & \ddots 
 & 
 & \vdots 
 \\
  \vdots 
 &
 &
 &
 \vdots 
 \\
 \vdots 
 &
 & \ddots
 & 
 \tilde c_{4}\frac{G_{0}(x_{k-1},x_{k})}{\gamma_{n}(K(x_{k-1})K(x_k))^{\frac{1}{2}}} 
 \\
  \tilde c_{4}\frac{G_{0}(x_{k},x_{1})}{\gamma_{n}(K(x_k)K(x_1))^{\frac{1}{2}}}
 & \ldots 
 &
 &
 \tilde c_{2}\frac{\Delta K(x_k)}{K(x_k)^{2}} 
 +
 \tilde c_{3}\frac{H(x_k)}{K(x_k)}  
 \end{pmatrix}
 $}
 \end{equation*}
 have non-vanishing least eigenvalues, where 
$
\tilde c_{2}=\sqrt{3\omega_{4}},\;  \tilde c_{3}=24\sqrt{3\omega_{4}}= \tilde c_{4}
$.
We say that $K$ is  {\em non-degenerate}, if it is 
 non-degenerate of all degrees. 
  \item in case $n\geq 5$, if $  \{ \nabla K  = 0 \} \cap \{ \Delta K = 0 \} = \emptyset$, i.e. 
  \eqref{eq:nd} holds. 
 \end{enumerate}
 \end{definition}

 \begin{remark}\label{rem_none_degeneracy}
 Non-degeneracy in case $n=4$ implies the existence of a least eigenvalue
 \begin{equation*}
 \mathcal{M}_{x_{1},\ldots,x_{k}}{\bf{x}}_{x_{1},\ldots,x_{k}}=\lambda_{x_{1},\ldots,x_{k}}{\bf{x}}_{x_{1},\ldots,x_{k}}
 \; \quad \text{ with }\;  \quad 
 \lambda_{x_{1},\ldots,x_{k}} \neq 0
 \end{equation*}
 and such that $\lambda_{x_{1,}\ldots,x_{k}}$ is simple and has a positive eigenvector, i.e. 
 \begin{equation*}
 {\bf{x}}_{x_{1},\ldots,x_{k}}
 =
 (x^{1}_{x_{1},\ldots,x_{k}},\ldots,x^{k}_{x_{1},\ldots,x_{k}})\;
 \text{ with }\; 
 {\bf{x}}^{l}_{x_{1},\ldots,x_{k}}>0\; \text{ for all } 1\leq l \leq k.
 \end{equation*}
 \end{remark}
 
 \medskip

 \begin{thm}
 \label{lem_top_down_cascade} 
 Let $\mathcal{M}_{x_{1},\ldots,x_{k}}$ be as in Definition \ref{def_none-degeneracy}, and 
 suppose that 
 \begin{equation*}
  \left\{
  \begin{matrix*}[l]
  K \; \text{ is non-degenerate of degree }\; q & 
  \; \text{ for }\; n=4
  \\
  K\; \text{ is non-degenerate } &
  \; \text{ for }\; n\geq 5
  \end{matrix*}
  \right\}.
  \end{equation*} 
 Then for $\varepsilon>0$ sufficiently small there exists $c>0$ such that for any 
 $ 
 u\in V(q,\varepsilon)$  with  $\; k_{\tau}=1
 $ 
 there holds 
 \begin{equation*}
 \vert \partial J_{\tau}(u)\vert
 \geq
 c\big( 
 \tau +\sum_{r\neq s}\frac{\vert \nabla K_{r}\vert}{\lambda_{r}}+\frac{1}{\lambda_{r}^{2}}
 +
 \big\vert 1
 -
 \frac{\alpha^{2}}{\alpha_{K,\tau}^{p+1}}
 \frac{K_{r}}{\lambda_{r}^{\theta}}\alpha_{r}^{p-1} \big\vert
 +\varepsilon_{r,s}\big),
 \end{equation*}
cf. \eqref{eq:akt},  unless 
 there is a violation of at least one of the four conditions 
 \begin{enumerate}[label=(\roman*)]
  \item \quad 
 $\tau>0$; 
  \item \quad 
 there exists 
 $
 x_{i}\neq x_{j}
 \in 
 \left\{ 
 \begin{matrix*}[l]
 \{\nabla K=0\}\cap \{\tilde c_{2}\frac{\Delta K}{K}  
 +
 \tilde c_{3}H<0\} &\text{for }\;n = 4,
 \\
  \{\nabla K=0\}\cap \{\Delta K<0\} &\text{for }\; n\geq 5
 \end{matrix*}
 \right\}
 $
 and
 $d(a_{i},x_{i})=O(\frac{1}{\lambda_{i}})$;
  \item \quad 
\resizebox{.9\hsize}{!}
{$
\left\{
\begin{matrix*}[l]
\alpha_{j}=
\Theta\,\cdot
\left(
\frac{\lambda_{j}^{\theta}}{K_{j}}
\biggr(
1 
+
\frac{1}{8}
\big(
\frac{\Delta K_{j}}{K_{j}\lambda_{j}^{2}}
-
60
\frac{H_{j}}{\lambda_{j}^{2 }}
-
\frac
{
\sum_{k}
(\frac{\Delta K_{k}}{K_{k}^{2}\lambda_{k}^{2}}
-
60
\frac{H_{k}}{K_{k}\lambda_{k}^{2 }}
)
}
{
\sum_{k}\frac{1}{K_{k}}
}
\big)
\biggr)
\right)^{\frac{1}{p-1}}
+
o(\frac{1}{\lambda_{j}^{2}})
&
\text{for } \;n=4,
\\
\alpha_{j}=\Theta\,\cdot(\frac{\lambda_{j}^{\theta}}{K_{j}})^{\frac{1}{p-1}}+o(\frac{1}{\lambda_{j}^{2}})
&\text{for }\; n\geq 5
\end{matrix*}
\right\};
$}

  \item \quad 
 $
 \left\{ 
 \begin{matrix*}[l]
 \mathcal{M}_{x_{1},\ldots, x_{q}}>0 \;\text{ and } \; \lambda_{j}=\frac{\sigma_{j}+o(1)}{\sqrt{\tau}}
 &
 \text{for }\; n=4, 
 \\
 \tilde c_{1}\tau
 =
 -
 \tilde c_{2}\frac{\Delta K_{k}}{K_{k}\lambda_{k}^{2}} 
 +
 o(\frac{1}{\lambda_{k}^{2}})
 & 
 \text{for }\; n\geq 5 
 \end{matrix*}
 \right\}
 $
 \end{enumerate}
 for all $j\neq i,j=1,\ldots,q$, where $ \sigma=(\sigma_{1},\ldots,\sigma_{q})$ in case $n=4$ is the  unique solution of 
 \begin{equation*}
\tilde  c_{1}
 \begin{pmatrix}
 \frac{ \sigma_{1} }{K(x_{1})} 
 \\ 
 \vdots \\
 \frac{ \sigma_{q}}{K(x_{q})}
 \end{pmatrix}
 =
 \mathcal{M}_{x_{1},\ldots,x_{q}}
 \begin{pmatrix}
 \frac{1}{\sigma_{1}}\\
 \vdots \\
 \frac{1}{\sigma_{q}}
 \end{pmatrix}
 \;\text{ with }\, \sigma_{j}>0,
 \end{equation*}
while 
 $\Theta$ is given in Remark \ref{r:precision}.
In the latter case there holds \,$\lambda_{1}\simeq \ldots \lambda_{q}\simeq \lambda=\frac{1}{\sqrt{\tau}}$\, and 
setting 
$$a_{j}=\exp_{g_{x_{j}}}(\bar a_{j})$$
we still have up to an error $o(\frac{1}{\lambda^{3}})$ the lower bound 
\begin{equation*}
\begin{split}
\vert \partial J  (u) \vert 
\gtrsim &
 \sum_{j}
 ( \vert
\tau
+
\frac{1}{2}\frac{\Delta K(x_{j})}{K(x_{j})\lambda_{j}^{2}} 
+
12
[
\frac{H(x_{j})}{\lambda_{j}^{2}}
+
\sum_{j\neq i }\sqrt{\frac{K(x_{j})}{K(x_{i})}} \frac{G_{g_{0}}(x_{i},x_{j})}{\gamma_{n}\lambda_{i}\lambda_{j}}
]
 \vert )
 \\
& 
+
\sum_{j}
(\vert 
\frac{\bar a_{j}}{\lambda_{j}}
+
\frac{1}{3}
(\nabla^{2}K(x_{j}))^{-1}
\frac{\nabla \Delta K(x_{j})}{\lambda_{j}^3}
+
8
  \sum_{j\neq i}
\sqrt{  \frac{K^{3}(x_{j})}{K(x_{i})} }
(\nabla^{2}K(x_{j}))^{-1}
  \frac{\nabla_{x_{j}}G_{g_{0}}(x_{i},x_{j})}{\gamma_{n}\lambda_{i}\lambda_{j}^{2}}
  \vert )
 \\
 & +
\sum_{j} 
\vert 
\alpha_{j}
-
\Theta \cdot
\sqrt[p-1]
{
\frac{\lambda_{j}^{\theta}}{K(a_{j})}
(1
+
\frac{1}{8}
\left(
\frac{\Delta K(x_{j})}{K(x_{j})\lambda_{j}^{2}}
-
60
\frac{H(x_{j})}{\lambda_{j}^{2 }}
-
\frac
{
\sum_{k}
(\frac{\Delta K(x_{k})}{K(x_{k})^{2}\lambda_{k}^{2}}
-
60
\frac{H(x_{k})}{K(x_{k})\lambda_{k}^{2 }}
)
}
{
\sum_{k}\frac{1}{K(x_{k})}
}
\right)
)
}
\vert 
  \end{split}
 \end{equation*}
in case $n=4$ and 
\begin{equation*}
\begin{split}
\vert \partial J  (u) \vert 
\gtrsim &
\sum_{j}
\vert
\tau
  +
\frac{2}{9}\frac{\Delta K(x_{j})}{K(x_{j})\lambda_{j}^{2}} 
+
\frac{512}{9\pi}
[
\frac{H(x_{j})}{\lambda_{j}^{3}}
+
\sum_{j\neq i }\sqrt{\frac{K(x_{j})}{K(x_{i})}} \frac{G_{g_{0}}(x_{i},x_{j})}{\gamma_{n}(\lambda_{i}\lambda_{j})^{\frac{3}{2}}}
] 
 \vert
  \\
& 
+
\sum_{j}
\vert 
\frac{\bar a_{j}}{\lambda_{j}}
+
\frac{\check c_{4}}{\check c_{3}}
(\nabla^{2}K(x_{j}))^{-1}
\frac{\nabla \Delta K(x_{j})}{\lambda_{j}^3} \vert 
\\
 & +
\sum_{j}
\vert 
\alpha_{j}
-
\Theta \cdot
\sqrt[p-1]
{
\frac{\lambda_{j}^{\theta}}{K(a_{j})}
(
1
-
\frac{1}{90}
\left(
\frac{\Delta K(x_{j})}{K(x_{j})\lambda_{j}^{2}}
+
\frac{2816}{\pi}
\frac{H(x_{j})}{\lambda_{j}^{3}}
-
\frac
{
\sum_{k}
(\frac{\Delta K(x_{k})}{K(x_{k})^{2}\lambda_{k}^{2}}
+
\frac{2816}{\pi}
\frac{H(x_{k})}{K(x_{k})\lambda_{k}^{3}}
)
}
{
\sum_{k}\frac{1}{K(x_{k})}
}
\right)
)
}
\vert
  \end{split}
 \end{equation*}
 in case $n=5$ and 
\begin{equation*}
\vert \partial J_{\tau}(u)\vert
\gtrsim 
\sum_{j}
(
\vert \tau+\frac{\tilde c_{2}}{\tilde c_{1}}\frac{\Delta K(x_{j})}{K(x_{j})\lambda_{j}^{2}}\vert
+
 \vert
\frac{ \bar a_{j}}{\lambda_{j}}
  +
\frac{  \check{c}_{4}}{\check{c}_{3}} (\nabla^{2}K(x_{j}))^{-1}\frac{\nabla \Delta K(x_{j})}{\lambda_{j}^3}
\vert
+
\vert 
\alpha_{j}
-
\Theta \cdot
\sqrt[p-1]
{
\frac{\lambda_{j}^{\theta}}{K(a_{j})}}
\vert
)
\end{equation*}
in case $n\geq 6$. 
The constants appearing above are defined by
$\bar c_{0}=\int_{\R^{n}}\frac{dx}{(1+r^{2})^n}$,  
\begin{equation*}
\tilde c_{1}=\frac{n(n-1)(n-2)^{2}}{\bar c_{0}^{\frac{n-2}{n}}}\underset{\R^{n}}{\int}\frac{1-r^{2}}{(1+r^{2})^{n+1}}\ln\frac{1}{1+r^{2}}dx
, \quad 
\tilde c_{2}=-\frac{(n-1)(n-2)}{\bar c_{0}^{\frac{n-2}{n}}}\underset{\R^{n}}{\int}\frac{r^{2}(1-r^{2})}{(1+r^{2})^{n+1}}dx
\end{equation*}
and 
\begin{equation*}
\check c_{3}=\int_{\R^{n}}\frac{4(n-1)(n-2)}{(1+r^{2})^{n}}dx
,\quad
\check c_{4}=\int_{\R^{n}}\frac{2(n-1)r^{2}}{(1+r^{2})^{n}}dx.
\end{equation*}
 \end{thm}
 
 \medskip 
 
 The differences in the above expressions for $n = 5$ and $n \geq 6$ is caused by a 
different decay of bubble functions causing
stronger mutual interactions  in lower dimension.

 \begin{remark}\label{r:precision}
 Under non-degeneracy conditions, Theorem \ref{lem_top_down_cascade} has the following immediate implications.  
 \begin{enumerate}
  \item In case $\tau=0$ there are no solutions of $\partial J(u)=\partial J_{0}(u)=0$ in $V(q,\varepsilon)$, cf.  Theorem 1.4 in \cite{cl2}.
  \item In case $\tau>0$  every solution $\partial J_{\tau}(u)=0$ in $V(q,\varepsilon)$ satisfies
 \begin{equation*}
 \lambda_{1}\simeq \ldots \simeq \lambda_{q}\simeq \frac{1}{\sqrt{\tau}}
  \end{equation*}
 and has isolated simple blow-ups occurring close to
 \begin{equation*}
 \Bigg\{
 \begin{matrix*}[l]
 &\{\nabla K=0\}\cap \{ \tilde c_{2}\frac{\Delta K}{K}+\tilde c_{3}H<0\} 
 &\text{for $n=4$}
 \vspace{8pt}_{}
 \\ 
 &\{\nabla K=0\}\cap \{\Delta K<0\} &\text{for $n\geq 5$}.
 \end{matrix*}
 \end{equation*}
 \item The  $\alpha_{j},\lambda_{j}$ and $a_{j}$'s are determined to a precision  $o(\tau^{\frac{3}{2}})+O(\vert \partial J_{\tau}(u)\vert)$.
Indeed, for e.g. $n=6$
 $$\vert \tau+\frac{\tilde c_{2}}{\tilde c_{1}}\frac{\Delta K(x_{j})}{K(x_{j})\lambda_{j}^{2}}\vert$$
determines $\lambda_{j}$ up to the latter error 
 from  $\tau$ and $x_{j}$, whence $a_{j}$ is determined as well by 
$$ \vert
\frac{ \bar a_{j}}{\lambda_{j}}
  +
\frac{  \check{c}_{4}}{\check{c}_{3}} (\nabla^{2}K(x_{j}))^{-1}\frac{\nabla \Delta K(x_{j})}{\lambda_{j}^3}
\vert
$$ 
from $\lambda_{j}$ and $x_{j}$, and finally up to the multiplicative constant
$\Theta$
also $\alpha_{j}$ is determined by 
$$
\vert 
\alpha_{j}
-
\Theta \cdot
\sqrt[p-1]
{
\frac{\lambda_{j}^{\theta}}{K(a_{j})}}
\vert
$$
from $\lambda_{j}, a_{j}$ and $\tau$, recalling $\theta=\frac{n-2}{2}\tau$ and $p=\frac{n+2}{n-2}-\tau$.
As for the multiplicative constant we have
\begin{equation*}
1
=
k_{\tau}
=
\int K(\alpha^{i}\varphi_{i}+v)^{p+1}d\mu_{g_{0}}
=
\int K(\alpha^{i}\varphi_{i})^{p+1}
=
  \sum_{i}
\frac{K(a_{i})}{\lambda_{i}^{\theta}} \alpha_{i}^{p+1}
  \left(
  \bar c_{0}
  +
  \bar c_{1} \tau  
  +
\bar c_{2}\frac{\Delta K(x_{i})}{K(x_{i})\lambda_{i}^{2}} 
  ) 
  \right)
\end{equation*}
up to some $o(\tau^{\frac{3}{2}})$, cf. \eqref{non_linear_v_part_interaction}, Lemma \ref{lem_v_part_gradient}, Lemma \ref{lem_interactions} and \eqref{intergral_sum_of_bubble_nonlinear_evaluated}, whence  
\begin{equation*}
1
=
\Theta^{p-1}  \sum_{i}
\alpha_{i}^{2}
  \left(
  \bar c_{0}
  +
  \bar c_{1} \tau  
  +
\bar c_{2}\frac{\Delta K(x_{i})}{K(x_{i})\lambda_{i}^{2}} 
  ) 
  \right)
  =
\Theta^{p+1}
 \sum_{i}
\left(\frac{\lambda_{j}^{\theta}}{K(a_{j})}\right)^{\frac{2}{p-1}}
  \left(
  \bar c_{0}
  +
  \bar c_{1} \tau  
  +
\bar c_{2}\frac{\Delta K(x_{i})}{K(x_{i})\lambda_{i}^{2}} 
  ) 
  \right)
\end{equation*}
up to the same error and so the multiplicative constant $\Theta$ is determined as well. 
 \end{enumerate}
 \end{remark}
 
 \begin{proof}[Proof of Theorem \ref{lem_top_down_cascade}]
 First we note that $k_{\tau}=1$ implies, that all the $\alpha_{i}$ do not tend to infinity and least one of them does not 
 approach zero. Hence by definition of $V(q,\varepsilon)$  all the $\alpha_{i}$  are uniformly bounded away from zero and infinity. 
 Secondly,  if  for some index $j=1,\ldots,q$ we have 
 \begin{equation*}
 \big\vert 1- \frac{\alpha^{2}}{\alpha_{K,\tau}^{p+1}}\frac{K_{j}}{\lambda_{j}^{\theta}}\alpha_{j}^{p-1}
 \big \vert
 \gg
 \tau
 +
 \sum_{r\neq s} 
 \frac{\vert \nabla K_{r}\vert}{\lambda_{r}}
 +
 \frac{1}{\lambda_{r}^{2}}
 +
 \varepsilon_{r,s},
 \end{equation*}
then the claim follows from  Lemma \ref{lem_alpha_derivatives_at_infinity}, whence we may henceforth assume
 that for all $j = 1, \dots, q$
 \begin{equation}\label{cascade_alpha_constraint}
 \frac{\alpha^{2}}{\alpha_{K,\tau}^{p+1}}\frac{K_{j}}{\lambda_{j}^{\theta}}\alpha_{j}^{p-1}
 =
 1
 +
 O
 \big(
 \tau
 +
 \sum_{r\neq s} 
 \frac{\vert \nabla K_{r}\vert}{\lambda_{r}}
 +
 \frac{1}{\lambda_{r}^{2}}
 +
 \varepsilon_{r,s}
 \big).
 \end{equation}
 Thus we have to show 
 \begin{equation}\label{cascade_what_is_to_show}
 \vert \partial J_{\tau}(u)\vert
 \gtrsim
 \tau +\sum_{j=1}^{q}\frac{\vert \nabla K_{j}\vert}{\lambda_{j}}+\frac{1}{\lambda_{j}^{2}}+\sum_{r\neq s}\varepsilon_{r,s}
 \end{equation}
 and arguing by contradiction we may assume that  
 \begin{equation*}
 \vert \partial J_{\tau}(u)\vert
 \lesssim
 \tau +\sum_{j=1}^{q}\frac{\vert \nabla K_{j}\vert}{\lambda_{j}}+\frac{1}{\lambda_{j}^{2}}+\sum_{r\neq s}\varepsilon_{r,s}.
 \end{equation*}
 Then by Lemmata \ref{lem_lambda_derivatives_at_infinity} and \ref{lem_a_derivatives_at_infinity} we have
 \begin{flalign*}
 \text{($a$)}\quad 
 \partial J_{\tau}(u)\phi_{3,j}
 = &
 \frac{-\alpha_{j}}{(\alpha_{K,\tau}^{\frac{2n}{n-2}})^{\frac{n-2}{n}}}
 (
 \check{c}_{3}\frac{\nabla K_{j}}{K_{j}\lambda_{j}}
 +
 \check{b}_{3}
 \sum_{j\neq i}
 \frac{\alpha_{i}}{\alpha_{j}}
 \frac{\nabla_{a_{j}}}{\lambda_{j}}\varepsilon_{i,j}
 );
 \MoveEqLeft[1]
 \end{flalign*}
 \begin{flalign*}
 \text{($\lambda$)} \quad 
 \partial  J  (u) \phi_{2,j}
 = &
 \frac{\alpha_{j}}{(\alpha_{K,\tau}^{\frac{2n}{n-2}})^{\frac{n-2}{n}}}
 \bigg(
 \tilde c_{1}\tau
 +
 \tilde c_{2}\frac{\Delta K_{j}}{K_{j}\lambda_{j}^{2}} 
 -
 \tilde b_{2}\sum_{j\neq i }\frac{\alpha_{i}}{\alpha_{j}}\lambda_{j}\partial_{\lambda_{j}}\varepsilon_{i,j} 
 \bigg)
 \MoveEqLeft[1]
 \end{flalign*}
 up to some errors of the form 
 $ 
 O(\frac{1}{\lambda_{j}^{3}})
 +
 O
 (
 \tau^{2}
 +
 \sum_{r\neq s} 
 \frac{\vert \nabla K_{r}\vert^{2}}{\lambda_{r}^{2}}
 +
 \frac{1}{\lambda_{r}^{4}}
 +
 \varepsilon_{r,s}^{\frac{n+2}{n}}
 ),
  $ 
 where we have to add for ($\lambda$) $\tilde d_{1}\frac{H_{i}}{\lambda_{i}^{2}}$ to $\tilde c_{2}\frac{\Delta K_{i}}{K_{i}\lambda_{i^{2}}}$ 
  in case $n=4$. 
 Ordering indices so that 
 $
 \lambda_{1}\geq \ldots \geq \lambda_{q}
 \Longleftrightarrow
 \frac{1}{\lambda_{1}}\leq \ldots \leq \frac{1}{\lambda_{q}}
 $
 and recalling \eqref{eq:eij}, we have
 \begin{equation*}
 -\lambda_{j}\partial_{\lambda_{j}}\varepsilon_{i,j}
 =
 \frac{n-2}{2}
 \frac
 {
 \frac{\lambda_{j}}{\lambda_{i}}
 -
 \frac{\lambda_{i}}{\lambda_{j}}
 +
 \lambda_{i}\lambda_{j}\gamma_{n}G_{g_{0}}^{\frac{2}{2-n}}(a_{i},a_{j})
 }
 {
 (
 \frac{\lambda_{j}}{\lambda_{i}}
 +
 \frac{\lambda_{i}}{\lambda_{j}}
 +
 \lambda_{i}\lambda_{j}\gamma_{n}G_{g_{0}}^{\frac{2}{2-n}}(a_{i},a_{j})
 )^{\frac{n}{2}}
 }
 \end{equation*}
 and therefore 
 \begin{equation}\label{lambda_interaction_derivative_lower_bound}
 \lambda_{j}\partial_{\lambda_{j}}\varepsilon_{i,j}
 =
 \frac{2-n}{2}
 \varepsilon_{i,j}
 +
 O(\frac{1}{\lambda_{j}^{4}}+\varepsilon_{i,j}^{\frac{n+2}{n}})
 \; \quad \text{ in case }\;
 j<i\; \text{ or }\; 
 d_{g_{0}}(a_{i},a_{j})
 \neq 
 o(1).
 \end{equation}
 From ($a$) and ($\lambda$) above  we  find uniformly 
  bounded vector fields $A_{1},\Lambda_{1}$  on
 $V(q,\varepsilon)$  such that
 \begin{flalign*}
 (\mathbf{A_{1}}) \quad 
 \partial J_{\tau}(u)A_{1}
 \gtrsim &
 \frac{\vert \nabla K_{1}\vert }{\lambda_{1}} 
 + 
 O(\frac{1 }{\lambda_{1}^{3}}+\sum_{1\neq i}\varepsilon_{1,i}) 
  +
 O
 \big(
 \tau^{2}
 +
 \sum_{r\neq s} 
 \frac{\vert \nabla K_{r}\vert^{2}}{\lambda_{r}^{2}}
 +
 \frac{1}{\lambda_{r}^{4}}
 +
 \varepsilon_{r,s}^{\frac{n+2}{n}}
 \big);
 \MoveEqLeft[1]
 \end{flalign*}
 \begin{flalign*}
 \mathbf{(\Lambda_{1}}) \quad 
 \partial J_{\tau}(u)\Lambda_{1}
 \simeq &
 \tilde c_{1}\tau+\tilde c_{2}\frac{\Delta K_{1}}{K_{1}\lambda_{1}^{2}}
 +
 \tilde c_{4}\sum_{1\neq j}
 \frac{\alpha_{i}}{\alpha_{1}}\varepsilon_{1,i} 
 +
 O(\frac{1}{\lambda_{1}^{3}})  
  +
 O
 \big(
 \tau^{2}
 +
 \sum_{r\neq s} 
 \frac{\vert \nabla K_{r}\vert^{2}}{\lambda_{r}^{2}}
 +
 \frac{1}{\lambda_{r}^{4}}
 +
 \varepsilon_{r,s}^{\frac{n+2}{n}}
 \big)
 \MoveEqLeft[1]
 \end{flalign*}
 with $\tilde{c}_4 = \frac{n-2}{2} \tilde{b}_2$, and combining $X_{1}=\Lambda_{1}+\epsilon A_{1}$ with some $\epsilon>0$ small and  fixed such that we keep a positive coefficient in front of $\varepsilon_{1,i}$, we get 
 \begin{equation*}
 (\mathbf{C_{1}})
 \quad 
 B_{1}
 =
 \partial J_{\tau}(u)X_{1}
 \gtrsim 
 \big(\tilde c_{1}\tau+\tilde c_{2}\frac{\Delta K_{1}}{K_{1}\lambda_{1}^{2}}\big) 
 +
 \epsilon
 \big(
 \frac{\vert \nabla K_{1}\vert}{\lambda_{1}}
 +
 \sum_{1\neq i}\varepsilon_{1,i}
 \big)
 +
 O(\frac{1}{\lambda_{1}^{3}})
  +
 O
 (
 \tau^{2}
 +
 \sum_{\underset{r,s>1}{r\neq s}}
 \frac{\vert \nabla K_{r}\vert^{2}}{\lambda_{r}^{2}}
 +
 \frac{1}{\lambda_{r}^{4}}
 +
 \varepsilon_{r,s}^{\frac{n+2}{n}}
 ).
 \end{equation*}
 Likewise from ($a$) and ($\lambda$) we find uniformly bounded vector fields $A_{2},\Lambda_{2}$ defined on 
 $V(q,\varepsilon)$ such that 
 \begin{flalign*}
 (\mathbf{A_{2}}) \quad 
 \partial J_{\tau}(u)A_{2}
 \gtrsim &
 \frac{\vert \nabla K_{2}\vert }{\lambda_{2}} 
 + 
 O\big(\frac{1 }{\lambda_{2}^{3}}+\sum_{2\neq i}\varepsilon_{1,i}\big) 
  +
 O
 \big(
 \tau^{2}
 +
 \sum_{r\neq s} 
 \frac{\vert \nabla K_{r}\vert^{2}}{\lambda_{r}^{2}}
 +
 \frac{1}{\lambda_{r}^{4}}
 +
 \varepsilon_{r,s}^{\frac{n+2}{n}}
 \big); 
 \MoveEqLeft[1]
 \end{flalign*}
 \begin{flalign*}
 \mathbf{(\Lambda_{2}}) \quad
 \partial J_{\tau}(u)\Lambda_{2}
 \simeq &
 \tilde c_{1}\tau+\tilde c_{2}\frac{\Delta K_{2}}{K_{2}\lambda_{2}^{2}}
 +
 \tilde c_{4}\sum_{2<i}\frac{\alpha_{i}}{\alpha_{2}}\varepsilon_{1,i} 
 +
 O\big(\frac{1}{\lambda_{2}^{3}}+\varepsilon_{1,2}\big)
  +
 O
 \big(
 \tau^{2}
 +
 \sum_{r\neq s} 
 \frac{\vert \nabla K_{r}\vert^{2}}{\lambda_{r}^{2}}
 +
 \frac{1}{\lambda_{r}^{4}}
 +
 \varepsilon_{r,s}^{\frac{n+2}{n}}
 \big)
 \MoveEqLeft[1]
 \end{flalign*}
 and combining them as $X_{2}=\Lambda_{2}+\epsilon A_{2}$ with $\epsilon>0$ small  we obtain 
 \begin{equation*}
 \begin{split}
 B_{2}=\partial J_{\tau}(u)X_{2}
 \gtrsim &
 \big(\tilde c_{1}\tau+\tilde c_{2}\frac{\Delta K_{2}}{K_{2}\lambda_{2}^{2}}\big) 
 +
 \epsilon
 \big(
 \frac{\vert \nabla K_{2}\vert}{\lambda_{2}}
 +
 \sum_{2<i}\varepsilon_{2,i}
 \big)
 +
 O
 \big(
 \frac{1}{\lambda_{2}^{3}}
 +
 \varepsilon_{1,2}
 \big)
  +
 O
 \big(
 \tau^{2}
 +
 \sum_{r\neq s}
 \frac{\vert \nabla K_{r}\vert^{2}}{\lambda_{r}^{2}}
 +
 \frac{1}{\lambda_{r}^{4}}
 +
 \varepsilon_{r,s}^{\frac{n+2}{n}}
 \big). 
 \end{split}
 \end{equation*}
  Therefore combining $B_1$ and $B_2$ so that the coefficient of $\varepsilon_{i,j}$ is positive
 \begin{flalign*}
 (\mathbf{C_{2}}) \quad 
 B_{1}+\epsilon B_{2}
 \gtrsim &
 \sum_{j=1}^{2}
 \big[
 \epsilon^{j}\big(\tilde c_{1}\tau+\tilde c_{2}\frac{\Delta K_{j}}{K_{j}\lambda_{j}^{2}}\big) 
 +
 \epsilon^{j+1}
 \big(
 \frac{\vert \nabla K_{j}\vert}{\lambda_{j}}
 +
 \sum_{j\neq i}\varepsilon_{j,i}
 \big)
 \big]
 +
 O\big(\frac{1}{\lambda_{2}^{3}}\big)
 +
 O
 \big(
 \tau^{2}
 +
 \sum_{\underset{r,s>2}{r\neq s}}
 \frac{\vert \nabla K_{r}\vert^{2}}{\lambda_{r}^{2}}
 +
 \frac{1}{\lambda_{r}^{4}}
 +
 \varepsilon_{r,s}^{\frac{n+2}{n}}
 \big). 
 \MoveEqLeft[1]
 \end{flalign*}
  Iteratively, for all $k=1,\ldots,q$ we can find uniformly bounded vector fields $A_k, \Lambda_{k}$ such that 
 \begin{flalign*}
 (\mathbf{A_{k}}) \quad 
 \partial J_{\tau}(u)A_{k}
 \gtrsim &
 \frac{\vert \nabla K_{k}\vert }{\lambda_{k}} 
 + 
 O\big(\frac{1 }{\lambda_{k}^{3}}+\sum_{k\neq i}\varepsilon_{k,i}\big) 
  +
 O
 \big(
 \tau^{2}
 +
 \sum_{r\neq s} 
 \frac{\vert \nabla K_{r}\vert^{2}}{\lambda_{r}^{2}}
 +
 \frac{1}{\lambda_{r}^{4}}
 +
 \varepsilon_{r,s}^{\frac{n+2}{n}}
 \big);
 \MoveEqLeft[1]
 \end{flalign*}
 \begin{flalign*}
 (\mathbf{\Lambda_k}) \quad
 \partial J_{\tau}(u)\Lambda_{k}
 \simeq &
 \tilde c_{1}\tau+\tilde c_{2}\frac{\Delta K_{k}}{K_{k}\lambda_{k}^{2}}
 +
 \tilde c_{4}\sum_{k<i}\frac{\alpha_{i}}{\alpha_{k}}\varepsilon_{k,i} +O\big(\frac{1}{\lambda_{k}^{3}}+\sum_{k>i}\varepsilon_{k,i}\big) 
  +
 O
 \big(
 \tau^{2}
 +
 \sum_{r\neq s} 
 \frac{\vert \nabla K_{r}\vert^{2}}{\lambda_{r}^{2}}
 +
 \frac{1}{\lambda_{r}^{4}}
 +
 \varepsilon_{r,s}^{\frac{n+2}{n}}
 \big);
 \MoveEqLeft[1]
 \end{flalign*}
 \begin{flalign*}
 (\mathbf{C_{k}}) \quad 
 \sum_{j=1}^{k}\epsilon^{j}B_{j}
 \gtrsim &
 \sum_{j=1}^{k}
 \big[
 \epsilon^{j}
 \big(\tilde c_{1}\tau+\tilde c_{2}\frac{\Delta K_{j}}{K_{j}\lambda_{j}^{2}}\big) 
 +
 \epsilon^{j+1}
 \big(
 \frac{\vert \nabla K_{j}\vert}{\lambda_{j}}
 +
 \sum_{j\neq i}\varepsilon_{j,i}
 \big)
 \big]
 +
 O\big(\frac{1}{\lambda_{k}^{3}}\big)
  +
 O
 \big(
 \tau^{2}
 +
 \sum_{r\neq s}
 \frac{\vert \nabla K_{r}\vert^{2}}{\lambda_{r}^{2}}
 +
 \frac{1}{\lambda_{r}^{4}}
 +
 \varepsilon_{r,s}^{\frac{n+2}{n}}
 \big),
 \MoveEqLeft[1]
 \end{flalign*}
 where we have to add $\tilde c_{3}\frac{H_{j}}{\lambda_{j}^{2}}$ to  $\tilde c_{2}\frac{\Delta K_{j}}{K_{j}\lambda_{j}^{2}}$  in case $n=4$, where 
\begin{equation}\label{def_tilde_c3}
 \tilde c_{3}=\tilde d_{1}
\end{equation}
 In particular 
 \begin{flalign*}
 (\mathbf{C_{q}}) \quad 
 \sum_{j=1}^{k}\epsilon^{j}B_{j}
 \gtrsim &
 \sum_{j=1}^{k}
 \big[
 \epsilon^{j}
 \big(\tilde c_{1}\tau+\tilde c_{2}\frac{\Delta K_{j}}{K_{j}\lambda_{j}^{2}}\big) 
 +
 \epsilon^{j+1}
 \big(
 \frac{\vert \nabla K_{j}\vert}{\lambda_{j}}
 +
 \sum_{j\neq i}\varepsilon_{j,i}
 \big)
 \big]  
 +
 O
 \big(
 \tau^{2}
 +
 \sum_{r\neq s}
 \frac{\vert \nabla K_{r}\vert^{2}}{\lambda_{r}^{2}}
 +
 \frac{1}{\lambda_{r}^{4}}
 +
 \varepsilon_{r,s}^{\frac{n+2}{n}}
 \big)
 .
 \MoveEqLeft[1]
 \end{flalign*}
 Then, if either 
 \begin{equation*}
 \begin{split}
 \frac{1}{\lambda_{q}^{2}}
 \ll &
 \tau 
 +
 \sum_{j=1}^{q} 
 \frac{\vert \nabla K_{j}\vert}{\lambda_{j}}
 +
 \sum_{r\neq s}\varepsilon_{r,s}
 \quad \; \quad \text{ or  }\; \quad \quad 
 \frac{1}{\lambda_{q}^{2}}
 \gg
 \tau 
 +
 \sum_{j=1}^{q} 
 \frac{\vert \nabla K_{j}\vert}{\lambda_{j}}
 +
 \sum_{r\neq s}\varepsilon_{r, s}, 
 \end{split}
 \end{equation*}
 we obviously have \eqref{cascade_what_is_to_show} from $(\mathbf{C_{q}})$. 
  Thus we may assume
 \begin{equation}\label{cascade_balancing}
 \frac{1}{\lambda_{q}^{2}}
 \simeq 
 \tau 
 +
 \sum_{j=1}^{q} 
 \frac{\vert \nabla K_{j}\vert}{\lambda_{j}}
 +
 \sum_{r\neq s}\varepsilon_{r,s},
 \end{equation}
 whence we may simplify the above formulas to
 \begin{flalign*}
 (\mathbf{A_{k}}) \quad 
 \partial J_{\tau}(u)A_{k}
 \gtrsim &
 \frac{\vert \nabla K_{k}\vert }{\lambda_{k}} 
 + 
 O\big(\sum_{k\neq i}\varepsilon_{k,i}\big)
 +
 o\big(\frac{1}{\lambda_{q}^{2}}\big);
 \MoveEqLeft[1]
 \end{flalign*}
 \begin{flalign*}
 (\mathbf{\Lambda_k}) \quad
 \partial J_{\tau}(u)\Lambda_{k}
 \simeq &
 \tilde c_{1}\tau+\tilde c_{2}\frac{\Delta K_{k}}{K_{k}\lambda_{k}^{2}}
 +
 \tilde c_{4}\sum_{k<i}\frac{\alpha_{i}}{\alpha_{k}}\varepsilon_{k,i} 
 +
 O\big(\sum_{k>i}\varepsilon_{k,i}\big) 
 +
 o\big(\frac{1}{\lambda_{q}^{2}}\big);
 \MoveEqLeft[1]
 \end{flalign*}
 \begin{flalign*}
 (\mathbf{C_{k}}) \quad 
 \sum_{j=1}^{k}\epsilon^{j}B_{j}
 \gtrsim &
 \sum_{j=1}^{k}
 \big[
 \epsilon^{j}
 \big(\tilde c_{1}\tau+\tilde c_{2}\frac{\Delta K_{j}}{K_{j}\lambda_{j}^{2}}\big) 
 +
 \epsilon^{j+1}
 \big(
 \frac{\vert \nabla K_{j}\vert}{\lambda_{j}}
 +
 \sum_{j\neq i}\varepsilon_{j,i}
  \big )
 \big]
 +
 o\big(\frac{1}{\lambda_{q}^{2}}\big), 
 \MoveEqLeft[1]
 \end{flalign*}
 adding $\tilde c_{3}\frac{H_{j}}{\lambda_{j}^{2}}$ to  $\tilde c_{2}\frac{\Delta K_{j}}{K_{j}\lambda_{j}^{2}}$  for $n=4$. 
 We first consider the pair $(q-1,q)$. 
 Suppose
 \begin{equation*}
 \frac{1}{\lambda_{q-1}^{2}}=o(\frac{1}{\lambda_{q}^{2}}).
 \end{equation*}
 To prove \eqref{cascade_what_is_to_show} we then may assume from $(\mathbf{C_{q-1}})$ and \eqref{cascade_balancing} that also 
 $
 \tau + \sum_{r\neq s}\varepsilon_{r,s}=o(\frac{1}{\lambda_{q}^{2}}),
 $
 since 
 \begin{equation*}
 \sum_{j=1}^{q-1}\sum_{j\neq i}\varepsilon_{i,j}=\sum_{q-1\geq r\neq s}\varepsilon_{r,s} 
 =
 \sum_{r\neq s}\varepsilon_{r,s}. 
 \end{equation*}
 As the coefficient of $\lambda_{q}^{-2}$ in $(\mathbf{\Lambda_{q}})$ is non zero 
 by non-degeneracy, 
  \eqref{cascade_what_is_to_show} follows. So we may assume
 \begin{equation*}
  \frac{1}{\lambda_{q-1}}\simeq \frac{1}{\lambda_{q}},
 \end{equation*}
 and therefore, still by \eqref{cascade_balancing},
 \begin{equation*}
 \vert \nabla K_{q-1}\vert\lesssim\frac{1}{\lambda_{q-1}},\quad \vert \nabla K_{q}\vert\lesssim
 \frac{1}{\lambda_{q}}. 
 \end{equation*}
 So, if $a_{q-1}$ is close to $a_{q}$, these points are close to the same critical point of $K$, which, as
 $K$ is Morse,  implies $d(a_{q-1},a_{q})) \lesssim \frac{1}{\lambda_{q}} \simeq 
 \frac{1}{\lambda_{q-1}}$. This however contradicts the fact that by Proposition \ref{blow_up_analysis} 
 \begin{equation*}
 \varepsilon_{q-1,q}
 \simeq 
 \frac{1}{(\lambda_{q-1}\lambda_{q}d^{2}(a_{q-1},a_{q}))^{\frac{n-2}{2}}}
 \longrightarrow 0. 
 \end{equation*}
 Therefore for the pair $(q-1,q)$ we may assume 
 \begin{equation*}
 \vert \nabla K_{q-1}\vert,\vert \nabla K_{q}\vert \lesssim \frac{1}{\lambda_{q-1}}\simeq \frac{1}{\lambda_{q}}, 
 \quad 
 \hbox{ and } \quad  
 d(a_{q-1},a_{q})>c. 
 \end{equation*}
 In particular in case $n\geq 5$ we have 
$
\varepsilon_{q-1,q}\simeq \frac{1}{\lambda_{q}^{n-2}}=o(\frac{1}{\lambda_{q}^{2}}), 
$
 whereas in case $n=4$
 \begin{equation*}
 \varepsilon_{q-1,q}=\frac{G_{g_{0}}(a_{q-1},a_{q})}{\gamma_{n}\lambda_{q-1}\lambda_{q}}
 +
 O(\frac{1}{\lambda_{q}^{4}}).
 \end{equation*}
 We turn to consider  the triple $(q-2,q-1,q)$. Suppose that 
$ 
 \frac{1}{\lambda_{q-2}^{2}}=o(\frac{1}{\lambda_{q-1}^{2}}).
$ 
 To get \eqref{cascade_what_is_to_show} we then may assume from $(\mathbf{C_{q-2}})$ and \eqref{cascade_balancing} that 
 \begin{equation*}
 \tau + \sum_{q-2\geq r\neq s}\varepsilon_{r,s}=o(\frac{1}{\lambda_{q}^{2}})
 \end{equation*}
 as well. But then clearly in case $n\geq 5$ we obtain \eqref{cascade_what_is_to_show} from 
 $(\mathbf{\Lambda_{q-1}})$ or $(\mathbf{\Lambda_{q}})$, since
$
 \varepsilon_{q-1,q}
 =
 o(\lambda_{q}^{-2})
$
 is already known. In case $n=4$ we have to argue more subtly. From ($\lambda$) we find 
\begin{equation*}
\partial  J  (u) \phi_{2,q-1}
= 
\frac{\alpha_{q-1}}{(\alpha_{K,\tau}^{\frac{2n}{n-2}})^{\frac{n-2}{n}}}
\bigg(
\tilde c_{2}\frac{\Delta K_{q-1}}{K_{q-1}\lambda_{q-1}^{2}} 
+
\tilde c_{3}\frac{H_{q-1}}{\lambda_{q-1}^{2}}
+
\tilde  c_{4}\frac{\alpha_{q}}{\alpha_{q-1}}\frac{G_{g_{0}}(a_{q-1},a_{q})}
{\gamma_{n}\lambda_{q-1}\lambda_{q}} 
\bigg)
\end{equation*}
and
\begin{equation*}
\partial  J  (u) \phi_{2,q}
 = 
\frac{\alpha_{q}}{(\alpha_{K,\tau}^{\frac{2n}{n-2}})^{\frac{n-2}{n}}}
 \bigg(
\tilde  c_{2}\frac{\Delta K_{q}}{K_{q}\lambda_{q}^{2}} 
 +
\tilde  c_{3}\frac{H_{q}}{\lambda_{q}^{2}}
 +
\tilde  c_{4}\frac{\alpha_{q-1}}{\alpha_{q}}\frac{G_{g_{0}}(a_{q-1},a_{q})}
 {\gamma_{n}\lambda_{q-1}\lambda_{q}} 
 \bigg)  \end{equation*}
 up to an error of order $o(\frac{1}{\lambda_{q}^{2}})$, cf. \eqref{lambda_interaction_derivative_lower_bound}. Obviously \eqref{cascade_what_is_to_show} then follows if either 
 \begin{equation*}
 \tilde c_{2}\frac{\Delta K_{q}}{K_{q}\lambda_{q}^{2}} 
 +
\tilde  c_{3}\frac{H_{q}}{\lambda_{q}^{2}}
 >0
 \quad \; \text{ or }\;  \quad 
  \tilde  c_{2}\frac{\Delta K_{q-1}}{K_{q-1}\lambda_{q-1}^{2}} 
 +
\tilde  c_{3}\frac{H_{q-1}}{\lambda_{q-1}^{2}}>0.
 \end{equation*}
 We may thus assume both summands to be negative. 
 Recalling \eqref{cascade_alpha_constraint}, we then obtain
 \begin{equation*}
 \resizebox{.9\hsize}{!}
 {$
 \partial J_{\tau}(u)
 \begin{pmatrix}
 \beta_{q-1}\phi_{2,q-1}\\
 \beta_{q}\phi_{2,q}
 \end{pmatrix}
 =
 \begin{pmatrix}
 \frac{1}{\lambda_{q-1}} & 0 \\
 0 & \frac{1}{\lambda_{q}}
 \end{pmatrix}
 \begin{pmatrix}
\tilde  c_{2}\frac{\Delta K_{q-1}}{K_{q-1}^{2}}  
 +
\tilde  c_{3}\frac{H_{q-1}}{K_{q-1}}
 &
\tilde  c_{4}\frac{G_{0}(a_{q-1},a_{q})}{\gamma_{n}(K_{q-1}K_{q})^{\frac{1}{2}}}
 \\
\tilde  c_{4}\frac{G_{0}(a_{q-1},a_{q})}{\gamma_{n}(K_{q-1}K_{q})^{\frac{1}{2}}}  
 &
\tilde  c_{2}\frac{\Delta K_{q}}{K_{q}^{2}} 
 +
\tilde  c_{3}\frac{H_{q}}{K_{q}}  
 \end{pmatrix}
 \begin{pmatrix}
 \frac{1}{\lambda_{q-1}} \\
 \frac{1}{\lambda_{q}}
 \end{pmatrix}
 $}
 \end{equation*}
 up to an error $o(\frac{1}{\lambda_{q}^{2}})$
 letting 
\begin{equation*}
K_{j}\alpha_{j}\beta_{j}=(\alpha_{K,\tau}^{\frac{2n}{n-2}})^{\frac{n-2}{n}}\;\text{ for }\; j=q-1,q,
\end{equation*}
and thus
$
 \vert \partial J_{\tau}(u)\vert 
 \gtrsim \lambda_{q}^{-2},
 $
 since otherwise  $a_{q-1},a_{q}$ close to $x_{q-1},x_{q}\in \{\nabla K=0\}\cap \{\tilde c_{2}\frac{\Delta K}{K}+\tilde c_{3}H <0\}$
 and
 \begin{equation*}
 \mathcal{M}_{q-1,q}
 =
 \begin{pmatrix}
 \tilde c_{2}\frac{\Delta K(x_{q-1})}{K(x_{q-1})^{2}}  
 +
\tilde  c_{3}\frac{H(x_{q-1})}{K(x_{q-1})}
 &
\tilde  c_{4}\frac{G_{0}(x_{q-1},x_{q})}{\gamma_{n}(K(x_{q-1})K(x_{q}))^{\frac{1}{2}}}
 \\
\tilde  c_{4}\frac{G_{0}(x_{q-1},x_{q})}{\gamma_{n}(K(x_{q-1})K(x_{q}))^{\frac{1}{2}}}  
 &
\tilde  c_{2}\frac{\Delta K(x_{q})}{K(x_{q})^{2}} 
 +
\tilde  c_{3}\frac{H(x_{q})}{K(x_{q})}  
 \end{pmatrix}
 \end{equation*}
 would have after a blow-up for $\tau\longrightarrow 0$ a vanishing eigenvalue with strictly positive eigenvector, which by Remark \ref{rem_none_degeneracy} is impossible. So \eqref{cascade_what_is_to_show} again follows. We may thus assume 
 \begin{equation*}
 \frac{1}{\lambda_{q-2}}\simeq  \frac{1}{\lambda_{q-1}}\simeq \frac{1}{\lambda_{q}}
 \end{equation*}
 and therefore by \eqref{cascade_balancing}
 \begin{equation*}
 \vert \nabla K_{q-2}\vert\lesssim\frac{1}{\lambda_{q-2}}, \quad 
 \vert \nabla K_{q-1}\vert\lesssim\frac{1}{\lambda_{q-1}}, \quad \vert \nabla K_{q}\vert\lesssim
 \frac{1}{\lambda_{q}}.
 \end{equation*}
 So, if $a_{q-2}$ is close to either $a_{q-1}$ or $a_{q}$,  these points are close to the same critical point of $K$, whence 
 \begin{equation*}
 \varepsilon_{q-2,q-1} \simeq 1\; \text{ or }\; \varepsilon_{q-2,1}\simeq 1
 \end{equation*}
 as before, contradicting Proposition \ref{blow_up_analysis}. Thus for
 $(q-2,q-1,q)$ we may assume  
 \begin{equation*}
 \vert \nabla K_{q-2},\vert, \vert \nabla K_{q-1}\vert,\vert \nabla K_{q}\vert\lesssim \frac{1}{\lambda_{q-2}}\simeq \frac{1}{\lambda_{q-1}}\simeq \frac{1}{\lambda_{q}}
 \end{equation*}
 and
 \begin{equation*}
 d(a_{q-2},a_{q-1}),d(a_{q-2},a_{q}),d(a_{q-1},a_{q})>c
 \end{equation*}
 analogously to the previous case of the pair $(q-1,1)$.
 In particular in case $n\geq 5$
 \begin{equation*}
 \varepsilon_{q-2,q-1},\varepsilon_{q-2,q},\varepsilon_{q-1,q}\simeq \frac{1}{\lambda_{q}^{n-2}}=o\big(\frac{1}{\lambda_{q}^{2}}\big), 
 \end{equation*}
 whereas in case $n=4$ up to an error $O(\frac{1}{\lambda_{q}^{4}})$
 \begin{equation*}
 \varepsilon_{q-2,q-1}=\frac{G_{g_{0}}(a_{q-2},a_{q-1})}{\gamma_{n}\lambda_{q-2}\lambda_{q-1}},
 \varepsilon_{q-2,q}=\frac{G_{g_{0}}(a_{q-2},a_{q})}{\gamma_{n}\lambda_{q-2}\lambda_{q}},
 \varepsilon_{q-1,q}=\frac{G_{g_{0}}(a_{q-1},a_{q})}{\gamma_{n}\lambda_{q-1}\lambda_{q}}.
 \end{equation*}
 Iteratively, we then may assume for all $k\neq l=1,\ldots,q$
 \begin{equation*}
 \vert \nabla K_{k}\vert \lesssim \frac{1}{\lambda_{k}}\simeq \frac{1}{\lambda_{l}}\;
 \; \text{ and }\;
 d(a_{k},a_{l})>c.
 \end{equation*}
 In particular $\varepsilon_{k,l}=o(\frac{1}{\lambda_{q}^{2}})$
 for $n\geq 5$ and
 $
 \varepsilon_{k,l}=\frac{G_{g_{0}}(a_{k},a_{l})}{\lambda_{k}\lambda_{l}}
 $
 for $n=4$. But then
 \begin{flalign*}
 (\mathbf{\Lambda_k}) \quad
 \partial J_{\tau}(u)\Lambda_{k}
 \simeq &
 (\tilde c_{1}\tau+\tilde c_{2}\frac{\Delta K_{k}}{K_{k}\lambda_{k}^{2}}) 
 +
 o\big(\frac{1}{\lambda_{q}^{2}}\big)
 \MoveEqLeft[1]
 \end{flalign*}
 in case $n\geq 5$ and thus
 \begin{equation*}
 \vert \partial J_{\tau}(u)\vert \gtrsim 
 \left|
 \tilde c_{1}\tau+\tilde c_{2}\frac{\Delta K_{k}}{K_{k}\lambda_{k}^{2}}
 \right|
 \end{equation*}
 up to some 
 $
 o(\frac{1}{\lambda_{q}^{2}})$.
 Therefore \eqref{cascade_what_is_to_show} holds unless 
 $
 \tilde c_{1}\tau+\tilde c_{2}\frac{\Delta K_{k}}{K_{k}\lambda_{k}^{2}} 
 =
 o(\frac{1}{\lambda_{q}^{2}})$, 
 while now  for $n=4$
 \begin{flalign*}
 \partial  J  (u) \phi_{2,j}
 = &
 \frac{\alpha_{j}}{(\alpha_{K,\tau}^{\frac{2n}{n-2}})^{\frac{n-2}{n}}}
 \bigg(
\tilde c_{1}\tau
 +
\tilde c_{2}\frac{\Delta K_{j}}{K_{j}\lambda_{j}^{2}} 
 +
 \tilde  c_{3}\frac{H_{j}}{\lambda_{j}^{2}}
 +
\tilde c_{4} \sum_{j\neq i }\frac{\alpha_{i}}{\alpha_{j}}
 \frac{G_{g_{0}}(a_{i},a_{j})}{\gamma_{n}\lambda_{i}\lambda_{j}}
 \bigg)
 \end{flalign*}
 up to some $o(\frac{1}{\lambda_{q}^{2}})$, cf.  \eqref{lambda_interaction_derivative_lower_bound},  for all $j=1,\ldots,q$. 
 Obviously \eqref{cascade_what_is_to_show} then follows, if for some $j=1,\ldots,q$
 \begin{equation*}
\tilde c_{2}\frac{\Delta K_{j}}{K_{j}\lambda_{j}^{2}} 
 +
\tilde c_{3}\frac{H_{j}}{\lambda_{j}^{2}}
 >0,
 \end{equation*}
 whence we may assume all these summands to be negative, proving (ii).
 From ($\lambda$) and \eqref{cascade_alpha_constraint} we then have
 \begin{equation*}
 \begin{split}
 \partial  J  (u) (\beta_{j}\phi_{2,j})
 = &
\tilde c_{1}\frac{\tau}{K_{j}}
 +
\tilde c_{2}\frac{\Delta K_{j}}{K_{j}^{2}\lambda_{j}^{2}} 
 +
\tilde c_{3}\frac{H_{j}}{K_{j}\lambda_{j}^{2}}
 +
\tilde c_{4} \sum_{j\neq i }
 \frac{G_{g_{0}}(a_{i},a_{j})}{\gamma_{n}\sqrt{K_{i}K_{j}}\lambda_{i}\lambda_{j}}
 \end{split}
 \end{equation*}
 up to some $o(\frac{1}{\lambda_{q}^{2}})$ letting as before  $\beta_{j}=\frac{(\alpha_{K,\tau}^{\frac{2n}{n-2}})^{\frac{n-2}{n}}}{K_{j}\alpha_{j}}$. Therefore 
 \begin{equation*}
 \vert\partial  J  (u) \vert
 \gtrsim 
 \left|
 \begin{pmatrix}
 \frac{ \tilde c_{1}\tau }{K_{1}} 
 \\ 
 \vdots \\
 \frac{\tilde c_{1}\tau }{K_{q}}
 \end{pmatrix}
 -
 diag(\frac{1}{\lambda_{1}},\ldots,\frac{1}{\lambda_{q}})
 \mathcal{M}_{a_{1},\ldots,a_{q}}
 \begin{pmatrix}
 \frac{1}{\lambda_{1}}\\
 \vdots \\
 \frac{1}{\lambda_{q}}
 \end{pmatrix}
 \right|
 \end{equation*}
 up to the same error. This implies that \eqref{cascade_what_is_to_show} holds true, unless we can solve
\begin{equation}\label{cascade_matrix_condition} \begin{pmatrix}
 \frac{\tilde c_{1}\tau \lambda_{1} }{K_{1}} 
 \\ 
 \vdots \\
 \frac{\tilde c_{1}\tau \lambda_{q}}{K_{q}}
 \end{pmatrix}
 =
 \mathcal{M}_{a_{1},\ldots,a_{q}}
 \begin{pmatrix}
 \frac{1}{\lambda_{1}}\\
 \vdots \\
 \frac{1}{\lambda_{q}}
 \end{pmatrix}
 +
 o\big(\frac{1}{\lambda_{q}}\big)
 \end{equation}
 and we may already assume, by (ii), that  $a_{j}$ is close to 
 $$x_{j}\in \{\nabla K=0\}\cap \{\tilde c_{2}\frac{\Delta K}{K}+\tilde c_{3}H<0\}.$$  
 In particular \eqref{cascade_what_is_to_show} follows in case $\tau=0$ by the non-degeneracy condition on $K$, proving (i).  In case $\tau>0$, 
 writing $\sigma_{j}=\sqrt\tau\lambda_{j}$, we find passing to the limit $\tau \longrightarrow 0$, that there has to exist a solution to 
 \begin{equation}\label{cascade_show_existence_n=4}
 \tilde c_{1}
 \begin{pmatrix}
 \frac{ \sigma_{1} }{K(x_{1})} 
 \\ 
 \vdots \\
 \frac{ \sigma_{q}}{K(x_{q})}
 \end{pmatrix}
 =
 \mathcal{M}_{x_{1},\ldots,x_{q}}
 \begin{pmatrix}
 \frac{1}{\sigma_{1}}\\
 \vdots \\
 \frac{1}{\sigma_{q}}
 \end{pmatrix}.
 \end{equation}
 In particular, testing the above relation with ${\bf{x}}={\bf{x}}_{x_{1},\ldots,x_{q}}$, cf. Remark \ref{rem_none_degeneracy}, we find
 \begin{equation*}
 0
 \leq
\tilde c_{1}\sum_{j}\frac{x_{j}\sigma_{j}}{K_{j}}
 =
 \lambda
 \sum_{j}\frac{x_{j}}{\sigma_{j}}, 
 \end{equation*}
 where $\lambda=\lambda_{x_{1},\ldots,x_{q}}$ is the least eigenvalue of 
 $\mathcal{M}_{x_{1},\ldots,x_{q}}$. Thus necessarily $ \mathcal{M}_{x_{1},\ldots,x_{q}}>0$.  Since 
 \begin{equation*}
 F(\sigma) 
 =
 \mathcal{M}_{x_{1},\ldots,x_{q}}
 \begin{pmatrix}
 \frac{1}{\sigma_{1}}\\
 \vdots \\
 \frac{1}{\sigma_{q}}
 \end{pmatrix}
 \begin{pmatrix}
 \frac{1}{\sigma_{1}}\\
 \vdots \\
 \frac{1}{\sigma_{q}}
 \end{pmatrix}
 +
 2\tilde c_{1}\sum \frac{\sigma_{i}}{K_{i}}
 \end{equation*}
 is  a sum of convex functions,  there exists a unique critical point of $F$ satisfying 
 \eqref{cascade_show_existence_n=4}.
 Hence we have comparability 
$
\lambda_{1}\simeq \ldots\simeq \lambda_{1}\simeq \sfrac{1}{\sqrt{\tau}} \simeq \lambda
$
like in case $n\geq 5$. Thus (iv) follows upon checking constants  for $n=4$, i.e. 
$\bar c_{0}=\underset{\R^{n}}{\int}(\frac{1}{1+r^{2}})^{n}= \frac{\omega_{4}}{12}$ and 
\begin{equation}\label{def_tilde_c3_tildec_4}
\begin{matrix}[ll]
\mathbf{1}. & \tilde c_{1}
=
\frac{n(n-1)(n-2)^{2}}{\bar c_{0}^{\frac{n-2}{n}}}\underset{\R^{n}}{\int}\frac{1-r^{2}}{(1+r^{2})^{n+1}}\ln\frac{1}{1+r^{2}}dx=2\sqrt{3\omega_{4}} ;
\\
\mathbf{2.} &
 \tilde c_{2}
 =
 -\frac{(n-1)(n-2)}{\bar c_{0}^{\frac{n-2}{n}}}\underset{\R^{n}}{\int}\frac{r^{2}(1-r^{2})}{(1+r^{2})^{n+1}}dx=\sqrt{3\omega_{4}} ;
 \\
\mathbf{3.} &
  \tilde c_{3}
  =
  \tilde d_{1}
  =-\frac{4n(n-1)}{\bar c_{0}^{\frac{n-2}{n}}}\underset{\R^{n}}{\int}\frac{r^{n}(n+2-nr^{2})}{(1+r^{2})^{n+2}}
  =
  24\sqrt{3\omega_{4}} ;
  \\
  \mathbf{4.} &
  \tilde c_{4}
  =
  \frac{n-2}{2}\tilde b_{2}
  =
 \frac{2n(n-1)(n-2)}{\bar c_{0}^{\frac{n-2}{n}}}\underset{\R^{n}}{\int}\frac{1}{(1+r^{2})^{\frac{n+2}{2}}}
 =
  24\sqrt{3\omega_{4}},
\end{matrix}
\end{equation}
cf. \eqref{def_const_lambda_derivatives} from the corresponding Lemma 
\ref{lem_lambda_derivatives_at_infinity}.
We turn next to (iii). In case $n\geq 5$  we may now assume
\begin{equation*}
\tilde c_{1}\tau+\tilde c_{2}\frac{\Delta K_{k}}{K_{k}\lambda_{k}^{2}}
=
o(\frac{1}{\lambda^{2}})
\; \text{ and }\; 
\varepsilon_{k,l}=o(\frac{1}{\lambda^{2}})
\; \text{ for }\; \lambda_{k}\simeq \lambda_{l} \simeq \lambda,
\end{equation*}
 which by Lemma \ref{lem_alpha_derivatives_at_infinity}  implies
\begin{equation*}
\vert \partial J_{\tau}(u)\vert
\gtrsim
\left|
1-\frac{\alpha^{2}}{\alpha_{K,\tau}^{p+1}}
\frac{K_{j}}{\lambda_{j}^{\theta}}\alpha_{j}^{p-1}
\right|
+o(\frac{1}{\lambda^{2}})
.
\end{equation*}
Note that $\alpha_{j}^{p-1}=\Theta^{p-1} \cdot \frac{\lambda_{j}^{\theta}}{K_{j}}$ is modulo scaling the unique and non-degenerate  maximum of 
\begin{equation*}
\alpha=(\alpha_{1},\ldots,\alpha_{q})
\longrightarrow 
\frac{\alpha^{2}}{(\alpha_{K,\tau}^{p+1})^{\frac{2}{p+1}}}
=
\frac{\sum \alpha_{i}^{2}}{(\sum \frac{K_{i}}{\lambda_{i}^{\theta}}\alpha_{i}^{p+1})^{\frac{2}{p+1}}}. 
\end{equation*}
Now \eqref{cascade_what_is_to_show} follows, unless 
$
\alpha_{j}^{p-1}=\Theta^{p-1}\cdot \frac{\lambda_{j}^{\theta}}{K_{j}}+o(\frac{1}{\lambda^{2}})
$
and there holds
\begin{equation*}
\vert \partial J_{\tau}(u)\vert
\gtrsim
\left|
\alpha_{j}-\Theta \cdot \sqrt[p-1]{\frac{\lambda_{j}^{\theta}}{K_{j}}}
\right|
+o(\frac{1}{\lambda^{2}})
.
\end{equation*}
In case $n=4$ we may  rewrite Lemma \ref{lem_alpha_derivatives_at_infinity} up to some $o(\frac{1}{\lambda^{2}})$ with constant given below as 
\begin{equation}\label{to_cancel_out}
\resizebox{.9\hsize}{!}
{$
\begin{split}
\partial J_{\tau}(u)\phi_{1,j}
= 
\frac{\alpha_{j}}
{(\alpha_{K,\tau}^{\frac{2n}{n-2}})^{\frac{n-2}{n}}}
\bigg(
&
\grave c_{0}(
1
-
\frac{\alpha^{2}}{\alpha_{K,\tau}^{p+1}}
\frac{K_{j}}{\lambda_{j}^{\theta}}\alpha_{j}^{p-1}
) \\
& 
-
K_{j}
(
\grave c_{2}
\frac{\Delta K_{j}}{K_{j}^{2}\lambda_{j}^{2}}
+
\grave d_{1}
\frac{H_{j}}{K_{j}\lambda_{j}^{2 }}
+
\grave b_{1}
\sum_{j\neq i}\frac{G_{g_{0}}(a_{i},a_{j})}{\gamma_{n}\sqrt{K_{i}K_{j}}\lambda_{i}\lambda_{j}}
)
\\
& 
+
\frac{\alpha_{K}^{\frac{2n}{n-2}}}{(\alpha^{2})^{2}}
(
\grave c_{2}
\sum_{k}\frac{\Delta K_{k}}{K_{k}^{2}\lambda_{k}^{2}}
+
\grave d_{1}
\sum_{k}\frac{H_{k}}{K_{k}\lambda_{k}^{2 }} 
+
\grave b_{1} 
\sum_{k\neq l}
\frac{G_{g_{0}}(a_{k},a_{l})}{\gamma_{n}\sqrt{K_{k}K_{l}}\lambda_{k}\lambda_{l}}
\bigg)
\end{split}
$}
\end{equation}
using \eqref{cascade_alpha_constraint} and $\lambda_{j}^{\theta}\simeq (\frac{1}{\sqrt{\tau}})^{\frac{n-2}{2}\tau}=1+O(\frac{\ln \lambda}{\lambda^{2}})$. Moreover, up to an error $o(1)$ there holds 
\begin{equation*}
\frac{(\alpha^{2})^{2}}{\alpha_{K}^{\frac{2n}{n-2}}}
=
\frac{\alpha^{2}\sum_{i}\alpha_{i}^{2}}{\alpha_{K}^{\frac{2n}{n-2}}}
=
\frac{\alpha^{2}\sum_{i}\frac{\alpha_{K}^{\frac{2n}{n-2}}}{\alpha^{2}}\frac{1}{K_{i}}}{\alpha_{K}^{\frac{2n}{n-2}}}
=
\sum_{i}\frac{1}{K_{i}}, 
\end{equation*}
and due to \eqref{cascade_matrix_condition} 
\begin{equation*}
\tilde  c_{2}
\sum_{k}\frac{\Delta K_{k}}{K_{k}^{2}\lambda_{k}^{2}}
+
\tilde c_{3}
\sum_{k}\frac{H_{k}}{K_{k}\lambda_{k}^{2 }} 
+
\tilde c_{4}
\sum_{k\neq l}
\frac{G_{g_{0}}(a_{k},a_{l})}{\gamma_{n}\sqrt{K_{k}K_{l}}\lambda_{k}\lambda_{l}}
=
\mathcal{M}_{a_{1},\ldots,a_{q}}
\begin{pmatrix}
\frac{1}{\lambda_{1}}\\
\vdots \\
\frac{1}{\lambda_{q}}
\end{pmatrix}
\begin{pmatrix}
\frac{1}{\lambda_{1}}\\
\vdots \\
\frac{1}{\lambda_{q}}
\end{pmatrix}
=
 \tilde c_{1} \sum_{i}\frac{\tau}{K_{i}}
\end{equation*}
and 
\begin{equation*}
\tilde c_{2}
\frac{\Delta K_{j}}{K_{j}^{2}\lambda_{j}^{2}}
+
\tilde c_{3}
\frac{H_{j}}{K_{j}\lambda_{j}^{2 }}
+
\tilde c_{4}
\sum_{j\neq i}\frac{G_{g_{0}}(a_{i},a_{j})}{\gamma_{n}\sqrt{K_{i}K_{j}}\lambda_{i}\lambda_{j}}
=
\mathcal{M}_{a_{1},\ldots,a_{q}}
\begin{pmatrix}
\frac{1}{\lambda_{1}}
 \\ \vdots \\
\frac{1}{\lambda_{q}}
\end{pmatrix}
\frac{e_{j}}{\lambda_{j}}
=
 \tilde c_{1}\frac{\tau}{K_{j}}
\end{equation*}
up to some $o(\frac{1}{\lambda^{2}})$. 
We may therefore cancel out the interaction terms in \eqref{to_cancel_out} and obtain 
\begin{equation}\label{cascade_for_analogy_alpha_space}
\resizebox{.9\hsize}{!}
{$
\begin{split}
\partial J_{\tau}(u)\phi_{1,j}
= 
\frac{\alpha_{j}}
{(\alpha_{K,\tau}^{\frac{2n}{n-2}})^{\frac{n-2}{n}}}
\bigg(
&
\grave c_{0}(
1
-
\frac{\alpha^{2}}{\alpha_{K,\tau}^{p+1}}
\frac{K_{j}}{\lambda_{j}^{\theta}}\alpha_{j}^{p-1}
) 
-
K_{j}
(
(\grave c_{2}-\frac{\grave b_{1}}{\tilde c_{4}}\tilde c_{2})
\frac{\Delta K_{j}}{K_{j}^{2}\lambda_{j}^{2}}
+
(\grave d_{1}-\frac{\grave b_{1}}{\tilde c_{4}}\tilde c_{3})
\frac{H_{j}}{K_{j}\lambda_{j}^{2 }}
)
\\
& 
+
\frac{1}{\sum_{k}\frac{1}{K_{k}}}
(
(\grave c_{2}-\frac{\grave b_{1}}{\tilde c_{4}}\tilde c_{2})
\sum_{k}\frac{\Delta K_{k}}{K_{k}^{2}\lambda_{k}^{2}}
+
(\grave d_{1}-\frac{\grave b_{1}}{\tilde c_{4}}\tilde c_{3})
\sum_{k}\frac{H_{k}}{K_{k}\lambda_{k}^{2 }} 
\bigg).
\end{split}
$}
\end{equation}
Checking constants for $n=4$, i.e. with $\bar c_{0}=\int_{\R^{n}}\frac{dx}{(1+r^2)^n}=\frac{\omega_{4}}{12}$
\begin{equation*}
\begin{matrix}[ll]
\mathbf{1.}  &
\grave c_{0}=8n(n-1)(\int_{\R^{n}}\frac{dx}{(1+r^2)^n})^{\frac{2}{n}}=16\sqrt{3\omega_{4}}
,\quad 
 \grave c_{2}=\frac{{8n}(n-1)}{\bar c_{0}^{\frac{n-2}{n}}}\frac{1}{2n}\underset{\R^{n}}{\int}\frac{r^{2}}{(1+r^{2})^{n}}
  =  4\sqrt{3\omega_{4}};
\\
\mathbf{2.}  &
 \grave d_{1}=\frac{{8n}(n-1)}{\bar c_{0}^{\frac{n-2}{n}}}\underset{\R^{n}}{\int}\frac{r^{n}}{(1+r^{2})^{n+1}}
  =
  24\sqrt{3\omega_{4}}
 ,\quad
 \grave b_{1}=\frac{8n(n-1)(n+2)}{\bar c_{0}^{\frac{n-2}{n}}(n-2)}\underset{\R^{n}}{\int}\frac{1}{(1+r^{2})^{\frac{n+2}{2}}}
 =
 144\sqrt{3\omega_{4}},
\end{matrix}
\end{equation*}
cf. \eqref{constants_alpha_derivative} from the corresponding  Lemma \ref{lem_alpha_derivatives_at_infinity}, we then find
\begin{equation*}
\vert \partial J_{\tau}(u)\vert
\gtrsim 
\left|
1 
-
\frac{\alpha^{2}K_{j}\alpha_{j}^{p-1}}{\alpha_{K,\tau}^{p+1}\lambda_{j}^{\theta}}
+
\frac{1}{8}
\left(
\frac{\Delta K_{j}}{K_{j}\lambda_{j}^{2}}
-
60\frac{H_{j}}{\lambda_{j}^{2 }}
-
\frac
{
\sum_{k}
(\frac{\Delta K_{k}}{K_{k}^{2}\lambda_{k}^{2}}
-
60
\frac{H_{k}}{K_{k}\lambda_{k}^{2 }}
)
}
{
\sum_{k}\frac{1}{K_{k}}
}
\right)
\right|
+
o(\frac{1}{\lambda^{2}}).
\end{equation*}
Note that setting 
\begin{equation*}
E_{j}=
\frac{1}{8}
\left(
\frac{\Delta K_{j}}{K_{j}\lambda_{j}^{2}}
-
60
\frac{H_{j}}{\lambda_{j}^{2 }}
-
\frac
{
\sum_{k}
(\frac{\Delta K_{k}}{K_{k}^{2}\lambda_{k}^{2}}
-
60
\frac{H_{k}}{K_{k}\lambda_{k}^{2 }}
)
}
{
\sum_{k}\frac{1}{K_{k}}
}
\right), 
\end{equation*}
there holds $E_{j}=O(\frac{1}{\lambda^{2}})$,  
$\sum_{j}\frac{E_{j}} {K_{j}}=0$, and  $\alpha_{j}^{p-1}=\Theta^{p-1}\frac{\lambda_{j}^{\theta}}{K_{j}}(1+E_{j})$ is modulo scaling the unique and non-degenerate maximum of 
\begin{equation*}
\alpha=(\alpha_{1},\ldots,\alpha_{q})
\longrightarrow 
\frac{\alpha^{2}}{(\alpha_{\frac{K}{1+E},\tau}^{p+1})^{\frac{2}{p+1}}}
=
\frac{\sum \alpha_{i}}
{(\sum \frac{K_{i}}{\lambda_{i}^{\theta}(1+E_{i})}\alpha_{i}^{p+1})^{\frac{2}{p+1}}}, 
\end{equation*}
and satisfies
\begin{equation*}
\begin{split}
\frac{\alpha^{2}}{\alpha_{K,\tau}^{p+1}}\frac{K_{j}}{\lambda_{j}^{\theta}}\alpha_{j}^{p-1}
= &
\Theta^{p-1}\cdot \frac{\alpha^{2}}{\alpha_{K,\tau}^{p+1}}(1+E_{j})
=
\frac
{
\sum [\frac{\lambda_{i}^{\theta}}{K_{i}}(1+E_{i})]^{\frac{2}{p-1}}
}
{
\sum \frac{K_{i}}{\lambda_{i}^{\theta}}[\frac{\lambda_{i}^{\theta}}{K_{i}}(1+E_{i})]^{\frac{p+1}{p-1}}
}
(1+E_{j}) \\
= &
\frac
{
\sum (\frac{\lambda_{i}^{\theta}}{K_{i}})^{\frac{2}{p-1}}
+
\frac{2}{p-1}\sum (\frac{\lambda_{i}^{\theta}}{K_{i}})^{\frac{2}{p-1}}E_{i}
}
{
\sum (\frac{\lambda_{i}^{\theta}}{K_{i}})^{\frac{2}{p-1}}
+
\frac{p+1}{p-1}
\sum (\frac{\lambda_{i}^{\theta}}{K_{i}})^{\frac{2}{p-1}}E_{i}
}
(1+E_{j})
=
1+E_{j}
+
o(\frac{1}{\lambda^{3}})
\end{split}
\end{equation*}
due to $(\frac{\lambda_{i}^{\theta}}{K_{i}})^{\frac{2}{p-1}}=\frac{1}{K_{i}}+O(\frac{\ln \lambda}{\lambda^{2}})$.  Thus \eqref{cascade_what_is_to_show} follows unless, 
up to some $o(\frac{1}{\lambda^{2}})$, 
\begin{equation}\label{cascade_gradient_lower_bound_lambda_alpha_n=4}
\vert \partial J_{\tau}(u)\vert
\gtrsim 
\left|
\alpha_{j}
-
\Theta
\sqrt[p-1]
{
\frac{\lambda_{j}^{\theta}}{K_{j}}
\left(
1
+
\frac{1}{8}
(
\frac{\Delta K_{j}}{K_{j}\lambda_{j}^{2}}
-
60
\frac{H_{j}}{\lambda_{j}^{2 }}
-
\frac
{
\sum_{k}
(\frac{\Delta K_{k}}{K_{k}^{2}\lambda_{k}^{2}}
-
60
\frac{H_{k}}{K_{k}\lambda_{k}^{2 }}
)
}
{
\sum_{k}\frac{1}{K_{k}}
}
)
\right)
}
\right|
.
\end{equation}
We have therefore proved (i)-(iv), which will be used for showing 
the second statement of the proposition.  
In this case the error terms in Lemmata \ref{lem_alpha_derivatives_at_infinity}, \ref{lem_lambda_derivatives_at_infinity} and \ref{lem_a_derivatives_at_infinity} are of type
$
o(\lambda^{-3})+O(\vert \partial J_{\tau}(u)\vert^{2}).
$
This follows immediately in case $n\geq 5$, while the terms 
$\varepsilon_{r,s}^{\frac{n+2}{n}}\simeq \lambda^{-3}$ in case $n=4$, for which however the underlying estimates  can be improved to derive a quadratic error in $\varepsilon_{r,s}$, cf. \cite{may-cv}. 
Let us first treat the lower bounds arising from Lemma \ref{lem_a_derivatives_at_infinity}. In case $n\geq 5$ we  find from the latter lemma
\begin{equation*}
\begin{split}
\vert \partial J_{\tau}(u)\vert 
 \gtrsim &
 \vert 
 \check{c}_{3}\frac{\nabla K_{j}}{K_{j}\lambda_{j}}
  +
  \check{c}_{4} \frac{\nabla \Delta K_j}{K_j \lambda_{j}^3}
\vert
\gtrsim
 \vert
 \check{c}_{3}\frac{\nabla K(a_{j})}{\lambda_{j}}
  +
  \check{c}_{4} \frac{\nabla \Delta K(x_{j})}{\lambda_{j}^3}
\vert
\end{split}
\end{equation*}
up to some $o(\lambda^{-3})$
and therefore, writing $a_{j}=\exp_{g_{x_{j}}}(\bar a_{j})$, that 
\begin{equation*}
\begin{split}
\vert \partial J_{\tau}(u)\vert 
\gtrsim &
 \vert
\frac{ \bar a_{j}}{\lambda_{j}}
  +
\frac{  \check{c}_{4}}{\check{c}_{3}} (\nabla^{2}K(x_{j}))^{-1}\frac{\nabla \Delta K(x_{j})}{\lambda_{j}^3}
\vert
+
o(\frac{1}{\lambda^{3}}).
\end{split}
\end{equation*}
Similarly in case $n=4$ we find up to some $o(\lambda^{-3})$
\begin{equation*}
 \vert \partial J_{\tau}(u)\vert 
 \gtrsim 
\vert  \check{c}_{3}\frac{\nabla K_{j}}{K_{j}\lambda_{j}}
  +
  \check{c}_{4} \frac{\nabla \Delta K_j}{K_j \lambda_{j}^3}
  +
  \check{b}_{3}
  \sum_{j\neq i}
  \frac{\alpha_{i}}{\alpha_{j}}
  \frac{\nabla_{a_{j}}G_{g_{0}}(a_{i},a_{j})}{\gamma_{n}\lambda_{i}\lambda_{j}^{2}}
  \vert. 
\end{equation*}
From (iii) we  have $\alpha_{i}=\Theta(\frac{\lambda_{i}^{\theta}}{K_{i}})^{\frac{1}{p-1}}+O(\frac{1}{\lambda^{2}})$, which by $\theta=\frac{n-2}{2}\tau$ and $\lambda_{i}\simeq \tau^{-\frac{1}{2}}$ due to (iv) becomes 
$ 
\alpha_{i}=\frac{\Theta}{\sqrt{K_{i}}}+O(\frac{\ln \lambda}{\lambda^{2}}).
$ 
Thus, still up to some $o(\lambda^{-3})$
\begin{equation*}
 \begin{split}
\vert \partial J_{\tau}(u)\vert 
 \gtrsim &
 \vert \frac{\nabla K(a_{j})}{\lambda_{j}}
  +
  \frac{\check{c}_{4}}{ \check{c}_{3}} \frac{\nabla \Delta K(x_{j})}{\lambda_{j}^3}
  +
\frac{  \check{b}_{3}}{ \check{c}_{3}}
  \sum_{j\neq i}
\sqrt{  \frac{K^{3}(x_{j})}{K(x_{i})} }
  \frac{\nabla_{x_{j}}G_{g_{0}}(x_{i},x_{j})}{\gamma_{n}\lambda_{i}\lambda_{j}^{2}}
  \vert \\
\gtrsim &
\vert 
\frac{\bar a_{j}}{\lambda_{j}}
+
\frac{\check{c}_{4}}{ \check{c}_{3}}
(\nabla^{2}K(x_{j}))^{-1}
\frac{\nabla \Delta K(x_{j})}{\lambda_{j}^3}
+
\frac{  \check{b}_{3}}{ \check{c}_{3}}
  \sum_{j\neq i}
\sqrt{  \frac{K^{3}(x_{j})}{K(x_{i})} }
(\nabla^{2}K(x_{j}))^{-1}
  \frac{\nabla_{x_{j}}G_{g_{0}}(x_{i},x_{j})}{\gamma_{n}\lambda_{i}\lambda_{j}^{2}}
  \vert, 
\end{split}
\end{equation*}
and checking constants from Lemma \ref{lem_a_derivatives_at_infinity}, cf.  \eqref{def_a_constants},  we have 
\begin{equation*}
\check{c}_{3}=\int_{\R^{n}}\frac{4(n-1)(n-2)dx}{(1+r^{2})^n}=3\omega_{4}
,\;
\check{c}_{4}=\underset{\R^{n}}{\int}\frac{2(n-1)r^{2}dx}{(1+r^{2})^{n}}=\omega_{4}
,\; 
\check{b}_{3}=\underset{\R^{n}}{\int}\frac{8n(n-1)dx}{(1+r^{2})^{\frac{n+2}{2}}}
=
24\omega_{4}. 
\end{equation*}
We conclude that, up to some $o(\frac{1}{\lambda^{3}})$
\begin{equation}\label{lower_bound_a_space}
\resizebox{.9\hsize}{!}
{$
\vert \partial J_{\tau}(u) \vert
\gtrsim
\left(
\begin{matrix}
\vert \frac{\bar a_{j}}{\lambda_{j}}
+
\frac{1}{3}
(\nabla^{2}K(x_{j}))^{-1}
\frac{\nabla \Delta K(x_{j})}{\lambda_{j}^3}
+
8
\underset{{j\neq i}}{\sum}
\sqrt{  \frac{K^{3}(x_{j})}{K(x_{i})} }
(\nabla^{2}K(x_{j}))^{-1}
  \frac{\nabla_{x_{j}}G_{g_{0}}(x_{i},x_{j})}{\gamma_{n}\lambda_{i}\lambda_{j}^{2}}
  \vert
& \text{for }\; n=4 \\
 \vert
\frac{ \bar a_{j}}{\lambda_{j}}
  +
\frac{\check c_{4}}{\check c_{3}}
(\nabla^{2}K(x_{j}))^{-1}\frac{\nabla \Delta K(x_{j})}{\lambda_{j}^3}
\vert
& \text{for }\; n\geq 5
\end{matrix}
\right)
$}
\end{equation}
By this, i.e. $\bar a_{j}=O(\frac{1}{\lambda^{2}})$,  and $\alpha_{i}=\frac{\Theta}{\sqrt{K_{i}}}+O(\frac{\ln \lambda}{\lambda^{2}})$ we then infer  from Lemma \ref{lem_lambda_derivatives_at_infinity} that up to some $o(\frac{1}{\lambda^{3}})$ 
  \begin{equation*}
  \begin{split}
  \vert \partial J_{\tau}(u)\vert
  \gtrsim 
  \vert
  \tilde c_{1}\tau
  +
  \tilde c_{2}\frac{\Delta K(x_{j})}{K(x_{j})\lambda_{j}^{2}} 
  +
\frac{n-2}{2}  \tilde b_{2}\sum_{j\neq i }\sqrt{\frac{K(x_{j})}{K(x_{i})}} \frac{G_{g_{0}}(x_{i},x_{j})}{\gamma_{n}(\lambda_{i}\lambda_{j})^{\frac{n-2}{2}}}
 +
  \tilde  d_{1}\frac{H(x_{j})}{\lambda_{j}^{n-2}}
 \vert
  \end{split}
  \end{equation*}
  with constants, cf. above, given for $n=4,5$ respectively by 
\begin{equation*}
\begin{matrix}[ll]
\mathbf{1.} &
\frac{\tilde c_{2}}{\tilde c_{1}}
=
-
\frac
{
\int_{\R^{n}}\frac{r^{2}(1-r^{2})}{(1+r^{2})^{n+1}}dx
}
{n(n-2)\int_{\R^{n}}\frac{1-r^{2}}{(1+r^{2})^{n+1}}\ln(\frac{1}{1+r^{2}})dx}
=
\frac{1}{2},\frac{2}{9};
\\
\mathbf{2.} &
\frac{\tilde c_{3}}{\tilde c_{1}}
=
\frac{\tilde d_{1}}{\tilde c_{1}}
=
\frac{4}{(n-2)^{2}}
\frac
{\int_{\R^{n}}\frac{r^{n}(n+2)-nr^{2}}{(1+r^{2})^{n+2}}dx
}
{
\int_{\R^{n}}\frac{1-r^{2}}{(1+r^{2})^{n+1}}\ln(\frac{1}{1+r^{2}})dx
}
=
12,\frac{512}{9\pi};
\\
\mathbf{3.} &
\frac{n-2}{2}\frac{\tilde b_{2}}{\tilde c_{1}}
=
\frac{\tilde c_{4}}{\tilde c_{1}} =
\frac{2}{n-2}
\frac
{
\int_{\R^{n}}\frac{dx}{(1+r^{2})^{\frac{n+2}{2}}}
}
{
\int_{\R^{n}}\frac{1-r^{2}}{(1+r^{2})^{n+1}}\ln(\frac{1}{1+r^{2}})dx
}=
12,\frac{512}{9\pi}, 
\end{matrix}
\end{equation*}
we conclude
\begin{equation}\label{lower_bound_lambda_space}
\vert \partial J_{\tau}(u) \vert
\gtrsim 
\left(
\begin{matrix}
  \vert
\tau
+
\frac{1}{2}\frac{\Delta K(x_{j})}{K(x_{j})\lambda_{j}^{2}} 
+
12
[
\frac{H(x_{j})}{\lambda_{j}^{2}}
+
\sum_{j\neq i }\sqrt{\frac{K(x_{j})}{K(x_{i})}} \frac{G_{g_{0}}(x_{i},x_{j})}{\gamma_{n}\lambda_{i}\lambda_{j}}
]
 \vert
&\text{for }\; n=4 \\
  \vert
\tau
  +
\frac{2}{9}\frac{\Delta K(x_{j})}{K(x_{j})\lambda_{j}^{2}} 
+
\frac{512}{9\pi}
[
\frac{H(x_{j})}{\lambda_{j}^{3}}
+
\sum_{j\neq i }\sqrt{\frac{K(x_{j})}{K(x_{i})}} \frac{G_{g_{0}}(x_{i},x_{j})}{\gamma_{n}(\lambda_{i}\lambda_{j})^{\frac{3}{2}}}
] 
 \vert
&\text{for }\; n=5 \\
  \vert
\tau
  +
  \frac{\tilde c_{2}}{\tilde c_{1}}\frac{\Delta K_{j}}{K_{j}\lambda_{j}^{2}} 
 \vert&\text{for }\; n \geq 6
\end{matrix}
\right).
\end{equation}
By similar reasoning, using 
$
\bar a_{j}=O(\frac{1}{\lambda^{2}})
\; \text{ and }\;
\alpha_{i}=\frac{\Theta}{\sqrt{K_{i}}}+O(\frac{\ln \lambda}{\lambda^{2}})
$
we finally have, up to some $o(\frac{1}{\lambda^{3}})$
\begin{equation*}
\vert \partial J_{\tau}(u)\vert
\gtrsim 
\left(
\begin{matrix}
\vert 
1 
-
\frac{\alpha^{2}K_{j}\alpha_{j}^{p-1}}{\alpha_{K,\tau}^{p+1}\lambda_{j}^{\theta}}
+
\frac{1}{8}
(
\frac{\Delta K_{j}}{K_{j}\lambda_{j}^{2}}
-
60\frac{H_{j}}{\lambda_{j}^{2 }}
-
\frac
{
\sum_{k}
(\frac{\Delta K_{k}}{K_{k}^{2}\lambda_{k}^{2}}
-
60
\frac{H_{k}}{K_{k}\lambda_{k}^{2 }}
)
}
{
\sum_{k}\frac{1}{K_{k}}
}
)
\vert 
& \text{ for } \; n=4
\\
\vert 
1 
-
\frac{\alpha^{2}K_{j}\alpha_{j}^{p-1}}{\alpha_{K,\tau}^{p+1}\lambda_{j}^{\theta}}
-
\frac{1}{90}
(
\frac{\Delta K_{j}}{K_{j}\lambda_{j}^{2}}
+
\frac{2816}{\pi}\frac{H_{j}}{\lambda_{j}^{3 }}
-
\frac
{
\sum_{k}
(\frac{\Delta K_{k}}{K_{k}^{2}\lambda_{k}^{2}}
+
\frac{2816}{\pi}
\frac{H_{k}}{K_{k}\lambda_{k}^{3 }}
)
}
{
\sum_{k}\frac{1}{K_{k}}
}
)
\vert 
& \text{ for } \; n=5
\\
\vert 
1 
-
\frac{\alpha^{2}K_{j}\alpha_{j}^{p-1}}{\alpha_{K,\tau}^{p+1}\lambda_{j}^{\theta}}
\vert 
& \text{ for } \; n \geq 6
\\
\end{matrix}
\right).
\end{equation*}
This follows in case $n\geq 6$ immediately from Lemma \ref{lem_alpha_derivatives_at_infinity}
and for $n=4$ by repeating the arguments leading to \eqref{to_cancel_out} and  \eqref{cascade_for_analogy_alpha_space}, while the case $n=5$  follows by arguing as in case $n=4$ using 
\eqref{lower_bound_lambda_space} to cancel out the interaction terms  when passing from \eqref{to_cancel_out} to \eqref{cascade_for_analogy_alpha_space}. 
Then arguing as for the passage from \eqref{cascade_for_analogy_alpha_space} to \eqref{cascade_gradient_lower_bound_lambda_alpha_n=4} we finally obtain that up to some 
$o(\frac{1}{\lambda^{3}})$
\begin{equation}\label{lower_bound_alpha_space}
\vert \partial J_{\tau}(u)\vert
\gtrsim 
\left(
\begin{matrix}
\vert 
\alpha_{j}
-
\Theta
\sqrt[p-1]{
\frac{\lambda_{j}^{\theta}}{K_{j}}(
1
+
\frac{1}{8}
(
\frac{\Delta K_{j}}{K_{j}\lambda_{j}^{2}}
-
60\frac{H_{j}}{\lambda_{j}^{2 }}
-
\frac
{
\sum_{k}
(\frac{\Delta K_{k}}{K_{k}^{2}\lambda_{k}^{2}}
-
60
\frac{H_{k}}{K_{k}\lambda_{k}^{2 }}
)
}
{
\sum_{k}\frac{1}{K_{k}}
}
)
)
}
\vert 
& \text{ for } \; n=4
\\
\vert 
\alpha_{j}
-
\Theta
\sqrt[p-1]{
\frac{\lambda_{j}^{\theta}}{K_{j}}(
1
-
\frac{1}{90}
(
\frac{\Delta K_{j}}{K_{j}\lambda_{j}^{2}}
+
\frac{2816}{\pi}\frac{H_{j}}{\lambda_{j}^{3 }}
-
\frac
{
\sum_{k}
(\frac{\Delta K_{k}}{K_{k}^{2}\lambda_{k}^{2}}
+
\frac{2816}{\pi}
\frac{H_{k}}{K_{k}\lambda_{k}^{3 }}
)
}
{
\sum_{k}\frac{1}{K_{k}}
}
)
)
}
\vert 
& \text{ for } \; n=5
\\
\vert 
\alpha_{j}
-
\Theta
\sqrt[p-1]{
\frac{\lambda_{j}^{\theta}}{K_{j}}
}
\vert 
& \text{ for } \; n \geq 6
\\
\end{matrix}
\right).
\end{equation}
Thus the second statement of the theorem follows from combining
\eqref{lower_bound_a_space}, \eqref{lower_bound_lambda_space} and \eqref{lower_bound_alpha_space}.
\end{proof}
 
\noindent In \cite{MM2} the next result will be needed. 
 
 \begin{lemma} 
 \label{lem_upper_bound} 
  For every $u\in V(q,\varepsilon)$ there holds
  \begin{equation*}
 \vert \partial J_{\tau}(u)\vert \
 \lesssim 
 \tau
 +
 \sum_{r\neq s}\frac{\vert \nabla K_{r}\vert}{\lambda_{r}}
 +
 \frac{1}{\lambda_{r}^{2}}
 +
 \frac{1}{\lambda_{r}^{n-2}}
 +
 \vert 1
 -
 \frac{\alpha^{2}}{\alpha_{K,\tau}^{p+1}}
 \frac{K_{r}}{\lambda_{r}^{\theta}}\alpha_{r}^{p-1} \vert
 +
 \varepsilon_{r,s}^{\frac{n+2}{2n}}
 +
 \Vert v \Vert. 
  \end{equation*}
 \end{lemma}
 \begin{proof}
 Recalling \eqref{eq:Hu} we can find $\vert \beta_{k.i}\vert,\vert \beta\vert =O(1)$
 and
 $ 
 \nu \in H_{u}(p,\varepsilon),\| \nu \| =1
 $ 
 such that 
 \begin{equation*}
 \vert \partial J_{\tau}(u)\vert
 \lesssim
 \vert \beta^{k,i}\vert \vert \partial J_{\tau}(u)\phi_{k,i}\vert 
 +
 \vert \beta \vert \vert \partial J_{\tau}(u) \nu \vert
 \lesssim
 \sum_{k,i}\vert \partial J_{\tau}(u)\phi_{k,i}\vert + \vert \partial J_{\tau}(u) \nu \vert.
 \end{equation*}
 From Lemmata \ref{lem_alpha_derivatives_at_infinity}, \ref{lem_lambda_derivatives_at_infinity} 
 and \ref{lem_a_derivatives_at_infinity} we then find 
 \begin{equation*}
 \sum_{k,i}\vert \partial J_{\tau}(u)\phi_{k,i}\vert
 \lesssim 
 \tau +\sum_{j=1}^{q}\frac{\vert \nabla K_{j}\vert}{\lambda_{j}}+\frac{1}{\lambda_{j}^{2}}
 +
 \frac{1}{\lambda_{j}^{n-2}} 
  +
 \vert 1
 -
 \frac{\alpha^{2}}{\alpha_{K,\tau}^{p+1}}
 \frac{K_{j}}{\lambda_{j}^{\theta}}\alpha_{j}^{p-1} \vert
 +
 \sum_{r\neq s}\varepsilon_{r,s}
 +
 \vert \partial J_{\tau}(u)\vert^{2},
 \end{equation*}
 whereas from Lemma \ref{lem_testing_with_v} we have
 \begin{equation*}
  \partial J_{\tau}(u)\nu 
 = 
  \partial J_{\tau}(\alpha^{i}\varphi_{i}) \nu +O(\Vert v \Vert) 
 = 
 O
 (
 \tau
 +
 \sum_{r}\frac{\vert \nabla K_{r}\vert}{\lambda_{r}}
 +
 \frac{1}{\lambda_{r}^{2}}
 +
 \frac{1}{\lambda_{r}^{n-2}}
 +
 \sum_{r\neq s}\varepsilon_{r,s}^{\frac{n+2}{2n}}
 +
 \Vert v \Vert).
 \end{equation*} 
 From this the claim follows.
 \end{proof}

 \section{Appendix} \subsection{Interactions}
 \begin{proof}[Proof of Lemma \ref{lem_interactions}] 
\begin{enumerate} [label=(\roman*)]
 \item follows using straightforwardly the expression of $\phi_{k,i}$.  
\item
\begin{enumerate}
 \item[$(\alpha)$] \quad\quad  {\bf Case $k=1$.} 
We have $\phi_{k,i}=\varphi_{i}$ for $k=1$, and thus for $c>0$ small
\begin{equation*}\begin{split}
\int \varphi_{i}^{\frac{2n}{n-2}-\tau}d\mu_{g_{0}}
= &
\int_{B_{c}(a_{i})}u_{a_{i}}^{-\tau}\bigg(\frac{\lambda_{i}}{1+\lambda_{i}^{2}\gamma_{n}G^{\frac{2}{2-n}}_{a_{i}}}\bigg)^{n-\theta}d\mu_{g_{a_{i}}}
+
O\bigg(\frac{1}{\lambda_{i}^{n-\theta}}\bigg).
\end{split}\end{equation*}
\quad \quad \;\;
On $B_{c}(a_{i})$ one has  $u_{a_{i}}^{-\tau}=1+O(\tau \vert x -a_{i}\vert^{2})$,  and by  \eqref{eq:bubbles} 
\begin{equation*}
\gamma_{n}G_{a_{i}}^{\frac{2}{2-n}}
=
r^{2}+
O
\begin{pmatrix}
r^{3} & \text{ for } n= 3\\
r^{4} & \text{ for } n= 4\\
r^{5} & \text{ for } n= 5\\
r^6 \ln r  & \text{ for } n=6 \\
r^{6} & \text{ for } n\geq 7
\end{pmatrix},
\end{equation*}
\quad \quad \;\;
whence passing to normal coordinates at $a_i$ 
\begin{equation*}\begin{split}
\int \varphi_{i}^{\frac{2n}{n-2}- \tau}d\mu_{g_{0}}
= &
\underset{B_{c\lambda_{i}}(0)}{\int}
\frac{\lambda_{i}^{-\theta}dx}{(1+r^{2})^{n-\theta}} +
O
\begin{pmatrix}
\frac{1}{\lambda^{1+\theta}} & \text{for } n= 3\\
\frac{1}{\lambda^{2+\theta}} & \text{for } n= 4\\
\frac{1}{\lambda^{3+\theta}} & \text{for } n= 5\\
\frac{\ln \lambda}{ \lambda ^{4+\theta}} & \text{for } n=6 \\
\frac{1}{\lambda^{4+\theta}} & \text{for } n\geq 7
\end{pmatrix}
\end{split}\end{equation*}
\quad \quad \;\;
up to some error 
$O(\frac{\tau}{\lambda_{i}^{2+\theta}})$,
whence the claim follows with $c_{1}=\underset{\R^{n}}{\int} \frac{dx}{(1+r^{2})^{n}}$. 
 \item[$(\beta)$] \quad\quad  {\bf Case $k=2$.} 
The proof works analogously to the one of case $k=1$ above.
 \item[$(\gamma)$] \quad\quad  {\bf Case $k=3$.} 
We have
$
\phi_{k,i}
= 
\frac{2-n}{2} u_{a_{i}}
\frac
{\lambda_{i}\gamma_{n}\nabla_{a_{i}}G^{\frac{2}{2-n}}_{a_{i}}}
{1+\lambda_{i}^{2}\gamma_{n}G^{\frac{2}{2-n}}_{a_{i}}}\varphi_{i}
+
\frac{\nabla_{a_{i}}u_{a_{i}}}{\lambda_{{i}}}\varphi_{i},
$ 
 whence 
\begin{equation*}
\gamma_{n}(\nabla_{a_{i}}G^{\frac{2}{2-n}}_{a_{i}})(x)
=
-2x
+
O
(
r^{2},
r^{3},
r^{4},
r^5 \ln r,
r^{5}
)
\; \text{ for }\; 
n=3,\ldots,6 \; \text{ and }\; n\geq 7.
\end{equation*}
\quad \quad \;\;Moreover $u_{a_{i}}=1+O(r_{a_{i}}^{2})$, implies $\nabla_{a_{i}}u_{a_{i}}=O(r_{a_{i}})$. Thus 
\begin{equation*}\begin{split}
& \int \varphi_{i}^{\frac{4}{n-2}-\tau}\vert \phi_{k,i}\vert^{2}d\mu_{g_{0}}
= 
\frac{(n-2)^{2}}{n}\underset{\R^{n}}{\int} \frac{\lambda_{i}^{-\theta }r^{2}dx}{(1+r^{2})^{n+2-\theta}} 
+
O\big(\frac{1}{\lambda_{i}^{2+\theta}}\big)
+
O
\begin{pmatrix}
\lambda^{-1-\theta} & \text{for } n=3\\
\lambda^{-2-\theta} & \text{for } n=4\\
\lambda^{-3-\theta} & \text{for } n=5\\
\frac{\ln \lambda}{\lambda^{4+\theta}}  & \text{for } n=6 \\
\lambda^{-4-\theta} & \text{for } n\geq 7
\end{pmatrix}.
\end{split}\end{equation*}
\quad \quad \;\;
From this the claim follows. 
\end{enumerate}
 \item  
We just prove the case $k=2$ and start showing that 
\begin{equation}\begin{split}\label{non_linear_symmetry}
-\lambda_{i}^{\theta}\lambda_{j}\int \varphi_{i}^{\frac{n+2}{n-2}-\tau}\partial_{ \lambda_{j}}\varphi_{j}d\mu_{g_{0}}
= &
-\lambda_{i}^{\theta}\lambda_{j}\int \varphi_{i}^{1-\tau}\partial_{\lambda_{j}} \varphi_{j}^{\frac{n+2}{n-2}}d\mu_{g_{0}}
\end{split}\end{equation}
up to some 
$
O\big(\tau^{2} + \sum_{i\neq j}\big(\frac{1}{\lambda_{i}^{4}}+\frac{1}{\lambda_{i}^{2(n-2)}}+\varepsilon_{i,j}^{\frac{n+2}{n}}\big)\big), 
$ 
  so we may evaluate either of these integrals. Clearly
\begin{equation*}
\begin{split}
-\lambda_{i}^{\theta}\lambda_{j}\int \varphi_{i}^{\frac{n+2}{n-2}-\tau}\partial_{ \lambda_{j}}\varphi_{j}d\mu_{g_{0}}
= &
-\lambda_{i}^{\theta}\lambda_{j}\int_{B_{c}(a_{i})} \varphi_{i}^{\frac{n+2}{n-2}-\tau}\partial_{ \lambda_{j}}\varphi_{j}d\mu_{g_{0}}
\end{split}
\end{equation*}
up to an error $O(  \frac{1}{\lambda_{i}^{\frac{n+2}{2}}}  \frac{1}{\lambda_{j}^{\frac{n-2}{2}}})$,
whence using Lemma \ref{lem_emergence_of_the_regular_part} we find
\begin{equation*}
\begin{split}
-\lambda_{i}^{\theta}\lambda_{j}\int \varphi_{i}^{\frac{n+2}{n-2}-\tau}\partial_{ \lambda_{j}}\varphi_{j}d\mu_{g_{0}}
= &
-\lambda_{i}^{\theta}\lambda_{j}\int_{B_{c}(a_{i})} \varphi_{i}^{-\tau}\partial_{ \lambda_{j}}\varphi_{j} \frac{L_{g_{0}}\varphi_{i}}{4n(n-1)}d\mu_{g_{0}}
\end{split}
\end{equation*}
up to  $O(\lambda_{i}^{-2}+\lambda_{i}^{-(n-2)})\varepsilon_{i,j}^{\frac{n+2}{2n}})$. Indeed we clearly have 
$ 
\lambda_{i}^{-\frac{n+2}{2}}\lambda_{j}^{-\frac{n-2}{2}}=O\big(\lambda_{i}^{-1}\varepsilon_{i,j}^{\frac{n+2}{2n}}\big), 
$ 
and the difference from $L_{g_{0}}\varphi_{i}$ to $4n(n-1)\varphi_{i}^{\frac{n+2}{n-2}}$ can be estimated by Lemma \ref{lem_emergence_of_the_regular_part} via 
quantities of the type 
\begin{equation*}
\underset{B_{c}(a_{i})}{\int} r_{a_{i}}^{\alpha}\varphi_{i}^{\beta}\varphi_{j}d\mu_{g_{0}}
= 
\underset{B_{c}(a_{i})}{\int}  r_{a_{i}}^{\alpha}\varphi_{i}^{\beta-\frac{n+2}{2n}}\varphi_{i}^{\frac{n+2}{2n}}\varphi_{j}d\mu_{g_{0}} 
= 
O
\big(
\varepsilon_{i,j}^{\frac{n+2}{2n}}\Vert r^{\alpha}\varphi_{0,\lambda}^{\beta-\frac{n+2}{2n}}\Vert_{L^{(\frac{2n}{n+2})^{2}}}
\big), 
\end{equation*}
thanks to case (v). Passing back to integrating on the whole manifold $M$ we find , estimating also 
mixed products of gradients of $\varphi_{i}$ and $\varphi_{j}$,
\begin{equation*}
\begin{split}
-\lambda_{i}^{\theta}\lambda_{j}\int \varphi_{i} & ^{\frac{n+2}{n-2}-\tau}\partial_{ \lambda_{j}}\varphi_{j}d\mu_{g_{0}} 
= 
-(1+O(\tau))\lambda_{i}^{\theta}\lambda_{j}\int_{B_{c}(a_{i})} \varphi_{i}^{1-\tau}\partial_{ \lambda_{j}}\frac{L_{g_{0}}\varphi_{j}}{4n(n-1)}d\mu_{g_{0}} \\
& +
O
\big(\lambda_{i}^{\theta}\int \varphi_{i}\Delta_{g_{0}}\varphi_{i}^{-\tau}\varphi_{j}d\mu_{g_{0}}\big)
+
O\big(\big(\frac{1}{\lambda_{i}^{2}}+\frac{1}{\lambda_{i}^{n-2}}\big)\varepsilon_{i,j}^{\frac{n+2}{2n}}\big).
\end{split}
\end{equation*}
By direct calculation $\Delta_{g_{0}}\varphi_{i}^{-\tau}=O(\tau \varphi_{i}^{\frac{4}{n-2}-\tau})$, whence
\begin{equation*}
\begin{split}
-\lambda_{i}^{\theta}\lambda_{j}\int \varphi_{i}^{\frac{n+2}{n-2}-\tau}\partial_{ \lambda_{j}}\varphi_{j}d\mu_{g_{0}}
= &
-\lambda_{i}^{\theta}\lambda_{j}\int\varphi_{i}^{1-\tau}\partial_{ \lambda_{j}}\frac{L_{g_{0} \varphi_j}}{4n(n-1)}  d\mu_{g_{0}}
+
O((\tau + \frac{1}{\lambda_{i}^{2}}+\frac{1}{\lambda_{i}^{n-2}})\varepsilon_{i,j}^{\frac{n+2}{2n}}).
\end{split}
\end{equation*}
Now applying Lemma \ref{lem_emergence_of_the_regular_part} as before, but
in differentiated form, \eqref{non_linear_symmetry}  follows. Let  
\begin{equation*}
R_{i,j}
=
O\big(\tau^{2}+\sum_{i\neq j}\big(\frac{1}{\lambda_{i}^{4}}+\frac{1}{\lambda_{i}^{2(n-2)}}+\varepsilon_{i,j}^{\frac{n+2}{n}}\big)\big)
\end{equation*}
denote a quantity such order. We now assume the non-exclusive alternative
\begin{equation}\label{lambda_d_lambda_interaction_discernation_1}
\begin{split}
\varepsilon_{i,j}^{\frac{2}{2-n}}\sim \frac{\lambda_{i}}{\lambda_{j}}
\quad \vee \quad 
\varepsilon_{i,j}^{\frac{2}{2-n}}\sim \lambda_{i}\lambda_{j}d^{2}(a_{i},a_{j})\gg \frac{\lambda_{j}}{\lambda_{i}}.
\end{split}
\end{equation} 
For $c>0$ small and fixed we have by the expression in \eqref{eq:bubbles}
\begin{equation*}\begin{split}
&-\lambda_{i}^{\theta}\lambda_{j}\int  \varphi_{i}^{\frac{n+2}{n-2}-\tau}\partial_{\lambda_{j}}\varphi_{j}d\mu_{g_{0}}
\\ & = 
\frac{n-2}{2}\lambda_{i}^{\theta}
\underset{B_{c}( a _{i})}{\int}
\big(\frac{ \lambda_{i} }{ 1+\lambda_{i}^{2}\gamma_{n}G_{ a _{i}}^{\frac{2}{2-n}}}\big)^{\frac{n+2}{2}-\theta}\frac{u_{ a _{j}}}{u_{ a _{i}}^{1+\tau}} 
\big(\frac{ \lambda_{j} }{ 1+\lambda_{j}^{2}\gamma_{n}G_{ a _{j}}^{\frac{2}{2-n}}}\big)^{\frac{n-2}{2}}
\frac{\lambda_{j}^{2}\gamma_{n}G_{a_{j}}^{\frac{2}{2-n}}-1}{\lambda_{j}^{2}\gamma_{n}G_{a_{j}}^{\frac{2}{2-n}}+1}
d\mu_{g_{ a _{i}}}
+
R_{i,j},
\end{split}\end{equation*}
whence passing to $g_{a_{i}}$-normal coordinates and recalling  \eqref{eq:bubbles} we find
\begin{equation}\begin{split}\label{expanding_lambdaipartiallambdaiepsij}
-\lambda_{i}^{\theta}\lambda_{j}\int \varphi_{i}^{\frac{n+2}{n-2}-\tau}\partial_{\lambda_{j}}\varphi_{j}d\mu_{g_{0}}
=  &
\frac{n-2}{2}
\underset{B_{ c\lambda_{i} }(0)}{\int}
\frac{u_{ a _{j}}( a _{i})}{(1+r^{2})^{\frac{n+2}{2}-\theta}} 
\frac{\lambda_{j}^{2}\gamma_{n}G_{a_{j}}^{\frac{2}{2-n}}(\exp_{g_{ a _{i}}}\frac{x}{\lambda_{i}})-1}
{\lambda_{j}^{2}\gamma_{n}G_{a_{j}}^{\frac{2}{2-n}}(\exp_{g_{ a _{i}}} \frac{x}{\lambda_{i}})+1}
\\
& \quad\quad\quad\quad\quad\;\,
\bigg(\frac{1}{\frac{\lambda_{i} }{ \lambda_{j} }+\lambda_{i}\lambda_{j}\gamma_{n}G_{ a _{j}}^{\frac{2}{2-n}}
(\exp_{g_{ a _{i}}} \frac{x}{\lambda_{i}})}\bigg)^{\frac{n-2}{2}} d\mu_{g_{a_{i}}}
\end{split}\end{equation}
up to the error 
$R_{i,j}$. 
Indeed for e.g. $n\geq 7$ \eqref{eq:bubbles} tells us that on $B_{c}(0)$
\begin{equation*}
\big(\frac{ \lambda_{i} }{ 1+\lambda_{i}^{2}\gamma_{n}G_{ a _{i}}^{\frac{2}{2-n}}}\big)^{\frac{n+2}{2}-\theta}
= 
\big(\frac{\lambda_{i}}{1+\lambda_{i}^{2}r^{2}}\big)^{\frac{n+2}{2}-\theta}\big(1+O\big(\frac{\lambda^{2}r^{4}}{1+\lambda^{2}r^{2}}\big)\big) 
= 
\big(\frac{\lambda_{i}}{1+\lambda_{i}^{2}r^{2}}\big)^{\frac{n+2}{2}-\theta}(1+r^{2}) 
\end{equation*}
in conformal normal coordinates, whence by H\"older's inequality and Lemma \ref{lem_interactions}
\begin{equation*}
\int_{B_{c}(a_{i})} r_{a_{i}}^{2}\varphi_{i}^{\frac{n+2}{n-2}-\tau}\varphi_{j}d\mu_{g_{a_{i}}}
\leq 
\Vert  r^{2}\varphi_{0,\lambda}^{\frac{n+2}{n-2}-\frac{n+2}{2n}-\tau}\Vert_{L^{(\frac{2n}{n+2})^{2}}}\varepsilon_{i,j}^{\frac{n+2}{2n}}
=
O\big(\frac{\varepsilon_{i,j}^{\frac{n+2}{2n}}}{\lambda_{i}^{2+\theta}}\big).
\end{equation*}
Due to \eqref{lambda_d_lambda_interaction_discernation_1} we have that either 
\begin{equation*}\begin{split}
\varepsilon_{i,j}^{\frac{2}{2-n}}
\sim 
\lambda_{i}\lambda_{j}\gamma_{n}G^{\frac{2}{2-n}}( a _{i}, a _{j})
 \quad \text{ or } \quad
\varepsilon_{i,j}^{\frac{2}{2-n}}
\sim
\frac{\lambda_{i}}{\lambda_{j}}, 
\end{split}\end{equation*}
and for $\epsilon > 0$ sufficiently small may expand on 
\begin{equation*}\begin{split}
\mathcal{A}
= &
\left\{ \big\vert  \frac{x}{\lambda_{i}}\big\vert \leq \epsilon\sqrt{\gamma_{n}G^{\frac{2}{2-n}}_{ a _{j}}( a _{i})}\right\} 
\cup
\left\{ \big\vert  \frac{x}{\lambda_{i}}\big\vert \leq \epsilon  \frac{1}{\lambda_{j}} \right\} 
\subset 
B_{c\lambda_{i}}(0)
\end{split}\end{equation*} 
the integrand in \eqref{expanding_lambdaipartiallambdaiepsij} as
\begin{equation*}\begin{split}
& 
\frac{1}
{
(\frac{\lambda_{i} }{ \lambda_{j} }
+
\lambda_{i}\lambda_{j}\gamma_{n}G_{ a _{j}}^{\frac{2}{2-n}}(\exp_{g_{ a _{i}}} \frac{x}{\lambda_{i}}))^{\frac{n-2}{2}} }
\frac{\lambda_{j}^{2}\gamma_{n}G_{a_{j}}^{\frac{2}{2-n}}(\exp_{g_{ a _{i}}} \frac{x}{\lambda_{i}})-1}
{\lambda_{j}^{2}\gamma_{n}G_{a_{j}}^{\frac{2}{2-n}}(\exp_{g_{ a _{i}}} \frac{x}{\lambda_{i}})+1}
\\
& \quad = 
(
\frac{\lambda_{i} }{ \lambda_{j} }
+
\lambda_{i}\lambda_{j}
\gamma_{n}G_{ a _{j}}^{\frac{2}{2-n}}( a _{i})
)^{\frac{2-n}{2}} 
\frac
{\lambda_{j}^{2}\gamma_{n}G_{a_{j}}^{\frac{2}{2-n}}(a_{i})-1}
{\lambda_{j}^{2}\gamma_{n}G_{a_{j}}^{\frac{2}{2-n}}(a_{i})+1}\\
& \quad \; +
\frac{2-n}{2}
\frac
{
\gamma_{n}\nabla G^{\frac{2}{2-n}}_{ a _{j}}( a _{i})\lambda_{j}x
}
{
(
\frac{\lambda_{i} }{ \lambda_{j} }
+
\lambda_{i}\lambda_{j}\gamma_{n}G_{ a _{j}}^{\frac{2}{2-n}}( a _{i})
)^{\frac{n}{2}}
}
\frac
{\lambda_{j}^{2}\gamma_{n}G_{a_{j}}^{\frac{2}{2-n}}(a_{i})-1}
{\lambda_{j}^{2}\gamma_{n}G_{a_{j}}^{\frac{2}{2-n}}(a_{i})+1}
\\
& \quad \;+
\frac
{
2
}
{
(
\frac{\lambda_{i} }{ \lambda_{j} }
+
\lambda_{i}\lambda_{j}\gamma_{n}G_{ a _{j}}^{\frac{2}{2-n}}( a _{i})
)^{\frac{n}{2}}
}
\frac
{\gamma_{n}\nabla G_{a_{j}}^{\frac{2}{2-n}}(a_{i})\lambda_{j}x}
{1+\lambda_{j}^{2}\gamma_{n}G_{a_{j}}^{\frac{2}{2-n}}(a_{i})}
 +
\frac
{
O(\frac{ \lambda_{j} }{ \lambda_{i} }\vert x\vert^{2})
}
{
(
\frac{\lambda_{i} }{ \lambda_{j} }
+
\lambda_{i}\lambda_{j}\gamma_{n}G_{ a _{j}}^{\frac{2}{2-n}}( a _{i})
)^{\frac{n}{2}}
}.
\end{split}\end{equation*}
Using radial symmetry we then get, with $\tilde b_{2}=\frac{n-2}{2}\underset{\R^{n}}{\int} \frac{dx}{(1+r^{2})^{\frac{n+2}{2}}}=\frac{n-2}{2}b_{1}$, 
\begin{equation*}\begin{split}
-\lambda_{i}^{\theta}\lambda_{j}\int  \varphi_{i}^{\frac{n+2}{n-2}}\varphi_{j} d\mu_{g_{0}}
= &
\frac
{\tilde b_{2}u_{ a _{j}}( a _{i})}
{(\frac{\lambda_{i} }{ \lambda_{j} }
+
\lambda_{i}\lambda_{j}
\gamma_{n}G_{ a _{j}}^{\frac{2}{2-n}}( a _{i})
)^{\frac{n-2}{2}}
}
\frac
{\lambda_{j}^{2}\gamma_{n}G_{a_{j}}^{\frac{2}{2-n}}(a_{i})-1}
{\lambda_{j}^{2}\gamma_{n}G_{a_{j}}^{\frac{2}{2-n}}(a_{i})+1} 
\end{split}\end{equation*}
up to errors of the form  $R_{i,j}$ and $I_{\mathcal{A}^{c}}$, where
\begin{equation*}
\begin{split}
I_{\mathcal{A}^{c}}
\simeq 
\underset{\mathcal{A}^{c}}{\int}
\frac{1}{(1+r^{2})^{\frac{n+2}{2}-\theta}} 
\big(\frac{1}{\frac{\lambda_{i} }{ \lambda_{j} }+\lambda_{i}\lambda_{j}\gamma_{n}G_{ a _{j}}^{\frac{2}{2-n}}
(\exp_{g_{ a _{i}}} \frac{x}{\lambda_{i}})}\big)^{\frac{n-2}{2}}d\mu_{g_{a_{i}}}.
\end{split}
\end{equation*}
In case $\varepsilon_{i,j}^{\frac{2}{2-n}}\sim \frac{\lambda_{i}}{\lambda_{j}}$, we obviously have
\begin{equation*}
\begin{split}
I_{\mathcal{A}^{c}}
\leq 
C\big(\frac{\lambda_{j}}{\lambda_{i}}\big)^{\frac{n+2}{2}-2\theta}
=
o(\varepsilon_{i,j}^{\frac{n+2}{n}}).
\end{split}
\end{equation*} 
Otherwise we may assume $\mathcal{A}^{c}\neq \emptyset$, thus $d(a_{i},a_{j})\ll 1$, and write 
$
\mathcal{A}^{c}
\subseteq 
\, \mathcal{B}_{1}\cup \mathcal{B}_{2},
$
where
\begin{equation*}
\begin{split}
\mathcal{B}_{1}
= &
\left\{ \epsilon \sqrt{\gamma_{n}G^{\frac{2}{2-n}}_{a_{j}}(a_{i})}
\leq 
\big\vert \frac{x}{\lambda_{i}}\big\vert 
\leq 
E\sqrt{\gamma_{n}G^{\frac{2}{2-n}}_{a_{j}}(a_{i})}\right\} 
\; \text{ and } \;
\mathcal{B}_{2}
= 
\left\{ 
E\sqrt{\gamma_{n}G^{\frac{2}{2-n}}_{a_{j}}(a_{i})}
\leq 
\big\vert \frac{x}{\lambda_{i}}\big\vert 
\leq c
\right\}
\end{split}
\end{equation*} 
 for a sufficiently large constant $E>0$.
We then may estimate
\begin{equation*}
\begin{split}
I_{\mathcal{B}_{1}}
= &
\underset{\mathcal{B}_{1}}{\int} \frac{1}{(1+r^{2})^{\frac{n+2}{2}-\theta}}
 \big(\frac{1}{\frac{\lambda_{i} }{ \lambda_{j} }+\lambda_{i}\lambda_{j}\gamma_{n}G_{ a _{j}}^{\frac{2}{2-n}}(\exp_{g_{ a _{i}}} \frac{x}{\lambda_{i}})}\big)^{\frac{n-2}{2}}d\mu_{g_{a_{i}}} \\
\leq &
\frac{C(\frac{\lambda_{i}}{\lambda_{j}})^{\frac{n+2}{2}}}{(1+\lambda_{i}^{2}\gamma_{n}G^{\frac{2}{2-n}}_{a_{j}}(a_{i}))^{\frac{n+2}{2}-\theta}} 
\underset{\left\{ \vert \frac{x}{\lambda_{j}}\vert \leq E\sqrt{\gamma_{n}G^{\frac{2}{2-n}}_{a_{j}}(a_{i})}\right\}}{\int} 
 \big(\frac{1}{1+\lambda_{j}^{2}\gamma_{n}G_{ a _{j}}^{\frac{2}{2-n}}(\exp_{g_{ a _{i}}} \frac{x}{\lambda_{j}})}\big)^{\frac{n-2}{2}}
d\mu_{g_{a_{i}}}.
\end{split}
\end{equation*} 
Changing coordinates via $d_{i,j}=\exp_{g_{a_{i}}}^{-1}\exp_{g_{a_{j}}}$, we get
\begin{equation*}
\begin{split}
I_{\mathcal{B}_{1}}
\leq &
\frac{C}{(\frac{\lambda_{j}}{\lambda_{i}}+\lambda_{i}\lambda_{j}G^{\frac{2}{2-n}}_{a_{j}}(a_{i}))^{\frac{n+2}{2}-\theta}}
\underset{\left\{ \vert \frac{x}{\lambda_{j}}\vert \leq \tilde  E d(a_{i},a_{j}) \right\}}{\int} 
 \big(\frac{1}{1+r^{2}}\big)^{\frac{n-2}{2}} dx, 
\end{split}
\end{equation*} 
and thus $I_{\mathcal{B}_{1}}=O(\varepsilon_{i,j}^{\frac{n}{n-2}-\tau})=o(\varepsilon_{i,j}^{\frac{n+2}{n}})$ using, \eqref{lambda_d_lambda_interaction_discernation_1}. Moreover
\begin{equation*}
\begin{split}
I_{\mathcal{B}_{2}}
= &
\underset{\mathcal{B}_{2}}{\int} \frac{ 1}{(1+r^{2})^{\frac{n+2}{2}-\theta}}
 (\frac{1}{\frac{\lambda_{i} }{ \lambda_{j} }+\lambda_{i}\lambda_{j}\gamma_{n}G_{ a _{j}}^{\frac{2}{2-n}}(\exp_{g_{ a _{i}}} \frac{x}{\lambda_{i}})})^{\frac{n-2}{2}}d\mu_{g_{a_{i}}} \\
 \leq &
\frac{C}{(\frac{\lambda_{i} }{ \lambda_{j} }+\lambda_{i}\lambda_{j}\gamma_{n}G_{ a _{j}}^{\frac{2}{2-n}}(a_{i}))^{\frac{n-2}{2}}} 
\underset{\left\{\vert x \vert \geq \sqrt{\lambda_{i}^{2}\gamma_{n}G^{\frac{2}{2-n}}_{a_{j}}(a_{i})}\right\}}{\int}
\frac{dx}{(1+r^{2})^{\frac{n+2}{2}}}.
\end{split}
\end{equation*} 
This shows
$
I_{\mathcal{A}^{c}}\lesssim I_{\mathcal{B}_{1}}+I_{\mathcal{B}_{2}}=o(\varepsilon_{i,j}^{\frac{n+2}{n}})
$, 
and we arrive at
\begin{equation*}\begin{split}
-&\lambda_{i}^{\theta}\lambda_{j}\int  \varphi_{i}^{\frac{n+2}{n-2}}\varphi_{j} d\mu_{g_{0}}
=
\frac
{\tilde b_{2}u_{ a _{j}}( a _{i})}
{(\frac{\lambda_{i} }{ \lambda_{j} }
+
\lambda_{i}\lambda_{j}
\gamma_{n}G_{ a _{j}}^{\frac{2}{2-n}}( a _{i})
)^{\frac{n-2}{2}}
}
\frac
{\lambda_{j}^{2}\gamma_{n}G_{a_{j}}^{\frac{2}{2-n}}(a_{i})-1}
{\lambda_{j}^{2}\gamma_{n}G_{a_{j}}^{\frac{2}{2-n}}(a_{i})+1} 
\end{split}\end{equation*}
up to some error of the form $R_{i,j}$. Due to conformal covariance, there holds
\begin{equation*}\begin{split}
G_{a_{j}}(a_{j},a_{i})=u_{a_{j}}^{-1}(a_{i})u_{a_{j}}^{-1}(a_{j})G_{g_{0}}(a_{i},a_{j})
\end{split}\end{equation*}
and we therefore conclude
\begin{equation}\label{didj_first_two_cases}
-\lambda_{i}^{\theta}\lambda_{j}\int \varphi_{i}^{\frac{n+2}{n-2}}\partial_{\lambda_{j}}\varphi_{j}d\mu_{g_{0}}
= 
b_{2}
\frac
{
\lambda_{i}\lambda_{j}\gamma_{n}G_{g_{0}}^{\frac{2}{2-n}}(a_{i},a_{j})-\frac{\lambda_{i}}{\lambda_{j}}
}
{(\frac{\lambda_{i} }{ \lambda_{j} }
+
\lambda_{i}\lambda_{j}
\gamma_{n}G^{\frac{2}{2-n}}_{g_{0}}( a _{i}, a _{j})
)^{\frac{n}{2}}
} 
+
R_{i,j}.
\end{equation}
We turn to the case left by \eqref{lambda_d_lambda_interaction_discernation_1}, i.e. 
\begin{equation}\label{lambda_d_lambda_interaction_discernation_2}
\varepsilon_{i,j}^{\frac{2}{2-n}}\sim \frac{\lambda_{j}}{\lambda_{i}} 
\end{equation}
and, recalling \eqref{non_linear_symmetry}, estimate  for $c>0$ small
\begin{equation*}
\begin{split}
&-\lambda_{i}^{\theta}\lambda_{j}\int 
\varphi_{i}^{1-\tau} \partial_{\lambda_{j}}\varphi_{j}^{\frac{n+2}{n-2}}d\mu_{g_{0}}
\\ & = 
\frac{n+2}{2}\lambda_{i}^{\theta}\underset{B_{c}( a _{j})}{\int}
\big(\frac{ \lambda_{i} }{ 1+\lambda_{i}^{2}\gamma_{n}G_{ a _{i}}^{\frac{2}{2-n}}}\big)^{\frac{n-2}{2}-\theta} 
\frac{u_{ a _{i}}^{1-\tau}}{u_{ a _{j}}}\big(\frac{ \lambda_{j} }{ 1+\lambda_{j}^{2}\gamma_{n}G_{ a _{j}}^{\frac{2}{2-n}}}\big)^{\frac{n+2}{2}}
\frac{\lambda_{j}^{2}\gamma_{n}G_{a_{j}}^{\frac{2}{2-n}}-1}{\lambda_{j}^{2}\gamma_{n}G_{a_{j}}^{\frac{2}{2-n}}+1}
d\mu_{g_{ a _{j}}}
\end{split}
\end{equation*}
up to some error $R_{i,j}$, whence up to the same error
\begin{equation*}\begin{split}
-\lambda_{i}^{\theta}\lambda_{j}\int  \varphi_{i} & \partial_{\lambda_{j}}\varphi_{j}^{\frac{n+2}{n-2}} d\mu_{g_{0}} 
= 
\frac{n+2}{2}
\underset{B_{ c\lambda_{j} }(0)}{\int}
\frac
{r^{2}-1}
{r^{2}+1}
(\frac{1}{1+r^{2}})^{\frac{n+2}{2}}
\frac
{
u_{ a _{i}}^{1-\tau}( a _{j})(\frac{\lambda_{i}}{\lambda_{j}})^{\theta}d\mu_{g_{a_{i}}}
}
{
(\frac{\lambda_{j}}{\lambda_{i}}+\lambda_{i}\lambda_{j}\gamma_{n}G^{\frac{2}{2-n}}_{a_{i}}(\exp_{g_{a_{j}}}\frac{x}{\lambda_{j}}))^{\frac{n-2}{2}-\theta}
}.
\end{split}\end{equation*}
 On  
$
\mathcal{A}
= 
\left\{ 
\big\vert  \frac{x}{\lambda_{j}} \big\vert \leq \varepsilon \sqrt{\gamma_{n}G^{\frac{2}{2-n}}_{a_{i}}(a_{j})}
\right\} 
\cup
\left\{ \big\vert  \frac{x}{\lambda_{j}}\big\vert \leq \epsilon  \frac{1}{\lambda_{i}} \right\} 
$
we may expand for $\epsilon>0$ sufficiently small
\begin{equation*}\begin{split}
\big(
\frac{\lambda_{j} }{ \lambda_{i} }
+
\lambda_{i}\lambda_{j}\gamma_{n}  G_{ a_{i}}^{\frac{2}{2-n}}(\exp_{g_{ a_{j}}} \frac{x}{\lambda_{i}})\big)^{\frac{2-n}{2}+\theta} 
= &
(
\frac{\lambda_{j} }{ \lambda_{i} }
+
\lambda_{i}\lambda_{j}
\gamma_{n}G_{ a_{i}}^{\frac{2}{2-n}}( a_{j})
)^{\frac{2-n}{2}+\theta} 
\\ & +
(\frac{2-n}{2}+\theta)
\frac
{
\gamma_{n}\nabla G^{\frac{2}{2-n}}_{ a_{i}}( a_{j})\lambda_{i}x
+
O\big(\frac{ \lambda_{i} }{ \lambda_{j} }\vert x\vert^{2}\big)
}
{
\big(
\frac{\lambda_{j} }{ \lambda_{i} }
+
\lambda_{i}\lambda_{j}\gamma_{n}G_{ a_{i}}^{\frac{2}{2-n}}( a_{j})
\big)^{\frac{n}{2}-\theta}
}.
\end{split}\end{equation*}
With analogous estimates as in the previous case we derive 
\begin{equation*} 
-\lambda_{i}^{\theta}\lambda_{j}\int \varphi_{i}^{\frac{n+2}{n-2}}\partial_{\lambda_{j}}\varphi_{j}d\mu_{g_{0}}
= 
 \bar b_{2}
\frac
{u_{ a_{i}}^{1-\tau}( a_{j})(\frac{\lambda_{i}}{\lambda_{j}})^{\theta}}
{(\frac{\lambda_{j} }{ \lambda_{i} }
+
\lambda_{i}\lambda_{j}
\gamma_{n}G_{ a_{i}}^{\frac{2}{2-n}}( a_{j})
)^{\frac{n-2}{2}-\theta}
} 
+
R_{i,j}
\end{equation*}
with 
\begin{equation}\label{def_bar_b2}
\bar b_{2}=\frac{n+2}{2}\underset{\R^{n}}{\int} \frac{r^{2}-1}{r^{2}+1}(\frac{1}{1+r^{2}})^{\frac{n+2}{2}}dx 
\end{equation}
and  indeed $\bar b_{2}=\tilde b_{2} = \frac{n-2}{2n} \omega_{n}$  
whence, using conformal covariance, as before \eqref{lambda_d_lambda_interaction_discernation_2} implies
\begin{equation}\begin{split}\label{didj_third_case}
-\lambda_{j}\int \varphi_{i}^{\frac{n+2}{n-2}}\partial_{\lambda_{j}}\varphi_{j}d\mu_{g_{0}}
= &
\frac
{
\bar b_{2}
}
{(\frac{\lambda_{i} }{ \lambda_{j} }
+
\lambda_{i}\lambda_{j}
\gamma_{n}G_{g_{0}}^{\frac{2}{2-n}}( a _{i}, a _{j})
)^{\frac{n-2}{2}}
}  +
R_{i,j}.
\end{split}\end{equation}
Now the claim follows comparing \eqref{didj_first_two_cases}  under \eqref{lambda_d_lambda_interaction_discernation_1}  
and \eqref{didj_third_case} under \eqref{lambda_d_lambda_interaction_discernation_2}.

 \item The first claim, i.e. that for $k\neq l$
\begin{equation*}
\int \varphi_{i}^{\frac{4}{n-2}-\tau} \phi_{k,i}\phi_{l,i}d\mu_{g_{0}}
= 
O\big(\frac{1}{\lambda_{i}^{n-2+\theta}}+\frac{1}{\lambda_{i}^{2+\theta}}\big)
\end{equation*}
follows like in case (ii), just with vanishing leading terms. The second one is proved analogously to (ii), cf. case ($\alpha$) in the proof.

\item The case $\tau=0$ is known, cf.  e.g. \cite{may-cv}, Lemma 3.4. By Lemma \ref{lem_interactions} we therefore have
\begin{equation*}
\int \varphi_{i}^{\alpha-\tau}\varphi_{j}^{\beta}d\mu_{g_{0}}
=
\int (\varphi_{i}^{\alpha-\tau}-\frac{1}{\lambda_{i}^{\theta}}\varphi_{i}^{\alpha})\varphi_{j}^{\beta}d\mu_{g_{0}}
+
O(\lambda_{i}^{-\theta}\varepsilon_{i,j}^{\beta}).
\end{equation*}
To estimate the integral in the above right-hand side, we write 
\begin{equation*}
\begin{split}
\int \varphi_{i}^{\alpha-\tau} &\big\vert 1-\frac{1}{\lambda_{i}^{\theta}}\varphi_{i}^{\tau}\big\vert \varphi_{j}^{\beta}d\mu_{g_{0}}
\leq 
\int^{1}_{0} d \sigma \int_{B_{c}(a_{i})}
\varphi_{i}^{\alpha-\tau}\big\vert \partial_{\sigma}(\frac{1}{1+\lambda_{i}^{2}r_{a_{i}}^{2}})^{\sigma \theta}\big\vert \varphi_{j}^{\beta}d\mu_{g_{0}}
\\
\leq &
\theta
\int_{B_{c}(a_{i})}
\varphi_{i}^{\alpha-\tau}\big\vert \ln \frac{1}{1+\lambda_{i}^{2}r_{a_{i}}^{2}}\big\vert \varphi_{j}^{\beta}d\mu_{g_{0}} \\
\leq &
\theta  
\Vert \varphi_{i}^{\alpha-\beta-\varepsilon-\tau}\ln\frac{1}{1+\lambda_{i}^{2}r_{a_{i}}^{2}}\Vert_{L^{\frac{2n}{2n-(n-2)(2\beta+\varepsilon)}}_{B_{c}(a_{i}),\mu_{g_{0}}}}
\Vert \varphi_{i}^{\beta+\varepsilon}\varphi_{j}^{\beta}\Vert_{L^{\frac{2n}{(n-2)(2\beta+\varepsilon)}}_{\mu_{g_{0}}}}.
\end{split}
\end{equation*}
From the case $\tau = 0$ and $\alpha+\beta=\frac{2n}{n-2}$ we then get
\begin{equation*}
\int  \varphi_{i}^{\alpha-\tau} \big\vert 1-\frac{1}{\lambda_{i}^{\theta}}\varphi_{i}^{\tau}\big\vert \varphi_{j}^{\beta},\mu_{g_{0}} 
 \leq 
C\theta \varepsilon_{i,j}^{\beta}
\big\Vert (\frac{\lambda_{i}}{1+\lambda_{i}^{2}r^{2}})^{(n-2)\alpha - n -\frac{n-2}{2}(\varepsilon+\tau)}\ln\frac{1}{1+\lambda_{i}^{2}r^{2}}
\big\Vert_{L^{\frac{n}{(n-2)\alpha -n - \frac{n-2}{2}\varepsilon}}_{B_{c}(0)}}.
\end{equation*}
By direct evaluation  the latter norm is of order $\lambda_{i}^{-\theta}$ and the claim follows. 
 \item also follows from the same above reference in \cite{may-cv}, while  (vii)  is  a straightforward computation. 
\end{enumerate}
\vspace{-12pt}
\end{proof}

\subsection{Derivatives}

In this appendix we give the remaining proofs from Section \ref{s:funct-infty}.

\begin{proof}[Proof of Lemma \ref{lem_alpha_derivatives_at_infinity}]
First note that the equalities up to the error in \eqref{eq:error-19-8}
\begin{equation*}
\partial J_{\tau}(u) \phi_{1,j}
= 
\partial J_{\tau}(\alpha^{j}\varphi_{j})\phi_{1,i}
=
\partial_{\alpha_{j}}J_{\tau}(\alpha^{i}\varphi_{i}) 
\end{equation*}
follow from Lemma
\ref{lem_v_part_interactions} and the chain rule of differentiation.
So we evaluate
\begin{equation*}
\partial J_{\tau}(\alpha^{i}\varphi_{i})\varphi_{j}
= 
\frac{2}{(\int K(\alpha^{i}\varphi_{i})^{p+1}d\mu_{g_{0}})^{\frac{2}{p+1}}} 
 \bigg(
\int
\alpha^{i} \varphi_{i} L_{g_{0}}\varphi_{j} d \mu_{g_{0}}
 -
\frac{\int (\alpha^{i}\varphi_{i}) L_{g_{0}}(\alpha^{k}\varphi_{k})d\mu_{g_{0}}}{\int K(\alpha^{i}\varphi_{i})^{p+1}d\mu_{g_{0}}}
 K(\alpha^{i}\varphi_{i})^{p}\varphi_{j}d\mu_{g_{0}}
\bigg)
\end{equation*}
and start expanding 
\begin{equation*}
\begin{split}
\int  K &( \alpha^{i} \varphi_{i})^{p}\varphi_{j}d\mu_{g_{0}}
= 
\underset{\{\alpha_{j}\varphi_{j}>\underset{j\neq i}{\sum}\alpha_{i}\varphi_{i}\}}{\int}
K\alpha_{j}^{p}\varphi_{j}^{p+1}
+
p\sum_{j\neq i}K\alpha_{j}^{p-1}\alpha_{i}\varphi_{j}^{p}\varphi_{i}
d\mu_{g_{0}}\\
& +
\underset{\{\alpha_{j}\varphi_{j}\leq \sum_{j\neq i}\alpha_{i}\varphi_{i}\}}{\int}
K(\sum_{j\neq i} \alpha_{i}\varphi_{i})^{p}\varphi_{j}d\mu_{g_{0}}  
+
O\big(\sum_{r\neq s}\int_{\{\varphi_{r}\geq \varphi_{s}\}}\varphi_{r}^{p-1}\varphi_{s}^{2}d\mu_{g_{0}}\big) 
d\mu_{g_{0}}.
\end{split}
\end{equation*}
The above error term  is of order $O\big(\sum_{r\neq s}\varepsilon_{r,s}^{\frac{n+2}{n}}\big)$ by Lemma \ref{lem_interactions}, whence
\begin{equation*}
\int  K(\alpha^{i}\varphi_{i})^{p}\varphi_{j}d\mu_{g_{0}}
= 
\int
K\alpha_{j}^{p}\varphi_{j}^{p+1}
+
p\sum_{j\neq i}K\alpha_{j}^{p-1}\alpha_{i}\varphi_{j}^{p}\varphi_{i}
d\mu_{g_{0}}
 +
\underset{\{\alpha_{j}\varphi_{j}\leq \sum_{i\neq j}\alpha_{i}\varphi_{i}\}}{\int}
K(\sum_{i\neq j} \alpha_{i}\varphi_{i})^{p}\varphi_{j}d\mu_{g_{0}}, 
\end{equation*}
up to an error of order $O(\sum_{r\neq s}\varepsilon_{r,s}^{\frac{n+2}{n}})$. Similarly 
\begin{equation*}
\underset{\{\alpha_{j}\varphi_{j}\leq \underset{j\neq i}{\sum}\alpha_{i}\varphi_{i}\}}{\int}
\hspace{-10pt}
K(\sum_{j\neq i} \alpha_{i}\varphi_{i})^{p}\varphi_{j}d\mu_{g_{0}}  = 
\underset
{
 \begin{matrix}
\scriptstyle\{\alpha_{j}\varphi_{j}\leq \sum_{j\neq i}\alpha_{i}\varphi_{i}\} \\
\scriptstyle \cap \; \; \{j\neq 1\} \; \cap \\
\scriptstyle\{\alpha_{1}\varphi_{1}>\sum_{j,1\neq i}\alpha_{i}\varphi_{i}\}
\end{matrix}
}
{\int}
\hspace{-10pt}
K\alpha_{1}^{p}\varphi_{1}^{p}\varphi_{j} d\mu_{g_{0}}
+
\underset
{
 \begin{matrix}
\scriptstyle\{\alpha_{j}\varphi_{j}\leq \underset{j\neq i}{\sum}\alpha_{i}\varphi_{i}\} \\
\scriptstyle\cap \; \; \{j\neq 1\} \; \cap \\
\scriptstyle\{\alpha_{1}\varphi_{1}\leq \underset{j,1\neq i}{\sum}\alpha_{i}\varphi_{i}\}
\end{matrix}
}
{\int}
\hspace{-10pt}
K(\underset{j,1\neq i}{\sum}\alpha_{i}\varphi_{i})^{p}\varphi_{j}d\mu_{g_{0}}
\end{equation*}
up to an error $O(\sum_{r\neq s}\varepsilon_{r,s}^{\frac{n+2}{n}})$, and thus
\begin{equation*}
\begin{split}
\underset{\{\alpha_{j}\varphi_{j}\leq \sum_{j\neq i}\alpha_{i}\varphi_{i}\}}{\int}
K(\sum_{j\neq i} \alpha_{i}\varphi_{i})^{p}\varphi_{j} d \mu_{g_{0}}
= & 
\chi_{\{1\neq j\}}
\int 
K\alpha_{1}^{p}\varphi_{1}^{p}\varphi_{j} d \mu_{g_{0}} \\
& +
\chi_{\{1\neq j\}}
\underset
{
\{\alpha_{j}\varphi_{j}\leq \sum_{j\neq i}\alpha_{i}\varphi_{i}\} 
}
{\int}
K(\sum_{j,1\neq i}\alpha_{i}\varphi_{i})^{p}\varphi_{j} d \mu_{g_{0}}.
\end{split}
\end{equation*}
Iteratively we  obtain 
$
\underset{\{\alpha_{j}\varphi_{j}\leq \sum_{j\neq i}\alpha_{i}\varphi_{i}\}}{\int}
K(\sum_{j\neq i} \alpha_{i}\varphi_{i})^{p}\varphi_{j} d \mu_{g_{0}}
= 
\sum_{j\neq i}\int K\alpha_{i}^{p}\varphi_{i}^{p}\varphi_{j} d \mu_{g_{0}}
$ 
and conclude
\begin{equation}\label{expansion_nonlinear_with_testing_1}
\begin{split}
\int  K(\alpha^{i}\varphi_{i})^{p}\varphi_{j}d\mu_{g_{0}}
= &
\sum_{i}\alpha_{i}^{p}\int K \varphi_{i}^{p}\varphi_{j}d\mu_{g_{0}}
+
p\sum_{j\neq i}K\alpha_{j}^{p-1}\alpha_{i}\varphi_{j}^{p}\varphi_{i}
d\mu_{g_{0}}
\end{split}
\end{equation}
up to an error of order $O(\sum_{r\neq s}\varepsilon_{r,s}^{\frac{n+2}{n}}).$
From this, we obviously have
\begin{equation*}
\begin{split}
\partial J_{\tau}(\alpha^{i}\varphi_{i})\varphi_{j}
= &
\frac{2}{(\int K(\alpha^{i}\varphi_{i})^{p+1}d\mu_{g_{0}})^{\frac{2}{p+1}}} 
\bigg(
\int
\alpha_{j} \varphi_{j} L_{g_{0}}\varphi_{j} d \mu_{g_{0}}
 -
\frac{\int (\alpha^{i}\varphi_{i}) L_{g_{0}}(\alpha^{k}\varphi_{k})d\mu_{g_{0}}}{\int K(\alpha^{i}\varphi_{i})^{p+1}d\mu_{g_{0}}}
\alpha_{j}^{p} K\varphi_{j}^{p+1}d\mu_{g_{0}}
\bigg) \\
& +
\frac{2}{(\sum_{k}\alpha_{k}^{p+1}\int K\varphi_{k}^{p+1}d\mu_{g_{0}})^{\frac{2}{p+1}}} \\
& \quad 
\sum_{j\neq i}\bigg(
\int
\alpha_{i}L_{g_{0}}\varphi_{j} d \mu_{g_{0}}
 -
\frac{\sum_{k}\alpha_{k}^{2}\int \varphi_{k} L_{g_{0}}\varphi_{k}d\mu_{g_{0}}}
{\sum_{k}\alpha_{k}^{p+1}\int K\varphi_{k}^{p+1}d\mu_{g_{0}}}
\alpha_{i}^{p}K\varphi_{i}^{p}\varphi_{j}d\mu_{g_{0}}
\bigg)
\\
& -
\frac{2p\sum_{k}\alpha_{k}^{2}\int \varphi_{k} L_{g_{0}}\varphi_{k}d\mu_{g_{0}}}{(\sum_{k}\alpha_{k}^{p+1}\int K\varphi_{k}^{p+1}d\mu_{g_{0}})^{\frac{2}{p+1}+1}} 
\sum_{j\neq i}\alpha_{j}^{p-1}\alpha_{i}\int K\varphi_{j}^{p}\varphi_{i}d\mu_{g_{0}}
\end{split}
\end{equation*}
up to some 
$ 
O(\sum_{r\neq s}\varepsilon_{r,s}^{\frac{n+2}{n}}).
$ 
Then \eqref{intergral_sum_of_bubble_nonlinear_evaluated} and \eqref{single_bubble_L_g_0_integral_expansion_exact}
applied to the second and third summands above show 
\begin{equation*}
\begin{split}
&\partial J_{\tau}(\alpha^{i}\varphi_{i})\varphi_{j}
= 
\frac{2}{(\int K(\alpha^{i}\varphi_{i})^{p+1}d\mu_{g_{0}})^{\frac{2}{p+1}}}  \bigg(
\int
\alpha_{j}\varphi_{j}L_{g_{0}}\varphi_{j} d \mu_{g_{0}}
 -
\frac{\int (\alpha^{i}\varphi_{i}) L_{g_{0}}(\alpha^{k}\varphi_{k})d\mu_{g_{0}}}{\int K(\alpha^{i}\varphi_{i})^{p+1}d\mu_{g_{0}}}
\alpha_{j}^{p} K\varphi_{j}^{p+1}d\mu_{g_{0}}
\bigg) \\
& +
\frac{8n(n-1)\bar c_{0}^{\frac{p-1}{p+1}}b_{1}}{(\alpha_{K,\tau}^{p+1})^{\frac{2}{p+1}}} 
\sum_{j\neq i}
\alpha_{i}\big(
1
-
\frac{\alpha^{2}}
{\alpha_{K,\tau}^{p+1}}
\frac{K_{i}}{\lambda_{i}^{\theta}}\alpha_{i}^{p-1}
\big)\varepsilon_{i,j}
 -
\frac{8 p n(n-1)\bar c_{0}^{-\frac{2}{p+1}}b_{1}\alpha^{2}}{(\alpha_{K,\tau}^{p+1})^{\frac{2}{p+1}+1}} 
\sum_{j\neq i}\alpha_{i}\frac{K_{j}}{\lambda_{j}^{\theta}}\alpha_{j}^{p-1}\varepsilon_{i,j}
\end{split}
\end{equation*}
up to an
$ 
O
(
\tau^{2}
+
\sum_{r\neq s}
\frac{\vert \nabla K_{r}\vert^{2}}{\lambda_{r}^{2}}
+
\frac{1}{\lambda_{r}^{4}}
+
\frac{1}{\lambda_{r}^{2(n-2)}}
+
\varepsilon_{r,s}^{\frac{n+2}{n}}
).
$ 
Then applying  \eqref{single_bubble_L_g_0_integral_expansion_exact} as well as \eqref{L_g_0_bubble_interaction} and 
Lemma \ref{lem_interactions} to the first summand above we find
\begin{equation*}
\begin{split}
& \partial J_{\tau}(\alpha^{i}\varphi_{i})\varphi_{j}
= 
\frac{8n(n-1)}{(\int K(\alpha^{i}\varphi_{i})^{p+1}d\mu_{g_{0}})^{\frac{2}{p+1}}}  \bigg(
\bar c_{0}
\alpha_{j}
 -
\frac{\bar c_{0}\alpha^{2}+b_{1}\sum_{k\neq l}\alpha_{k}\alpha_{l}\varepsilon_{k,l}}{\int K(\alpha^{i}\varphi_{i})^{p+1}d\mu_{g_{0}}}
\alpha_{j}^{p} \int K\varphi_{j}^{p+1}d\mu_{g_{0}}
\bigg) 
\\
& +
\frac{8n(n-1)\bar c_{0}^{\frac{p-1}{p+1}}b_{1}}{(\alpha_{K,\tau}^{p+1})^{\frac{2}{p+1}}} 
\sum_{j\neq i}
\alpha_{i}\big(
1
-
\frac{\alpha^{2}}
{\alpha_{K,\tau}^{p+1}}
\frac{K_{i}}{\lambda_{i}^{\theta}}\alpha_{i}^{p-1}
\big)\varepsilon_{i,j}
 -
\frac{8 p n(n-1)\bar c_{0}^{-\frac{2}{p+1}}b_{1}\alpha^{2}}{(\alpha_{K,\tau}^{p+1})^{\frac{2}{p+1}+1}} 
\sum_{j\neq i}\alpha_{i}\frac{K_{j}}{\lambda_{j}^{\theta}}\alpha_{j}^{p-1}\varepsilon_{i,j}.
\end{split}
\end{equation*}
Using \eqref{intergral_sum_of_bubble_nonlinear_evaluated} for the first term in the right-hand side, we then get
\begin{equation*}
\begin{split}
\partial J_{\tau}(\alpha^{i}\varphi_{i})\varphi_{j}
= &
\frac{8n(n-1)\bar c_{0}
\alpha_{j}}{(\int K(\alpha^{i}\varphi_{i})^{p+1}d\mu_{g_{0}})^{\frac{2}{p+1}}} \bigg(
1
 -
\frac{\alpha^{2}\alpha_{j}^{p-1} }{\int K(\alpha^{i}\varphi_{i})^{p+1}d\mu_{g_{0}}}
\int K\varphi_{j}^{p+1}d\mu_{g_{0}}
\bigg) \\
& -
\frac{8n(n-1)\bar c_{0}^{-\frac{2}{p+1}}b_{1}}{(\alpha_{K,\tau}^{p+1})^{\frac{2}{p+1}}} 
\frac{\sum_{k\neq l}\alpha_{k}\alpha_{l}\varepsilon_{k,l}}{\sum_{k}\frac{K_{k}}{\lambda_{k}^{\theta}}\alpha_{k}^{p+1}}
\alpha_{j}^{p}\frac{K_{j}}{\lambda_{j}^{\theta}}
 -
\frac{8 p n(n-1)\bar c_{0}^{-\frac{2}{p+1}}b_{1}\alpha^{2}}{(\alpha_{K,\tau}^{p+1})^{\frac{2}{p+1}+1}} 
\sum_{j\neq i}\alpha_{i}\frac{K_{j}}{\lambda_{j}^{\theta}}\alpha_{j}^{p-1}\varepsilon_{i,j}
\end{split}
\end{equation*}
up to an error of order 
\begin{equation*}
O
\big(
\tau^{2}
+
\sum_{r\neq s} 
\big\vert 1-\frac{\alpha^{2}}{\alpha_{K,\tau}^{\frac{2n}{n-2}}}\frac{K_{r}}{\lambda_{r}^{\theta}}\alpha_{r}^{\frac{4}{n-2}}\big\vert^{2}
+
\frac{\vert \nabla K_{r}\vert^{2}}{\lambda_{r}^{2}}
+
\frac{1}{\lambda_{r}^{4}}
+
\frac{1}{\lambda_{r}^{2(n-2)}}
+
\varepsilon_{r,s}^{\frac{n+2}{n}}
\big).
\end{equation*}
Applying now \eqref{intergral_sum_of_bubble_nonlinear_evaluated} to the first coefficient above we find
\begin{equation*}
\begin{split}
& \partial J_{\tau}(\alpha^{i}  \varphi_{i})\varphi_{j}
= 
\frac{8n(n-1)\bar c_{0}^{\frac{p-1}{p+1}}\alpha_{j}}{(\alpha_{K,\tau}^{p+1})^{\frac{2}{p+1}}}  \bigg(
1
-
\frac{\alpha^{2}}{\int K(\alpha^{i}\varphi_{i})^{p+1}d\mu_{g_{0}}}
\alpha_{j}^{p-1}\int K\varphi_{j}^{p+1}d\mu_{g_{0}}
\bigg) \\
& -
\frac{16n(n-1)}{p+1}
\frac{\bar c_{0}^{-\frac{2}{p+1}}}{(\alpha_{K,\tau}^{p+1})^{\frac{2}{p+1}+1}} 
\alpha_{j}\bigg(
1
 -
\frac{\alpha^{2}}{\alpha_{K,\tau}^{p+1}}
 \frac{K_{j}}{\alpha_{j}^{\theta}}\alpha_{j}^{p-1}
\bigg) \\
& \quad 
\bigg(
\bar c_{1}\sum_{i}\frac{K_{i}}{\lambda_{i}^{\theta}}\alpha_{i}^{\frac{2n}{n-2}}\tau 
+
\bar c_{2}\sum_{i}\frac{\Delta K_{i}}{\lambda_{i}^{2+\theta}} \alpha_{i}^{\frac{2n}{n-2}} 
+
\bar d_{1}\sum_{i}\frac{K_{i}}{\lambda_{i}^{\theta}}\alpha_{i}^{\frac{2n}{n-2}}
\begin{pmatrix}
\frac{H_{i}}{\lambda_{i}} \\
\frac{H_{i}+O(\frac{\ln \lambda_{i}}{\lambda_{i}^{2}})}{\lambda_{i}^{2 }} \\
\frac{H_{i}}{\lambda_{i}^{3 }} \\
\frac{W_{i}\ln \lambda_{i}}{\lambda_{i}^{4}} \\
0 
\end{pmatrix}  +
\bar b_{1}\sum_{i\neq j}\alpha_{i}^{\frac{n+2}{n-2}}\alpha_{j}
\frac{K_{i}}{\lambda_{i}^{\theta}}\varepsilon_{i,j}
\bigg)
\\
& -
\frac{8n(n-1)\bar c_{0}^{-\frac{2}{p+1}}b_{1}\alpha^{2}}{(\alpha_{K,\tau}^{p+1})^{\frac{2}{p+1}+1}}
\frac{K_{j}}{\lambda_{j}^{\theta}}\alpha_{j}^{p-1}
\bigg(
\sum_{k\neq l}\alpha_{j}\frac{\alpha_{k}\alpha_{l}}{\alpha^{2}}\varepsilon_{k,l}
+
p
\sum_{j\neq i}\alpha_{i}\varepsilon_{i,j}
\bigg),
\end{split}
\end{equation*}
and obviously the second summand is of order of the previous error term. Thus
\begin{equation*}
\begin{split}
\partial J_{\tau}(\alpha^{i}\varphi_{i})\varphi_{j}
= &
\frac{8n(n-1)\bar c_{0}^{\frac{p-1}{p+1}}}{(\alpha_{K,\tau}^{p+1})^{\frac{2}{p+1}}}
\alpha_{j}\bigg(
1
-
\frac{\alpha^{2}}{\int K(\alpha^{i}\varphi_{i})^{p+1}d\mu_{g_{0}}}
\alpha_{j}^{p-1} \int K\varphi_{j}^{p+1}d\mu_{g_{0}}
\bigg)
\\
& -
\frac{8n(n-1)\bar c_{0}^{-\frac{2}{p+1}}b_{1}\alpha^{2}}{(\alpha_{K,\tau}^{p+1})^{\frac{2}{p+1}+1}}
\frac{K_{j}}{\lambda_{j}^{\theta}}\alpha_{j}^{p-1}
\bigg(
\sum_{k\neq l}\alpha_{j}\frac{\alpha_{k}\alpha_{l}}{\alpha^{2}}\varepsilon_{k,l}
+
p
\sum_{j\neq i}\alpha_{i}\varepsilon_{i,j}
\bigg)
\end{split}
\end{equation*}
up to the same error, and applying finally \eqref{single_bubble_K_integral_expansion_exact} and \eqref{intergral_sum_of_bubble_nonlinear_evaluated}
we arrive at
\begin{equation*}
\begin{split}
\partial  J_{\tau} (  \alpha^{i} \varphi_{i})\varphi_{j}
= &
\frac{8n(n-1)\bar c_{0}^{\frac{p-1}{p+1}}}{(\alpha_{K,\tau}^{p+1})^{\frac{2}{p+1}}}
\alpha_{j}\bigg(
1
-
\frac{\alpha^{2}}{\alpha_{K,\tau}^{p+1}}
\frac{K_{j}}{\lambda_{j}^{\theta}}\alpha_{j}^{p-1}
\bigg) \\
& -
\frac{8n(n-1)\bar c_{0}^{-\frac{2}{p+1}}\alpha^{2}\alpha_{j}^{p}}{(\alpha_{K,\tau}^{p+1})^{\frac{2}{p+1}+1}}
\bigg(
\bar c_{1}\frac{K_{j}\tau}{\lambda_{j}^{\theta}}
+
\bar c_{2}\frac{\Delta K_{j}}{\lambda_{j}^{2+\theta}} 
+
\bar d_{1}K_{j}
\begin{pmatrix}
\frac{H_{j}}{\lambda_{j}^{1+\theta }} \\
\frac{H_{j}}{\lambda_{j}^{2+\theta }}+O(\frac{\ln \lambda_{i}}{\lambda_{i}^{4+\theta}}) \\
\frac{H_{j}}{\lambda_{j}^{3+\theta }} \\
\frac{W_{j}\ln \lambda_{j}}{\lambda_{i}^{4+\theta}} \\
0 
\end{pmatrix} 
\bigg)
\\
& + 
\frac{8n(n-1)\bar c_{0}^{-\frac{2}{p+1}}\alpha^{2}\alpha_{j}^{p}}{(\alpha_{K,\tau}^{p+1})^{\frac{2}{p+1}+2}}
\frac{K_{j}}{\lambda_{j}^{\theta}}
\quad 
\bigg(
\bar c_{1}\sum_{k}\frac{K_{k}}{\lambda_{k}^{\theta}}\alpha_{k}^{\frac{2n}{n-2}}\tau  
+
\bar c_{2}\sum_{k}\frac{\Delta K_{k}}{\lambda_{k}^{2+\theta}} \alpha_{k}^{\frac{2n}{n-2}} 
\\
& \quad \quad\quad\quad\quad\quad\;  +
\bar d_{1}\sum_{k}\frac{K_{k}}{\lambda_{k}^{\theta}}\alpha_{k}^{\frac{2n}{n-2}}
\begin{pmatrix}
\frac{H_{k}}{\lambda_{k}} \\
\frac{H_{k}+O(\frac{\ln \lambda_{i}}{\lambda_{k}^{2}})}{\lambda_{k}^{2 }} \\
\frac{H_{k}}{\lambda_{k}^{3 }} \\
\frac{W_{k}\ln \lambda_{k}}{\lambda_{k}^{4}} \\
0 
\end{pmatrix} 
+
\bar b_{1}\sum_{k\neq l}\alpha_{k}^{\frac{n+2}{n-2}}\alpha_{l}
\frac{K_{k}}{\lambda_{k}^{\theta}}\varepsilon_{k,l}
\bigg)
\\
& -
\frac{8n(n-1)\bar c_{0}^{-\frac{2}{p+1}}b_{1}\alpha^{2}}{(\alpha_{K,\tau}^{p+1})^{\frac{2}{p+1}+1}}
\frac{K_{j}}{\lambda_{j}^{\theta}}\alpha_{j}^{p-1}
\bigg(
\sum_{k\neq l}\alpha_{j}\frac{\alpha_{k}\alpha_{l}}{\alpha^{2}}\varepsilon_{k,l}
+
p
\sum_{j\neq i}\alpha_{i}\varepsilon_{i,j}
\bigg),
\end{split}
\end{equation*}
again up to the same error term. 
Recalling that $\bar b_{1}=\frac{2n}{n-2}b_{1}$, we can rewrite this as 
\begin{equation*}
\begin{split}
 \partial J_{\tau} & (\alpha^{i}\varphi_{i})\varphi_{j}
= 
\frac{8n(n-1)\bar c_{0}^{\frac{p-1}{p+1}}}{(\alpha_{K,\tau}^{p+1})^{\frac{2}{p+1}}}
\alpha_{j}\bigg(
1
-
\frac{\alpha^{2}}{\alpha_{K,\tau}^{p+1}}
\frac{K_{j}}{\lambda_{j}^{\theta}}\alpha_{j}^{p-1}
\bigg) 
 -
\frac{8n(n-1)\bar c_{0}^{-\frac{n-2}{n}}\alpha^{2}\alpha_{j}^{\frac{n+2}{n-2}}}
{(\alpha_{K,\tau}^{\frac{2n}{n-2}})^{\frac{n-2}{n}+1}}
\frac{K_{j}}{\lambda_{j}^{\theta}} \\
& \quad 
\bigg(
\bar c_{1}
\big(1-\sum_{k}\frac{K_{k}}{\lambda_{k}^{\theta}}\frac{\alpha_{k}^{\frac{2n}{n-2}}}{\alpha_{K,\tau}^{\frac{2n}{n-2}}}\big) \tau
+
\bar c_{2}
\big(
\frac{\Delta K_{j}}{K_{j}\lambda_{j}^{2}}
-
\sum_{k}\frac{\Delta K_{k}}{K_{k}\lambda_{k}^{2}}\frac{\frac{K_{k}}{\lambda_{k}^{\theta}}\alpha_{k}^{\frac{2n}{n-2}}}{\alpha_{K,\tau}^{\frac{2n}{n-2}}}
\big)
\\
&  
\quad \quad +
\bar d_{1}
\begin{pmatrix}
\frac{H_{j}}{\lambda_{j}}-\sum_{k}\frac{\frac{K_{k}}{\lambda_{k}^{\theta}}\alpha_{k}^{\frac{2n}{n-2}}}{\alpha_{K,\tau}^{\frac{2n}{n-2}}}\frac{H_{k}}{\lambda_{k}} \\
\frac{H_{j}}{\lambda_{j}^{2 }}-\sum_{k}\frac{\frac{K_{k}}{\lambda_{k}^{\theta}}\alpha_{k}^{\frac{2n}{n-2}}}{\alpha_{K,\tau}^{\frac{2n}{n-2}}}\frac{H_{k}+O(\sum_{r}\frac{\ln \lambda_{r}}{\lambda_{r}^{2}})}{\lambda_{k}^{2 }} \\
\frac{H_{j}}{\lambda_{j}^{3}} -\sum_{k}\frac{\frac{K_{k}}{\lambda_{k}^{\theta}}\alpha_{k}^{\frac{2n}{n-2}}}{\alpha_{K,\tau}^{\frac{2n}{n-2}}}\frac{H_{k}}{\lambda_{k}^{3 }} \\
\frac{W_{j}\ln \lambda_{j}}{\lambda_{i}^{4}}-\sum_{k}\frac{\frac{K_{k}}{\lambda_{k}^{\theta}}\alpha_{k}^{\frac{2n}{n-2}}}{\alpha_{K,\tau}^{\frac{2n}{n-2}}}\frac{W_{k}\ln \lambda_{k}}{\lambda_{k}^{4}} \\
0 
\end{pmatrix} 
\bigg)
\\
& -
\frac{8n(n-1)\bar c_{0}^{-\frac{n-2}{n}}b_{1}\alpha^{2}}{(\alpha_{K,\tau}^{\frac{2n}{n-2}})^{\frac{n-2}{n}+1}}
\frac{K_{j}}{\lambda_{j}^{\theta}}\alpha_{j}^{\frac{4}{n-2}} 
 \bigg(
\sum_{k\neq l}
\alpha_{j}
\big(\frac{\alpha_{k}\alpha_{l}}{\alpha^{2}}
- 
\frac{}{}
\frac{2n}{n-2}
\frac{K_{k}}{\lambda_{k}^{\theta}}
\frac{\alpha_{k}^{\frac{n+2}{n-2}}\alpha_{l}}{\alpha_{K,\tau}^{\frac{2n}{n-2}}}
\big)
\varepsilon_{k,l}
+
\frac{n+2}{n-2}
\sum_{j\neq i}\alpha_{i}\varepsilon_{i,j}
\bigg)
\end{split}
\end{equation*} 
up to an error of the form 
\begin{equation*}
O
\big(
\tau^{2}
+
\sum_{r\neq s} 
\big\vert 1-\frac{\alpha^{2}}{\alpha_{K,\tau}^{\frac{2n}{n-2}}}\frac{K_{r}}{\lambda_{r}^{\theta}}\alpha_{r}^{\frac{4}{n-2}}\big\vert^{2}
+
\frac{\vert \nabla K_{r}\vert^{2}}{\lambda_{r}^{2}}
+
\frac{1}{\lambda_{r}^{4}}
+
\frac{1}{\lambda_{r}^{2(n-2)}}
+
\varepsilon_{r,s}^{\frac{n+2}{n}}
\big).
\end{equation*}
Note that by \eqref{eq:akt} the coefficient of $\bar c_{1}$ in the above term  vanishes. This then tells us in a first step, that 
\begin{equation*}
\forall\; i\;:\; 1-\frac{\alpha^{2}}{\alpha_{K,\tau}^{p+1}}\frac{K_{j}}{\lambda_{j}^{\theta}}\alpha_{j}^{p-1}
=
O
\left(
\tau^{2}
+
\sum_{r}\frac{1}{\lambda_{r}^{2}}+\frac{1}{\lambda_{r}^{n-2}}+\sum_{r\neq s}\varepsilon_{r,s}+\vert \partial J_{\tau}(u)\vert
\right)
\end{equation*}
and therefore
\begin{equation*}
\forall\; i\;:\; 1-\frac{\alpha^{2}}{\alpha_{K,\tau}^{\frac{2n}{n-2}}}\frac{K_{j}}{\lambda_{j}^{\theta}}\alpha_{j}^{\frac{4}{n-2}}
=
O
\left(
\tau +
\sum_{r}\frac{1}{\lambda_{r}^{2}}+\frac{1}{\lambda_{r}^{n-2}}+\sum_{r\neq s}\varepsilon_{r,s}+\vert \partial J_{\tau}(u)\vert
\right). 
\end{equation*}
Using this we derive  
up to an error of the form 
$
O
\big(
\tau^{2}
+
\sum_{r\neq s} 
\frac{\vert \nabla K_{r}\vert^{2}}{\lambda_{r}^{2}}
+
\frac{1}{\lambda_{r}^{4}}
+
\frac{1}{\lambda_{r}^{2(n-2)}}
+
\varepsilon_{r,s}^{\frac{n+2}{n}}
+
\vert \partial J_{\tau}(u)\vert^{2}
\big)
$
\begin{equation*}
\begin{split}
& \partial J_{\tau}(\alpha^{i}\varphi_{i})\varphi_{j}
= 
\frac{8n(n-1)\bar c_{0}^{\frac{2}{n}}}{(\alpha_{K,\tau}^{\frac{2n}{n-2}})^{\frac{n-2}{n}}}
\alpha_{j}\bigg(
1
-
\frac{\alpha^{2}}{\alpha_{K,\tau}^{p+1}}
\frac{K_{j}}{\lambda_{j}^{\theta}}\alpha_{j}^{p-1}
\bigg) \\
& -
\frac{8n(n-1)\bar c_{0}^{-\frac{n-2}{n}}\alpha_{j}}
{(\alpha_{K,\tau}^{\frac{2n}{n-2}})^{\frac{n-2}{n}}}
\bigg(
\bar c_{2}
(
\frac{\Delta K_{j}}{K_{j}\lambda_{j}^{2}}
-
\sum_{k}\frac{\Delta K_{k}}{K_{k}\lambda_{k}^{2}}
\frac{\alpha_{k}^{2}}{\alpha^{2}}
)
+
\bar d_{1}
\begin{pmatrix}
\frac{H_{j}}{\lambda_{j}}-\sum_{k}\frac{\alpha_{k}^{2}}{\alpha^{2}}\frac{H_{k}}{\lambda_{k}} \\
\frac{H_{j}}{\lambda_{j}^{2 }}-\sum_{k}\frac{\alpha_{k}^{2}}{\alpha^{2}}\frac{H_{k}+O(\sum_{r}\frac{\ln \lambda_{r}}{\lambda_{r}^{2}})}{\lambda_{k}^{2 }} \\
\frac{H_{j}}{\lambda_{j}^{3}} -\sum_{k}\frac{\alpha_{k}^{2}}{\alpha^{2}}\frac{H_{k}}{\lambda_{k}^{3 }} \\
\frac{W_{j}\ln \lambda_{j}}{\lambda_{i}^{4}}-\sum_{k}\frac{\alpha_{k}^{2}}{\alpha^{2}}\frac{W_{k}\ln \lambda_{k}}{\lambda_{k}^{4}} \\
0 
\end{pmatrix} 
\bigg)
\\
& -
\frac{8n(n-1)\bar c_{0}^{-\frac{n-2}{n}}b_{1}}{(\alpha_{K,\tau}^{\frac{2n}{n-2}})^{\frac{n-2}{n}}}
 \quad \bigg(
\sum_{k\neq l}
\alpha_{j}
(\frac{\alpha_{k}\alpha_{l}}{\alpha^{2}}
- 
\frac{2n}{n-2}
\frac{\alpha_{k}\alpha_{l}}{\alpha^{2}}
)
\varepsilon_{k,l}
+
\frac{n+2}{n-2}
\sum_{j\neq i}\alpha_{i}\varepsilon_{i,j}
\bigg).
\end{split}
\end{equation*}
Finally note that the last summand can be simplified to 
\begin{equation*}
\begin{split}
\frac{n+2}{n-2}\frac{8n(n-1)\bar c_{0}^{-\frac{n-2}{n}}b_{1}}{(\alpha_{K,\tau}^{\frac{2n}{n-2}})^{\frac{n-2}{n}}}
 \bigg(
\sum_{k\neq l}
\alpha_{j}
\frac{\alpha_{k}\alpha_{l}}{\alpha^{2}}
\varepsilon_{k,l}
-
\sum_{j\neq i}\alpha_{i}\varepsilon_{i,j}
\bigg).
\end{split}
\end{equation*}
From this the lemma follows setting  
\begin{equation}\label{constants_alpha_derivative}
\grave b_{1}
=
\frac{8n(n-1)(n+2)}{\bar c_{0}^{\frac{n-2}{n}}(n-2)}b_{1}
,  \quad 
\grave c_{2}
=
\frac{{8n}(n-1)}{\bar c_{0}^{\frac{n-2}{n}}}
\bar c_{2}
,  \quad 
\grave d_{1}
=
\frac{{8n}(n-1)}{\bar c_{0}^{\frac{n-2}{n}}}
\bar d_{1}
,  \quad 
\grave c_{0}= 8n(n-1)\bar c_{0}^{\frac{2}{n}},
\end{equation}
cf. \eqref{mass_integral_expansion}, \eqref{single_bubble_K_integral_expansion_exact} and Lemma \ref{lem_interactions}.
\end{proof}

\begin{proof}[Proof of Lemma \ref{lem_lambda_derivatives_at_infinity}]
From Lemma \ref{lem_v_part_interactions} and the chain rule of differentiation we obtain
\begin{equation*}
\partial J_{\tau}(u)\phi_{2,j}
= 
\partial J_{\tau}(\alpha^{i}\varphi_{i})\phi_{2,j}=\lambda_{j}\partial_{\lambda_{j}}J_{\tau}(\alpha^{i}\varphi_{i}),
\end{equation*}
up to the error in \eqref{eq:error-19-8-2}, and evaluate
$
\partial  J_{\tau}(\alpha^{i}  \varphi_{i})\phi_{2,j}
= 
\frac{2\Lambda}{(\int K(\alpha^{i}\varphi_{i})^{p+1}d\mu_{g_{0}})^{\frac{2}{p+1}}}
$
with
\begin{equation*}
\begin{split}
\Lambda
= &
\int
\alpha^{i}\varphi_{i}L_{g_{0}}\lambda_{j}\partial_{\lambda_{j}}\varphi_{j} d \mu_{g_{0}}
 -
\frac{\int (\alpha^{i}\varphi_{i}) L_{g_{0}}(\alpha^{k}\varphi_{k})d\mu_{g_{0}}}{\int K(\alpha^{i}\varphi_{i})^{p+1}d\mu_{g_{0}}}
K(\alpha^{i}\varphi_{i})^{p}\lambda_{j}\partial_{\lambda_{j}}\varphi_{j}d\mu_{g_{0}}. 
\end{split}
\end{equation*}
Arguing as for \eqref{expansion_nonlinear_with_testing_1}, we find 
\begin{equation*}
\begin{split}
\Lambda
= &
\alpha_{j}\int \varphi_{j} L_{g_{0}}\lambda_{j}\partial_{\lambda_{j}}\varphi_{j} d \mu_{g_{0}}
 -
\frac{\int (\alpha^{i}\varphi_{i}) L_{g_{0}}(\alpha^{k}\varphi_{k})d\mu_{g_{0}}}{\int K(\alpha^{i}\varphi_{i})^{p+1}d\mu_{g_{0}}}
K\alpha_{j}^{p}\varphi_{j}^{p}\lambda_{j}\partial_{\lambda_{j}}\varphi_{j}d\mu_{g_{0}} \\
& +
\sum_{j\neq i}
\alpha_{i}\int \varphi_{i}
L_{g_{0}}\lambda_{j}\partial_{\lambda_{j}}\varphi_{j} d \mu_{g_{0}}
-
\frac{\int (\alpha^{i}\varphi_{i}) L_{g_{0}}(\alpha^{k}\varphi_{k})d\mu_{g_{0}}}{\int K(\alpha^{i}\varphi_{i})^{p+1}d\mu_{g_{0}}}
K\alpha_{i}^{p}\varphi_{i}^{p}\lambda_{j}\partial_{\lambda_{j}}\varphi_{j}d\mu_{g_{0}} \\
& -
p\frac{\int (\alpha^{i}\varphi_{i}) L_{g_{0}}(\alpha^{k}\varphi_{k})d\mu_{g_{0}}}{\int K(\alpha^{i}\varphi_{i})^{p+1}d\mu_{g_{0}}}
\sum_{j\neq i}\int K\alpha_{j}^{p-1}\alpha_{i}\varphi_{j}^{p-2}\varphi_{i}\lambda_{j}\partial_{\lambda_{j}}\varphi_{j}d\mu_{g_{0}} 
\end{split}
\end{equation*}
and arguing as for  \eqref{linearized_bubble_interaction}
\eqref{L_g_0_bubble_interaction}, \eqref{single_bubble_L_g_0_integral_expansion_exact} we see  that
\begin{equation*}
\int K \varphi_{i}^{p}\lambda_{j}\partial_{\lambda_{j}}\varphi_{j}d\mu_{g_{0}}
=
b_{2}\frac{K_{i}}{\lambda_{i}^{\theta}}\lambda_{j}\partial_{\lambda_{j}}\varepsilon_{i,j}
+
O
\big(
\tau^{2}
+
\sum_{r\neq s}
\frac{\vert \nabla K_{r}\vert^{2}}{\lambda_{r}^{2}}
+
\frac{1}{\lambda_{r}^{4}}
+
\frac{1}{\lambda_{r}^{2(n-2)}}
+
\varepsilon_{r,s}^{\frac{n+2}{n}}
\big),
\end{equation*}
and
\begin{equation}\label{def_tilde_b2}
\int \varphi_{i} L_{g_{0}}\lambda_{j}\partial_{\lambda_{j}}\varphi_{j} d \mu_{g_{0}}
=
\tilde b_{2}\lambda_{j}\partial_{\lambda_{j}}\varepsilon_{i,j}
+
O\big(\sum_{r\neq s}\frac{1}{\lambda_{r}^{4}}+\frac{1}{\lambda_{r}^{2(n-2)}}+\varepsilon_{r,s}^{\frac{n+2}{n}}\big),\; \tilde b_{2}=4n(n-1)b_{2} 
\end{equation}
as well as
$
\int \varphi_{j} L_{g_{0}}\lambda_{j}\partial_{\lambda_{j}}\varphi_{j}d\mu_{g_{0}}
=
O\big(\tau^{2}+\frac{1}{\lambda_{j}^{4}}+\frac{1}{\lambda_{j}^{2(n-2)}}\big).
$
 Using these, we arrive at
\begin{equation*}
\begin{split}
\Lambda
= &
 -
\frac{\int (\alpha^{i}\varphi_{i}) L_{g_{0}}(\alpha^{k}\varphi_{k})d\mu_{g_{0}}}{\int K(\alpha^{i}\varphi_{i})^{p+1}d\mu_{g_{0}}}
\int K\alpha_{j}^{p}\varphi_{j}^{p}\lambda_{j}\partial_{\lambda_{j}}\varphi_{j}d\mu_{g_{0}} \\
& +
4n(n-1)b_{2}\sum_{j\neq i}
\alpha_{i}\partial_{\lambda_{j}}\varepsilon_{i,j}
-
\frac{\int (\alpha^{i}\varphi_{i}) L_{g_{0}}(\alpha^{k}\varphi_{k})d\mu_{g_{0}}}{\int K(\alpha^{i}\varphi_{i})^{p+1}d\mu_{g_{0}}}
b_{2}\frac{K_{i}}{\lambda_{i}^{\theta}}\alpha_{i}^{p}\lambda_{j}\partial_{\lambda_{j}}\varepsilon_{i,j} \\
& -
p\frac{\int (\alpha^{i}\varphi_{i}) L_{g_{0}}(\alpha^{k}\varphi_{k})d\mu_{g_{0}}}{\int K(\alpha^{i}\varphi_{i})^{p+1}d\mu_{g_{0}}}
\sum_{j\neq i }\alpha_{j}^{p-1}\alpha_{i}\int K\varphi_{j}^{p-1}\varphi_{i}\lambda_{j}\partial_{\lambda_{j}}\varphi_{j}d\mu_{g_{0}} \\
& +
O
\big(
\tau^{2}
+
\sum_{r\neq s}
\frac{\vert \nabla K_{r}\vert^{2}}{\lambda_{r}^{2}}
+
\frac{1}{\lambda_{r}^{4}}
+
\frac{1}{\lambda_{r}^{2(n-2)}}
+
\varepsilon_{r,s}^{\frac{n+2}{n}}
\big)
\end{split}
\end{equation*}
Moreover, still arguing as for \eqref{L_g_0_bubble_interaction} and using Lemma \ref{lem_interactions}, we have up to the same error as above
\begin{equation*}
\int K\varphi_{j}^{p-1}\varphi_{i}\lambda_{j}\partial_{\lambda_{j}}\varphi_{j}d\mu_{g_{0}} 
= 
\frac{b_{2}}{p}\frac{K_{j}}{\lambda_{j}^{\theta}}\lambda_{j}\partial_{\lambda_{j}}\varepsilon_{i,j}.
\end{equation*}
Combining this with 
\eqref{intergral_sum_of_bubble_nonlinear_evaluated}, \eqref{L_g_0_bubble_interaction} and \eqref{single_bubble_L_g_0_integral_expansion_exact}
we get with the same precision 
\begin{equation*}
\begin{split}
\Lambda
= & 
-
4n(n-1)\frac{\alpha^{2}+\sum_{k\neq l}\alpha_{k}\alpha_{l}\varepsilon_{k,l}}{\int K(\alpha^{i}\varphi_{i})^{p+1}d\mu_{g_{0}}}
\int K\alpha_{j}^{p}\varphi_{j}^{p}\lambda_{j}\partial_{\lambda_{j}}\varphi_{j}d\mu_{g_{0}}  +
4n(n-1)b_{2}\sum_{j\neq i}
\alpha_{i}\partial_{\lambda_{j}}\varepsilon_{i,j}
\\& 
- 
\frac{4n(n-1)\alpha^{2}}{\alpha_{K,\tau}^{p+1}}b_{2}
\frac{K_{i}}{\lambda_{i}^{\theta}}\alpha_{i}^{p}\lambda_{j}\partial_{\lambda_{j}}\varepsilon_{i,j} \
 -
\frac{4n(n-1)b_{2}\alpha^{2}}{\alpha_{K,\tau}^{p+1}}
\sum_{j\neq i }\alpha_{i}\frac{K_{j}}{\lambda_{j}^{\theta}}\alpha_{j}^{p-1}\lambda_{j}\partial_{\lambda_{j}}\varepsilon_{i,j}.
\end{split}
\end{equation*}
Using Lemma \ref{lem_alpha_derivatives_at_infinity} we find by cancellation
\begin{equation*}
\Lambda
= 
-
4n(n-1)\frac{\alpha^{2}+\sum_{k\neq l}\alpha_{k}\alpha_{l}\varepsilon_{k,l}}{\int K(\alpha^{i}\varphi_{i})^{p+1}d\mu_{g_{0}}}
\int K\alpha_{j}^{p}\varphi_{j}^{p}\lambda_{j}\partial_{\lambda_{j}}\varphi_{j}d\mu_{g_{0}} 
 -
4n(n-1)b_{2}
\sum_{j\neq i }\alpha_{i}\lambda_{j}\partial_{\lambda_{j}}\varepsilon_{i,j},
\end{equation*}
up to some 
$
O
\big(
\tau^{2}
+
\sum_{r\neq s}
\frac{\vert \nabla K_{r}\vert^{2}}{\lambda_{r}^{2}}
+
\frac{1}{\lambda_{r}^{4}}
+
\frac{1}{\lambda_{r}^{2(n-2)}}
+
\varepsilon_{r,s}^{\frac{n+2}{n}}
+
\vert \partial J_{\tau}(u)\vert^{2}
\big).
$
Moreover from Lemma \ref{lem_interactions} we have
\begin{equation*}
\begin{split}
\int K\varphi_{j}^{p}\lambda_{j}\partial_{\lambda_{j}}\varphi_{j}d\mu_{g_{0}} 
= &
K_{j}\int \alpha_{j}^{p}\varphi_{j}^{p}\lambda_{j}\partial_{\lambda_{j}}\varphi_{j}d\mu_{g_{0}}
+
O\big(\frac{\vert \nabla K_{j}\vert }{\lambda_{j}^{1+\theta}}+O\big(\frac{1}{\lambda_{j}^{2}}\big)\big) \\
= &
O\big( 
\frac{\tau}{\lambda_{j}^{\theta}}+\frac{1}{\lambda_{j}^{n-2+\theta}}
+
\frac{\vert \nabla K_{j}\vert }{\lambda_{j}^{1+\theta}}+O\big(\frac{1}{\lambda_{j}^{2+\theta}}\big)\big),
\end{split}
\end{equation*}
whence recalling \eqref{intergral_sum_of_bubble_nonlinear_evaluated} we get 
\begin{equation*}
\Lambda
= 
-
4n(n-1)\frac{\alpha^{2}}{\alpha_{K,\tau}^{\frac{2n}{n-2}}}
\int K\alpha_{j}^{p}\varphi_{j}^{p}\lambda_{j}\partial_{\lambda_{j}}\varphi_{j}d\mu_{g_{0}}  -
4n(n-1)b_{2}
\sum_{j\neq i }\alpha_{i}\lambda_{j}\partial_{\lambda_{j}}\varepsilon_{i,j}
\end{equation*}
up to some 
$
O
\big(
\tau^{2}
+
\sum_{r\neq s}
\frac{\vert \nabla K_{r}\vert^{2}}{\lambda_{r}^{2}}
+
\frac{1}{\lambda_{r}^{4}}
+
\frac{1}{\lambda_{r}^{2(n-2)}}
+
\varepsilon_{r,s}^{\frac{n+2}{n}}
+
\vert \partial J_{\tau}(u)\vert^{2}
\big).
$
 Therefore  
\begin{equation}\label{lambda_testing_after_interaction}
\begin{split}
& \partial  J_{\tau}(\alpha^{i} \varphi_{i})\phi_{2,j}
= 
\frac{2\Lambda}{(\int K(\alpha^{i}\varphi_{i})^{p+1}d\mu_{g_{0}})^{\frac{2}{p+1}}} \\
= &
-
\frac{4n(n-1)\bar c_{0}^{-\frac{n-2}{n}}\alpha^{2}}{(\alpha_{K,\tau}^{\frac{2n}{n-2}})^{\frac{n-2}{n}+1}}\alpha_{j}^{p}
\int K\varphi_{j}^{p}\lambda_{j}\partial_{\lambda_{j}}\varphi_{j}d\mu_{g_{0}}  -
\frac{4n(n-1)\bar c_{0}^{-\frac{n-2}{n}}b_{2}}{(\alpha_{K,\tau}^{\frac{2n}{n-2}})^{\frac{n-2}{n}}}
\sum_{j\neq i }\alpha_{i}\lambda_{j}\partial_{\lambda_{j}}\varepsilon_{i,j}
\end{split}
\end{equation}
up to the same error. Thus we are left with analysing 
\begin{equation*}
\begin{split}
\int K\varphi_{j}^{p} & \lambda_{j}\partial_{\lambda_{j}}\varphi_{j}d\mu_{g_{0}}
= 
\int_{B_{c}(a_{j})} K\varphi_{j}^{p}\partial_{\lambda_{j}}\varphi_{j} d\mu_{g_{0}}
+
O\big(\frac{1}{\lambda_{j}^{n-\theta}}\big) \\
= &
K_{j}\int_{B_{c}(a_{j})} \varphi_{j}^{p} \partial_{\lambda_{j}}\varphi_{j}d\mu_{g_{0}}
+
\nabla K_{j}\int_{B_{c}(a_{j})} x\varphi_{j}^{p}\partial_{\lambda_{j}}\varphi_{j} d\mu_{g_{0}} \\
& +
\frac{\nabla^{2}}{2}K_{j}\int_{B_{c}(a_{j})} x^{2}\varphi_{j}^{p}\partial_{\lambda_{j}}\varphi_{j} d\mu_{g_{0}}
+
\frac{\nabla^{3}}{6}K_{j}\int_{B_{c}(a_{j})} x^{3}\varphi_{j}^{p}\partial_{\lambda_{j}}\varphi_{j} d\mu_{g_{0}} 
 +
O\big(\frac{1}{\lambda_{j}^{4}}+\frac{1}{\lambda_{j}^{2(n-2)}}\big). 
\end{split}
\end{equation*}
Expanding the bubble $\varphi_{j}$ and its derivative $\lambda_{j}\partial_{\lambda_{j}}\varphi_{j}$ in conformal normal coordinates, i.e.
\begin{equation*}
\begin{split}
(p+1)\varphi_{j}^{p} & \partial_{j}\varphi_{j}
= 
\lambda_{j}\partial_{j}\varphi_{j}^{\frac{2n}{n-2}-\tau}
= 
u_{a_{j}}^{\frac{2n}{n-2}-\tau}
\lambda_{j}
\partial_{\lambda_{j}}\big(\frac{\lambda_{j}}{1+\lambda_{j}^{2}r_{a_{j}}^{2}(1+r_{a_{j}}^{n-2}H_{a_{j}})^{\frac{2}{2-n}}}\big)^{n-\theta} \\
= &
(n-\theta)
\big(\frac{\lambda_{j}}{1+\lambda_{j}^{2}r_{a_{j}}^{2}(1+r_{a_{j}}^{n-2}H_{a_{j}})^{\frac{2}{2-n}}}\big)^{n-\theta}
\frac{1-\lambda_{j}^{2}r_{a_{j}}^{2}(1+r_{a_{j}}^{n-2}H_{a_{j}})^{\frac{2}{2-n}}}{1+\lambda_{j}^{2}r_{a_{j}}^{2}(1+r_{a_{j}}^{n-2}H_{a_{j}})^{\frac{2}{2-n}}} \\
= &
(n-\theta)
\big(\frac{\lambda_{j}}{1+\lambda_{j}^{2}r_{a_{j}}^{2}}\big)^{n-\theta}
\frac{1-\lambda_{j}^{2}r_{a_{j}}^{2}}{1+\lambda_{j}^{2}r_{a_{j}}^{2}} 
+
\frac{2(n-\theta)^{2}}{n-2}
\big(\frac{\lambda_{j}}{1+\lambda_{j}^{2}r_{a_{j}}^{2}}\big)^{n-\theta}
\frac{\lambda_{j}^{2}r_{a_{j}}^{n}H_{a_{j}}}{1+\lambda_{j}^{2}r_{a_{j}}^{2}}
\frac{\frac{n+2-\theta}{n-2\theta}\lambda_{j}^{2}r_{a_{j}}^{2}-1}{1+\lambda_{j}^{2}r_{a_{j}}^{2}}
\\
& +
O
\big(
\big(\frac{\lambda_{j}}{1+\lambda_{j}^{2}r_{a_{j}}^{2}}\big)^{n-\theta}
\frac{\lambda_{j}^{4}r_{a_{j}}^{2n}H_{a_{j}}^{2}}{(1+\lambda_{j}^{2}r_{a_{j}}^{2})^{2}}
\big)
\end{split}
\end{equation*}
and arguing as for \eqref{single_bubble_K_integral_expansion} we find using radial symmetry
\begin{enumerate}
 \item[(1)] \quad 
 $
 \int_{B_{c}(a_{j})} x\varphi_{j}^{p}\partial_{\lambda_{j}}\varphi_{j} d\mu_{g_{0}} 
\; , \; 
\int_{B_{c}(a_{j})} x^{3}\varphi_{j}^{p}\partial_{\lambda_{j}}\varphi_{j} d\mu_{g_{0}} 
=
O\big(\tau^{2}+\frac{1}{\lambda_{j}^{4}}+\frac{1}{\lambda_{j}^{2(n-2)}}\big);
 $
 \item[(2)] \quad 
 $
 \frac{\nabla^{2}}{2}K_{j}\int_{B_{c}(a_{j})} x^{2}\varphi_{j}^{p}\partial_{\lambda_{j}}\varphi_{j} d\mu_{g_{0}} 
=
\frac{n-2}{4n}\frac{\Delta K_{j}}{\lambda_{j}^{2+\theta}}\underset{\R^{n}}{\int}\frac{r^{2}(1-r^{2})}{(1+r^{2})^{n+1}}dx
+
O(\tau^{2}+\frac{1}{\lambda_{j}^{4}}+\frac{1}{\lambda_{j}^{2(n-2)}})
 $,
\end{enumerate}
Finally we have 
\begin{equation*}
\begin{split}
\int_{B_{c}(a_{j})}\varphi_{j}^{p}\lambda_{j}\partial_{\lambda_{j}}\varphi_{j}d\mu_{g_{0}} 
= &
\frac{n-2}{2}
\int_{B_{c}(0)}
\big(\frac{\lambda_{j}}{1+\lambda_{j}^{2}r^{2}}\big)^{n-\theta}
\frac{1-\lambda_{j}^{2}r^{2}}{1+\lambda_{j}^{2}r^{2}} dx \\
& +
\int_{B_{c}(0)}
\big(\frac{\lambda_{j}}{1+\lambda_{j}^{2}r^{2}}\big)^{n-\theta}
\frac{\lambda_{j}^{2}r^{n}H_{a_{j}}}{1+\lambda_{j}^{2}r^{2}}
\frac{n+2-n\lambda_{j}^{2}r^{2}}{1+\lambda_{j}^{2}r^{2}} dx
\end{split}
\end{equation*}
up to some $O(\tau^{2}+\frac{1}{\lambda_{j}^{4}}+\frac{1}{\lambda_{j}^{2(n-2)}})$, and see that for the first summand above there holds
\begin{equation*}
\begin{split}
\frac{n-2}{2}\int_{B_{c}(0)}
\big(\frac{\lambda_{j}}{1+\lambda_{j}^{2}r^{2}}\big)^{n-\theta}
\frac{1-\lambda_{j}^{2}r^{2}}{1+\lambda_{j}^{2}r^{2}}dx
=
-\frac{n-2}{2}\frac{\theta}{\lambda_{j}^{\theta}}\underset{\R^{n}}{\int}\big(\frac{1}{1+r^{2}}\big)^{n}\frac{1-r^{2}}{1+r^{2}}\ln\frac{1}{1+r^{2}}dx, 
\end{split}
\end{equation*}
up to the same error. Defining 
\begin{equation}\label{def_tilde_c1_tilde_c2}
\tilde c_{1}=\frac{(n-2)^{2}}{4}\underset{\R^{n}}{\int}\frac{1-r^{2}}{(1+r^{2})^{n+1}}\ln\frac{1}{1+r^{2}}dx, \qquad 
\tilde c_{2}=-\frac{n-2}{4n}\underset{\R^{n}}{\int}\frac{r^{2}(1-r^{2})}{(1+r^{2})^{n+1}}dx, 
\end{equation}
it can be shown,  that 
\begin{equation*}
\tilde c_1 = \frac{(n-2)^{2}}{48n} \omega_{n} \frac{\Gamma(n/2)^2}{\Gamma(n)} > 0
\; \text{ and }\;
\tilde c_2 = \frac{n-2}{4n} \omega_{n}  \frac{\Gamma \left(\frac{n}{2}+1\right) \Gamma \left(\frac{n}{2}\right)+\Gamma \left(\frac{n}{2}-1\right) \Gamma \left(\frac{n}{2}+2\right)}{2 \Gamma (n+1)} > 0,
\end{equation*}
so we arrive at
\begin{equation*}
\begin{split}
\int K\varphi_{j}^{p} \lambda_{j}\partial_{\lambda_{j}}\varphi_{j}d\mu_{g_{0}}
= &
-
\tilde c_{1}\frac{K_{j}}{\lambda_{j}^{\theta}}\tau
-
\tilde c_{2}\frac{\Delta K_{j}}{\lambda_{j}^{2+\theta}} 
+
K_{j}\int_{B_{c}(0)}
\big(\frac{\lambda_{j}}{1+\lambda_{j}^{2}r^{2}}\big)^{n-\theta}
\frac{\lambda_{j}^{2}r^{n}H_{a_{j}}}{1+\lambda_{j}^{2}r^{2}}
\frac{n+2-n\lambda_{j}^{2}r^{2}}{1+\lambda_{j}^{2}r^{2}}dx
\end{split}
\end{equation*}
up to some $O(\tau^{2}+\frac{1}{\lambda_{j}^{4}}+\frac{1}{\lambda_{j}^{2(n-2)}})$ and arguing as for \eqref{mass_integral_expansion} we find
\begin{equation}\label{def_tilde_d1}
\begin{split}
& \int_{B_{c}(0)}  (\frac{\lambda_{j}}{1+\lambda_{j}^{2}r^{2}})^{n-\theta}
\frac{n+2-n\lambda_{j}^{2}r^{2}}{1+\lambda_{j}^{2}r^{2}}
\frac{\lambda_{j}^{2}r^{n}H_{a_{j}}}{1+\lambda_{j}^{2}r^{2}} dx \\
& = 
\frac{1}{\lambda_{j}^{n-2+\theta}}\int_{B_{c\lambda_{j}}(0)} 
\frac{n+2-nr^{2}}{(1+r^{2})^{n+2-\theta}}r^{n}
\begin{pmatrix}
H_{j}+\nabla H_{j}\frac{x}{\lambda_{j}}+O(\frac{r^{2}}{\lambda_{j}^{2}}) \\
H_{j}+\nabla H_{j}\frac{x}{\lambda_{j}}+O(\frac{r^{2}}{\lambda_{j}^{2}}\ln \frac{r}{\lambda_{j}}) \\
H_{j}+O(\frac{r}{\lambda_{j}}) \\
-W_{j}\ln \frac{r}{\lambda_{j}} +O(\frac{r}{\lambda_{j}}\ln \frac{r}{\lambda_{j}}) \\
O(\frac{r^{6-n}}{\lambda_{j}^{6-n}}) 
\end{pmatrix} dx
\\
& = 
- 
\tilde d_{1}
\frac
{\vartheta_{j}}
{\lambda_{j}^{\theta}}
+
O\big(\tau^{2}+\frac{1}{\lambda_{j}^{4}}+\frac{1}{\lambda_{j}^{2(n-2)}}\big),
\;
\vartheta_{j}
=
\begin{pmatrix}
\frac{H_{j}}{\lambda_{j}^{1+\theta }} \\
\frac{H_{j}}{\lambda_{j}^{2+\theta }}+O(\frac{\ln \lambda_{j}}{\lambda_{j}^{4+\theta}})\\
\frac{H_{j}}{\lambda_{j}^{3+\theta }} \\
\frac{W_{j}\ln \lambda_{j}}{\lambda_{j}^{4+\theta}} \\
0 
\end{pmatrix} 
,
\;
 \tilde d_{1}
= 
- \underset{\R^{n}}{\int}\frac{r^{n}(n+2-nr^{2})}{(1+r^{2})^{n+2}}dx.
\end{split}
\end{equation}
We conclude that 
\begin{equation*}
\begin{split}
\int K\varphi_{j}^{p} \lambda_{j}\partial_{\lambda_{j}}\varphi_{j}d\mu_{g_{0}}
= &
-
\tilde c_{1}\frac{K_{j}}{\lambda_{j}^{\theta}}\tau
-
\tilde c_{2}\frac{\Delta K_{j}}{\lambda_{j}^{2+\theta}} 
-
\tilde  d_{1} K_{j}
\tilde d_{1}
\frac
{\vartheta_{j}}
{\lambda_{j}^{\theta}}
+
O(\tau^{2}+\frac{1}{\lambda_{j}^{4}}+\frac{1}{\lambda_{j}^{2(n-2)}}).
\end{split}
\end{equation*}
Plugging this into 
\eqref{lambda_testing_after_interaction}, we then have
\begin{equation*}
\begin{split}
\partial J & _{\tau}(u)\phi_{2,j}
= 
\frac{4n(n-1)\bar c_{0}^{-\frac{n-2}{n}}\alpha^{2}}{(\alpha_{K,\tau}^{\frac{2n}{n-2}})^{\frac{n-2}{n}+1}}
\frac{K_{j}}{\lambda_{j}^{\theta}}\alpha_{j}^{p}
\big(
\tilde c_{1}\tau
+
\tilde c_{2}\frac{\Delta K_{j}}{K_{j}\lambda_{j}^{2}} 
+
\tilde  d_{1}\vartheta_{j}
\big) \\
&  -
\frac{4n(n-1)\bar c_{0}^{-\frac{n-2}{n}}b_{2}}{(\alpha_{K,\tau}^{\frac{2n}{n-2}})^{\frac{n-2}{n}}}
\sum_{j\neq i }\alpha_{i}\lambda_{j}\partial_{\lambda_{j}}\varepsilon_{i,j}
+
O
(
\tau^{2}
+
\sum_{r\neq s}
\frac{\vert \nabla K_{r}\vert^{2}}{\lambda_{r}^{2}}
+
\frac{1}{\lambda_{r}^{4}}
+
\frac{1}{\lambda_{r}^{2(n-2)}}
+
\varepsilon_{r,s}^{\frac{n+2}{n}}
+
\vert \partial J_{\tau}(u)\vert^{2}
).
\end{split}
\end{equation*}
Now the claim follows from Lemma \ref{lem_alpha_derivatives_at_infinity}  by replacing the constants as follows 
\begin{equation}\label{def_const_lambda_derivatives}
(\tilde c_{1},\tilde c_{2},\tilde d_{1},\tilde b_{2})
\rightsquigarrow
\frac{4n(n-1)}{\bar c_{0}^{\frac{n-2}{n}}}(\tilde c_{1},\tilde c_{2},\tilde d_{1},b_{2}),
\end{equation}
cf. \eqref{def_tilde_b2}, \eqref{def_tilde_c1_tilde_c2} and \eqref{def_tilde_d1} as well as Lemma \ref{lem_interactions}.
\end{proof}

\begin{proof}[Proof of Lemma \ref{lem_a_derivatives_at_infinity}] 
From Lemma \ref{lem_v_part_interactions} and the chain rule  we obtain up to the 
error in \eqref{eq:error-19-8-3}
\begin{equation*}
\partial J_{\tau}(u)\phi_{3,j}
= 
\partial J_{\tau}(\alpha^{i}\varphi_{i})\phi_{3,j}=\frac{\nabla_{a_{j}}}{\lambda_{j}}J_{\tau}(\alpha^{i}\varphi_{i})
\end{equation*}
and write
\begin{equation}\label{a_expansion_A_term}
\partial  J_{\tau}(\alpha^{i}  \varphi_{i})\phi_{2,j}
= 
\frac{2A}{(\int K(\alpha^{i}\varphi_{i})^{p+1}d\mu_{g_{0}})^{\frac{2}{p+1}}}
\end{equation}
with
\begin{equation*}
\begin{split}
A
= &
\int
\alpha^{i}L_{g_{0}}\varphi_{i}\frac{\nabla_{a_{j}}}{\lambda_{j}}\varphi_{j}
 -
\frac{\int L_{g_{0}}(\alpha^{i}\varphi_{i})(\alpha^{k}\varphi_{k})d\mu_{g_{0}}}{\int K(\alpha^{i}\varphi_{i})^{p+1}d\mu_{g_{0}}}
K(\alpha^{i}\varphi_{i})^{p}\frac{\nabla_{a_{j}}}{\lambda_{j}}\varphi_{j}d\mu_{g_{0}}. 
\end{split}
\end{equation*}
Arguing as for \eqref{expansion_nonlinear_with_testing_1}, we find 
\begin{equation*}
\begin{split}
A
= &
\alpha_{j}\int \varphi_{j}  L_{g_{0}}\frac{\nabla_{a_{j}}}{\lambda_{j}}\varphi_{j} d \mu_{g_{0}}
 -
\frac{\int (\alpha^{i}\varphi_{i}) L_{g_{0}}(\alpha^{k}\varphi_{k})d\mu_{g_{0}}}{\int K(\alpha^{i}\varphi_{i})^{p+1}d\mu_{g_{0}}}
K\alpha_{j}^{p}\varphi_{j}^{p}\frac{\nabla_{a_{j}}}{\lambda_{j}}\varphi_{j}d\mu_{g_{0}} \\
& +
\sum_{j\neq i}
\alpha_{i}\int
\varphi_{i} L_{g_{0}}\frac{\nabla_{a_{j}}}{\lambda_{j}}\varphi_{j} d \mu_{g_{0}}
-
\frac{\int (\alpha^{i}\varphi_{i}) L_{g_{0}}(\alpha^{k}\varphi_{k})d\mu_{g_{0}}}{\int K(\alpha^{i}\varphi_{i})^{p+1}d\mu_{g_{0}}}
K\alpha_{i}^{p}\varphi_{i}^{p}\frac{\nabla_{a_{j}}}{\lambda_{j}}\varphi_{j}d\mu_{g_{0}} \\
& -
p\frac{\int (\alpha^{i}\varphi_{i}) L_{g_{0}}(\alpha^{k}\varphi_{k})d\mu_{g_{0}}}{\int K(\alpha^{i}\varphi_{i})^{p+1}d\mu_{g_{0}}}
\sum_{j\neq i}\int K\alpha_{j}^{p-1}\alpha_{i}\varphi_{j}^{p-2}\varphi_{i}\frac{\nabla_{a_{j}}}{\lambda_{j}}\varphi_{j}d\mu_{g_{0}} 
\end{split}
\end{equation*}
and arguing as for  \eqref{linearized_bubble_interaction} and \eqref{L_g_0_bubble_interaction}, in particular using Lemma \ref{lem_interactions}, we obtain 
\begin{equation*}
\begin{split}
A
= &
\alpha_{j}\int \varphi_{j} L_{g_{0}}\frac{\nabla_{a_{j}}}{\lambda_{j}}\varphi_{j} d \mu_{g_{0}}
 -
\frac{\int (\alpha^{i}\varphi_{i}) L_{g_{0}}(\alpha^{k}\varphi_{k})d\mu_{g_{0}}}{\int K(\alpha^{i}\varphi_{i})^{p+1}d\mu_{g_{0}}}
K\alpha_{j}^{p}\varphi_{j}^{p}\frac{\nabla_{a_{j}}}{\lambda_{j}}\varphi_{j}d\mu_{g_{0}} \\
& -
4n(n-1)b_{3}
\sum_{j\neq i}
\bigg(
\alpha_{i}
-
\frac{\alpha^{2}}{\alpha_{K,\tau}^{p+1}}
\frac{K_{i}}{\lambda_{i}^{\theta}}\alpha_{i}^{p}  
-
p\frac{\alpha^{2}}{\alpha_{K,\tau}^{p+1}}
\frac{K_{j}}{\lambda_{j}^{\theta}}\alpha_{j}^{p-1}\alpha_{i}
\bigg)
\frac{\nabla_{a_{j}}}{\lambda_{j}}\varepsilon_{i,j} \\
= &
\alpha_{j}\int \varphi_{j} L_{g_{0}}\frac{\nabla_{a_{j}}}{\lambda_{j}}\varphi_{j} d \mu_{g_{0}}
 -
\frac{\int (\alpha^{i}\varphi_{i}) L_{g_{0}}(\alpha^{k}\varphi_{k})d\mu_{g_{0}}}{\int K(\alpha^{i}\varphi_{i})^{p+1}d\mu_{g_{0}}}
K\alpha_{j}^{p}\varphi_{j}^{p}\frac{\nabla_{a_{j}}}{\lambda_{j}}\varphi_{j}d\mu_{g_{0}}  -
4n(n-1)b_{3}
\sum_{j\neq i}
\alpha_{i}
\frac{\nabla_{a_{j}}}{\lambda_{j}}\varepsilon_{i,j} 
\end{split}
\end{equation*}
up to some 
\begin{equation*}
O
\big(
\tau^{2}
+
\sum_{r\neq s}
\frac{\vert \nabla K_{r}\vert^{2}}{\lambda_{r}^{2}}
+
\frac{1}{\lambda_{r}^{4}}
+
\frac{1}{\lambda_{r}^{2(n-2)}}
+
\varepsilon_{r,s}^{\frac{n+2}{n}}
+
\vert \partial J_{\tau}(u)\vert^{2}
\big),
\end{equation*}
using Lemma \ref{lem_alpha_derivatives_at_infinity} for the last step. Consider a cut-off function $\eta$ such that 
\begin{equation*}
\eta\in C^{\infty}(M,[0,1]),\;\eta=1 \; \text{ on } \;B_{c}(a)\; \text{ and }\;\eta=0 \; \text{ on }\; B_{2c}^{c}(a), 
\end{equation*}
with $c>0$ sufficiently small and some $a\in M$ sufficiently close to $a_{j}$. Then
\begin{equation*}
\int K\varphi_{j}^{p}\frac{\nabla_{a_{j}}}{\lambda_{j}}\varphi_{j}d\mu_{g_{0}} 
= 
\int K\eta\varphi_{j}^{p}\frac{\nabla_{a_{j}}}{\lambda_{j}}\varphi_{j}d\mu_{g_{0}} +O\big(\frac{1}{\lambda_{j}^{n-\theta}}\big) 
= 
\frac{1}{p+1}\frac{\nabla_{a_{j}}}{\lambda_{j}}\int K\eta \varphi_{j}^{p+1}d\mu_{g_{0}}
+O\big(\frac{1}{\lambda_{j}^{n-\theta}}\big)
\end{equation*}
and passing to conformal normal coordinates around $a_{j}$ we have
\begin{equation*}
\begin{split}
& \frac{\nabla_{a_{j}}}{\lambda_{j}} \int K\eta  \varphi_{j}^{p+1}d\mu_{g_{0}}
= 
\frac{\nabla_{a_{j}}}{\lambda_{j}}
\int 
(K\eta)\circ \exp_{g_{a_{j}}}
\big(
\frac
{\lambda_{j}}
{
1+\lambda_{j}^{2}r^{2}(1+r^{n-2}H_{a_{j}})^{\frac{2}{2-n}}
}
\big)^{n-\theta} dx \\
= &
\int \left( \frac{\nabla_{a_{j}}\exp_{g_{a_{j}}}}{\lambda_{j}} (\eta\; K)\circ\exp_{g_{a_{j}}} \right) 
\big(
\frac
{\lambda_{j}}
{
1+\lambda_{j}^{2}r^{2}(1+r^{n-2}H_{a_{j}})^{\frac{2}{2-n}}
}
\big)^{n-\theta} dx \\
& -
(n-\theta)
\int 
(K\eta)\circ \exp_{g_{a_{j}}}
(
\frac
{\lambda_{j}}
{
1+\lambda_{j}^{2}r^{2}(1+r^{n-2}H_{a_{j}})^{\frac{2}{2-n}}
}
)^{n-\theta} 
\frac
{\lambda_{j}^{2}r^{2}\frac{\nabla_{a_{j}}}{\lambda_{j}}(1+r^{n-2}H_{a_{j}})^{\frac{2}{2-n}}}
{
1+\lambda_{j}^{2}r^{2}(1+r^{n-2}H_{a_{j}})^{\frac{2}{2-n}}
}dx
+
O\big(\frac{1}{\lambda_{j}^{n-\theta}}\big)
\\
= &
\Gamma-(n-\theta)\mathfrak{M}+O\big(\frac{1}{\lambda_{j}^{n-\theta}}\big), 
\end{split}
\end{equation*}
where 
\begin{equation*}
\begin{split}
\mathfrak{M}
= &
\frac{2}{2-n}
\int_{B_{c}(a_{j})}
K(\exp_{g_{a_{j}}})
\big(
\frac
{\lambda_{j}}
{
1+\lambda_{j}^{2}r^{2}
}
\big)^{n-\theta} 
\frac
{\lambda_{j}^{2}r^{n}\frac{\nabla_{a_{j}}}{\lambda_{j}}H_{a_{j}}(1+O(r^{n-2}H_{a_{j}}))}
{
1+\lambda_{j}^{2}r^{2}
} dx
\end{split}
\end{equation*}
up to some $O(\frac{1}{\lambda_{j}^{n-\theta}})$.
From \eqref{conformal_normal_coordinates_expansion} anda\eqref{eq:bubbles}  and using radial symmetry we obtain
\begin{equation}\label{check_c_3_c_4}   
\Gamma
=
\check{c}_{3}\frac{\nabla K_{j}}{\lambda_{j}^{1+\theta}}
+
 \check{c}_{4}\frac{\nabla \Delta K_{j}}{\lambda_{j}^{3+\theta}}
\; \text{ with }\; 
\;
\check c_{3}=\underset{\R^{n}}{\int} \frac{dx}{(1+r^{2})^{n}}
\; \text{ and } \; 
\check c_{4}= \frac{1}{2n} \underset{\R^{n}}{\int} \frac{ r^2 dx}{(1+r^{2})^{n}}.
\end{equation}
up to some $
O
\big(
\tau^{2}
+
\frac{\vert \nabla K_{j}\vert^{2}}{\lambda_{j}^{2}}
+
\frac{1}{\lambda_{j}^{4}}
+
\frac{1}{\lambda_{j}^{2(n-2)}}
\big)$
.
By  \eqref{eq:bubbles} we have
$\nabla H_{a_{j}}H_{a_{j}}=O(1)$ for $n=3,4,5$ and 
\begin{equation*}
\nabla H_{a_{j}}H_{a_{j}}
=
O
\begin{pmatrix}
\ln^{2} r & \text{for }\; n=6 \\
r^{12-2n} & \text{for }\; n\geq 7 \\
\end{pmatrix},
\end{equation*}
whence up to some $O\big(\frac{1}{\lambda_{j}^{4}}+\frac{1}{\lambda_{j}^{2(n-2)}}\big)$
\begin{equation*}
\begin{split}
\mathfrak{M}
= &
\frac{2}{2-n}
\int_{B_{c}(a_{j})}
K(\exp_{g_{a_{j}}})
(
\frac
{\lambda_{j}}
{
1+\lambda_{j}^{2}r^{2}
}
)^{n-\theta} 
\frac
{\lambda_{j}^{2}r^{n}\frac{\nabla_{a_{j}}}{\lambda_{j}}H_{a_{j}}}
{
1+\lambda_{j}^{2}r^{2}
}  d \mu_{g_{0}}
\\
= &
\frac{2}{2-n}
\int_{B_{c}(a_{j})}
(K_{j}+\nabla K_{j}x+O(r^{2})) 
(
\frac
{\lambda_{j}}
{
1+\lambda_{j}^{2}r^{2}
}
)^{n-\theta} 
\frac
{\lambda_{j}^{2}r^{n}}
{
1+\lambda_{j}^{2}r^{2}
}
\begin{pmatrix}
\frac{\nabla_{a_{j}}H_{j}}{\lambda_{j}} + \frac{\nabla_{a_{j}}\nabla H_{j}x}{\lambda_{j}} + O(\frac{r^{2}}{\lambda_{j}}) \\
\frac{\nabla_{a_{j}}H_{j}}{\lambda_{j}} + \frac{\nabla_{a_{j}}\nabla H_{j}x}{\lambda_{j}} + O(\frac{r^{2}\ln r}{\lambda_{j}}) \\
\frac{\nabla_{a_{j}}H_{j}}{\lambda_{j}} +  O(\frac{r}{\lambda_{j}}) \\
-\frac{\nabla_{a_{j}}W_{j}}{\lambda_{j}}\ln r + O(\frac{1}{\lambda_{j}}) \\
O(\frac{r^{6-n}}{\lambda_{j}}),
\end{pmatrix} dx,
\end{split}
\end{equation*}
and we obtain
\begin{equation*}
\begin{split}
\mathfrak{M}
= &
\frac{2K_{j}}{2-n}
\int_{B_{c}(a_{j})}
(
\frac
{\lambda_{j}}
{
1+\lambda_{j}^{2}r^{2}
}
)^{n-\theta} 
\frac
{\lambda_{j}^{2}r^{n}}
{
1+\lambda_{j}^{2}r^{2}
}
\begin{pmatrix}
\frac{\nabla_{a_{j}}H_{j}}{\lambda_{j}} + O(\frac{r^{2}}{\lambda_{j}}) \\
\frac{\nabla_{a_{j}}H_{j}}{\lambda_{j}}  + O(\frac{r^{2}\ln r}{\lambda_{j}}) \\
\frac{\nabla_{a_{j}}H_{j}}{\lambda_{j}} +  O(\frac{r}{\lambda_{j}}) \\
-\frac{\nabla_{a_{j}}W_{j}}{\lambda_{j}}\ln r + O(\frac{1}{\lambda_{j}}) \\
O(\frac{r^{6-n}}{\lambda_{j}})
\end{pmatrix} dx
+
O\big(\frac{1}{\lambda_{j}^{4}}\big) \\
= &
\check{d}_{1}
K_{j}
\frac{\vartheta_{j}}{\lambda_{j}^{\theta}}
+
O\big(\tau^{2}+\frac{1}{\lambda_{j}^{4}}\big),
\; \quad \check{d}_{1}=\frac{2}{2-n}\underset{\R^{n}}{\int}\frac{r^{n}}{(1+r^{2})^{n+1}}dx
, \quad \vartheta_{j}
=
\begin{pmatrix}
0 \\
\frac{\nabla_{a_{j}}H_{j}}{\lambda_{j}^{3}}  \\
0\\0\\0
\end{pmatrix}
\end{split}
\end{equation*}
up to some $O(\frac{1}{\lambda_{j}^{4}}+\frac{1}{\lambda_{j}^{2(n-2)}})$. Collecting terms we arrive at
\begin{equation}\label{non_linear_interaction_testing_3}
\int K\varphi_{j}^{p}\frac{\nabla_{a_{j}}}{\lambda_{j}}\varphi_{j}d\mu_{g_{0}}
= 
\frac{\Gamma-(n-\theta)\mathfrak{M}}{p+1} 
= 
\frac{n-2}{2n}
\big(
\check{c}_{3}\frac{\nabla K_{j}}{\lambda_{j}^{1+\theta}}
+
\check{c}_{4}\frac{\nabla \Delta K_{j}}{\lambda_{j}^{3+\theta}}
+
n\check{d}_{1}
K_{j}
\frac{\vartheta_{j}}{\lambda_{j}^{\theta}}
\big)
\end{equation}
up to some
$
O\big(\tau^{2}+\frac{\vert \nabla K_{j}\vert^{2}}{\lambda_{j}^{2}}+\frac{1}{\lambda_{j}^{4}}+\frac{1}{\lambda_{j}^{2(n-2)}}\big)
$ 
and conclude
\begin{equation*}
\begin{split}
A
= &
\alpha_{j}\int \varphi_{j} L_{g_{0}}\frac{\nabla_{a_{j}}}{\lambda_{j}}\varphi_{j}d\mu_{g_{0}} 
-
\frac{4n(n-1)(n+2)b_{3}}{n-2}
\sum_{j\neq i}
\alpha_{i}
\frac{\nabla_{a_{j}}}{\lambda_{j}}\varepsilon_{i,j} \\
& -
\frac{2(n-1)(n-2)\alpha^{2}}{\alpha_{K,\tau}^{\frac{2n}{n-2}}}
K_{j}\alpha_{j}^{p}
\big(
\check{c}_{3}\frac{\nabla K_{j}}{K_{j}\lambda_{j}^{1+\theta}}
+
\check{c}_{4}\frac{\nabla \Delta K_{j}}{\lambda_{j}^{3+\theta}}
+
n\check{d}_{1}\frac{\vartheta_{j}}{\lambda_{j}^{\theta}}
\big)
\end{split}
\end{equation*}
up to some 
$
O
\big(
\tau^{2}
+
\sum_{r\neq s}
\frac{\vert \nabla K_{r}\vert^{2}}{\lambda_{r}^{2}}
+
\frac{1}{\lambda_{r}^{4}}
+
\frac{1}{\lambda_{r}^{2(n-2)}}
+
\varepsilon_{r,s}^{\frac{n+2}{n}}
+
\vert \partial J_{\tau}(u)\vert^{2}
\big).
$ 
  Applying Lemma \ref{lem_alpha_derivatives_at_infinity} we find 
\begin{equation}\label{Testing_3_up_to_L_g_0_interaction} 
\begin{split}
A
= &
\alpha_{j}\int \varphi_{j} L_{g_{0}}\frac{\nabla_{a_{j}}}{\lambda_{j}}\varphi_{j}d\mu_{g_{0}} 
-
4n(n-1)b_{3}
\sum_{j\neq i}
\alpha_{i}
\frac{\nabla_{a_{j}}}{\lambda_{j}}\varepsilon_{i,j} \\
& -
2(n-1)(n-2)\alpha_{j}
\big(
\check{c}_{3}\frac{\nabla K_{j}}{K_{j}\lambda_{j}}
+
\check{c}_{4}\frac{\nabla \Delta K_{j}}{K_j \lambda_{j}^{3}}
+
n\check{d}_{1}
\vartheta_{j}
\big)
\end{split}
\end{equation}
up to the same error. We are left with estimating 
\begin{equation*}
\begin{split}
\int \varphi_{j} L_{g_{0}}\frac{\nabla_{a_{j}}}{\lambda_{j}}\varphi_{j}d\mu_{g_{0}}
= &
\int_{B_{c}(a_{j})} \varphi_{j} L_{g_{0}}\frac{\nabla_{a_{j}}}{\lambda_{j}}\varphi_{j}d\mu_{g_{0}}
+
O\big(\frac{1}{\lambda_{j}^{n-\theta}}\big).
\end{split}
\end{equation*}
Then from Lemma \ref{lem_emergence_of_the_regular_part} we see that in case $n=4,5$ 
\begin{equation*}
\begin{split}
\int \frac{\varphi_{j} L_{g_{0}} \frac{\nabla_{a_{j}}}{\lambda_{j}}\varphi_{j}}{4n(n-1)} & d\mu_{g_{0}}
= 
\int_{B_{c}(a_{j})} \varphi_{j}^{\frac{n+2}{n-2}}\frac{\nabla_{a_{j}}}{\lambda_{j}}\varphi_{j}d\mu_{g_{0}}  -
\frac{c_{n}}{2}\int_{B_{c}(a_{j})} 
r_{a_{j}}^{n-2}(H_{j}+n\nabla H_{j}x)\varphi_{j}^{\frac{n+2}{n-2}}\frac{\nabla_{a_{j}}}{\lambda_{j}}\varphi_{j}d\mu_{g_{0}} \\
\end{split}
\end{equation*}
up to some $O
(
\frac{1}{\lambda_{j}^{4}}+\frac{1}{\lambda_{j}^{2(n-2)}}
)$,
and thus due to \eqref{non_linear_interaction_testing_3}
\begin{equation*}
\begin{split}
\int \frac{ \varphi_{j} }{4n(n-1)} & L_{g_{0}} \frac{\nabla_{a_{j}}}{\lambda_{j}}\varphi_{j}d\mu_{g_{0}}
= 
\frac{n-2}{2}
\check{d}_{1}
\vartheta_{j}
  -
\frac{c_{n}}{2}\int_{B_{c}(a_{j})} 
r_{a_{j}}^{n-2}(H_{j}+n\nabla H_{j}x_{a_{j}})\varphi_{j}^{\frac{n+2}{n-2}}\frac{\nabla_{a_{j}}}{\lambda_{j}}\varphi_{j}d\mu_{g_{0}} 
\end{split}
\end{equation*}
up to some 
$
O\big(\tau^{2}+\frac{\vert \nabla K_{j}\vert^{2}}{\lambda_{j}^{2}}+\frac{1}{\lambda_{j}^{4}}+\frac{1}{\lambda_{j}^{2(n-2)}}\big)
$.
Finally we observe that 
\begin{equation*}
\frac{\nabla_{a_{j}}}{\lambda_{j}}\varphi_{j}
= 
\frac{2-n}{2}u_{a_{j}}\big(\frac{\lambda_{j}}{1+\lambda_{j}r_{a_{j}}^{2}(1+r_{a_{j}}^{n-2}H_{a_{j}})^{\frac{2}{2-n}}}\big)^{\frac{n-2}{2}} 
\frac{\lambda_{j}\nabla_{a_{j}}(r_{a_{j}}^{2}(1+r_{a_{j}}^{n-2}H_{a_{j}})^{\frac{2}{2-n}})}{1+\lambda_{j}^{2}r_{a_{j}}^{2}(1+r_{a_{j}}^{n-2}H_{a_{j}})^{\frac{2}{2-n}}}
+
O\big(\frac{r_{a_{j}}}{\lambda_{j}}\varphi_{j}\big),
\end{equation*}
and using the smoothness of conformal normal coordinates with respect to $a_j$ we find 
\begin{equation*}
\begin{split}
\frac{\nabla_{a_{j}}}{\lambda_{j}}\varphi_{j}
= &
\frac{2-n}{2}u_{a_{j}}\big(\frac{\lambda_{j}}{1+\lambda_{j}r_{a_{j}}^{2}(1+r_{a_{j}}^{n-2}H_{a_{j}})^{\frac{2}{2-n}}}\big)^{\frac{n-2}{2}} 
\frac{2\lambda_{j}x_{a_{j}}}{1+\lambda_{j}^{2}r_{a_{j}}^{2}(1+r_{a_{j}}^{n-2}H_{a_{j}})^{\frac{2}{2-n}}} \\
& +
O\big(
\big( \frac{\lambda_{j}r_{a_{j}}^{4}}{1+\lambda_{j}^{2}r_{a_{j}}^{2}}
+
 \frac{\lambda_{j}r_{a_{j}}^{n-1}}{1+\lambda_{j}^{2}r_{a_{j}}^{2}}
+
\frac{r_{a_{j}}}{\lambda_{j}}
\big)
\varphi_{j}
\big).
\end{split}
\end{equation*}
This gives
\begin{equation*}
\begin{split}
\int \frac{\varphi_{j} }{4n(n-1)} L_{g_{0}} \frac{\nabla_{a_{j}}}{\lambda_{j}}\varphi_{j}d\mu_{g_{0}}
= 
\frac{n-2}{2}
\check{d}_{1}
\vartheta_{j}
  +
\frac{(n-2)c_{n}}{2}\underset{B_{c}(a_{j})}{\int}  
\frac{\lambda_{j}x_{a_{j}}r_{a_{j}}^{n-2}(H_{j}+n\nabla H_{j}x)\varphi_{j}^{\frac{2n}{n-2}} }{1+\lambda_{j}^{2}r_{a_{j}}^{2}(1+r_{a_{j}}^{n-2}H_{a_{j}})^{\frac{2}{2-n}}}
d\mu_{g_{0}} 
\end{split}
\end{equation*}
up to some 
$
O\big(\tau^{2}+\frac{\vert \nabla K_{j}\vert^{2}}{\lambda_{j}^{2}}+\frac{1}{\lambda_{j}^{4}}+\frac{1}{\lambda_{j}^{2(n-2)}}\big).
$
Passing to conformal normal coordinates around $a_{j}$, we find
\begin{equation*}
\begin{split}
\int_{B_{c}(a_{j})}  &
\frac{\lambda_{j}x_{a_{j}}r_{a_{j}}^{n-2}(H_{j}+n\nabla H_{j}x)\varphi_{j}^{\frac{2n}{n-2}} }{1+\lambda_{j}^{2}r_{a_{j}}^{2}(1+r_{a_{j}}^{n-2}H_{a_{j}})^{\frac{2}{2-n}}} 
d\mu_{g_{0}} 
= 
\int_{B_{c}(0)}  
\frac{\lambda_{j}^{n+1}xr^{n-2}H_{j} }{(1+\lambda_{j}^{2}r^{2})^{n+1}} dx
=
0
\end{split}
\end{equation*}
up to some $O(\frac{1}{\lambda_{j}^{4}}+\frac{1}{\lambda_{j}^{2(n-2)}}))$.
We therefore conclude 
\begin{equation*}
\begin{split}
\int  \varphi_{j}  L_{g_{0}}& \frac{\nabla_{a_{j}}}{\lambda_{j}}\varphi_{j}d\mu_{g_{0}}
= 
2n(n-1)(n-2)
\check{d}_{1}
\vartheta_{j}
\end{split}
\end{equation*}
up to some 
$
O\big(\tau^{2}+\frac{\vert \nabla K_{j}\vert^{2}}{\lambda_{j}^{2}}+\frac{1}{\lambda_{j}^{4}}+\frac{1}{\lambda_{j}^{2(n-2)}}\big).
$
Plugging into \eqref{Testing_3_up_to_L_g_0_interaction} we arrive at
\begin{equation}\label{a_derivative_final_expansion}
\begin{split}
A
= &
-
2(n-1)(n-2)\alpha_{j} \left( 
\check{c}_{3}\frac{\nabla K_{j}}{K_{j}\lambda_{j}}
+
\check{c}_{4}\frac{\nabla \Delta K_{j}}{K_j \lambda_{j}^{3}} \right) 
-
4n(n-1)b_{3}
\sum_{j\neq i}
\alpha_{i}
\frac{\nabla_{a_{j}}}{\lambda_{j}}\varepsilon_{i,j}, 
\end{split}
\end{equation}
up to some 
$$
O
\big(
\tau^{2}
+
\sum_{r\neq s}
\frac{\vert \nabla K_{r}\vert^{2}}{\lambda_{r}^{2}}
+
\frac{1}{\lambda_{r}^{4}}
+
\frac{1}{\lambda_{r}^{2(n-2)}}
+
\varepsilon_{r,s}^{\frac{n+2}{n}}
+ 
\vert \partial J_{\tau}(u)\vert^{2}
\big).
$$
Recalling \eqref{a_expansion_A_term} the claim follows by setting or replacing 
\begin{equation}\label{def_a_constants}
(\check c_{3},\check c_{4},\check b_{3})
\rightsquigarrow
4(n-1)(n-2)(\check c_{3},\check c_{4},\frac{2n}{n-2}b_{3}),
\end{equation} 
cf. \ref{check_c_3_c_4} and Lemma \ref{lem_interactions}.
\end{proof}

\subsection{List of constants}
We give here a list of constants, referring to where they can be found. 
\begin{equation*}
\begin{array}{c||c|c|c|c|c|c}
      &  & \bar{\quad}  & \hat{\quad} & \grave{\quad} & \tilde{\quad}&  \check{\quad} \\ \hline
c_{0} &  & \eqref{barc0_barc1} & \eqref{hat_constants}& \eqref{constants_alpha_derivative} &  & \\ \hline
c_{1}  & \text{Lemma}\; \ref{lem_interactions} & \eqref{barc0_barc1}  & \eqref{hat_constants}&  & \eqref{def_const_lambda_derivatives} & \\ \hline 
c_{2} & \text{Lemma}\; \ref{lem_interactions} &  \eqref{single_bubble_K_integral_expansion_exact} & \eqref{hat_constants} & \eqref{constants_alpha_derivative} & \eqref{def_const_lambda_derivatives} & \\ \hline 
c_{3} & \text{Lemma}\; \ref{lem_interactions} & &  &  &  \eqref{def_tilde_c3}& \eqref{def_a_constants}\\ \hline 
c_{4} &  & &  & & \eqref{def_tilde_c3_tildec_4} &\eqref{def_a_constants}  \\ \hline 
d_{1} &  & \eqref{mass_integral_expansion} & \eqref{hat_constants} & \eqref{constants_alpha_derivative} & \eqref{def_const_lambda_derivatives} & \\ \hline  
b_{1} & \text{Lemma}\; \ref{lem_interactions} &  \eqref{def_bar_b_1} & \eqref{hat_constants} & \eqref{constants_alpha_derivative}  & \eqref{L_g_0_bubble_interaction} & \\ \hline 
b_{2} & \text{Lemma}\; \ref{lem_interactions} & \eqref{def_bar_b2}  &  & & \eqref{def_const_lambda_derivatives} &  \\ \hline 
b_{3} & \text{Lemma}\; \ref{lem_interactions} &  &  &  &  &  \eqref{def_a_constants} 
\end{array}
\end{equation*}
For instance,  
$c_{2}$ is found in Lemma \ref{lem_interactions}, $\bar c_{2}$ in equation \eqref{single_bubble_K_integral_expansion_exact} and $\hat d_{1}$ in equation \eqref{hat_constants}.
For the empty fields the corresponding combination of accent and symbol is non-existent. As a caveat please note that we have within some proofs redefined constants for the sake for normalization, hence we point to the final definition, from which upwards mentioned constants can be easily recovered. Finally we recall that $c_{n}$ is the normalizing constants in the definition of the conformal laplacian
\begin{equation*}
 L_{g}=-c_{n}\Delta_{g}+R_{g},\; \qquad c_{n}=\frac{4(n-1)}{n-2}.
\end{equation*}

\end{document}